\documentclass[reqno]{amsart}
\usepackage[utf8]{inputenc}
\usepackage[margin = 1in]{geometry}
\usepackage{latexsym,amsmath,color,amsthm,epsfig,mathrsfs,mathdots}
\usepackage{mathtools}
\mathtoolsset{showonlyrefs}
\usepackage{graphicx}
\usepackage{amssymb}
\usepackage{mdframed}
\usepackage{fancyhdr}
\usepackage{tikz-cd}
\usepackage{verbatim}
\usepackage{nicefrac}
\usepackage{bbm}
\usepackage{tikz}
\usepackage{hyperref}
\hypersetup{
    colorlinks = true,
    citecolor=black,
    filecolor=black,
    linkcolor=black,
    urlcolor=blue
}
\usepackage{amsfonts, esint}
\usepackage{color}
\usepackage{dsfont}

\numberwithin{equation}{section}

\usepackage{lastpage}
\newtheorem{thm}{Theorem}[section]
\newtheorem{prop}[thm]{Proposition}
\newtheorem{cor}[thm]{Corollary}

\newtheorem{lem}[thm]{Lemma}
\newtheorem{defn}[thm]{Definition}

\newtheorem{preremark}[thm]{Remark}

\newenvironment{remark}{\begin{preremark}\rm}{\medskip \end{preremark}}
\numberwithin{equation}{section}

\newcommand{\R}{\mathbb R}
\newcommand{\eps}{\varepsilon}
\newcommand{\vphi}{\varphi}

\newcommand{\ds}{\displaystyle}

\newcommand{\mcL}{\mathcal{L}}

\DeclareMathOperator*{\osc}{osc}

\definecolor{sh}{RGB}{255,0,100}

\newcounter{case}
\renewcommand{\thecase}{\Alph{case}}
\newcounter{proofcase}[case]
\renewcommand{\theproofcase}{(\thecase\arabic{proofcase})}
\newif\ifusedcase

\newcommand{\proofcase}{%
  \ifusedcase\else\usedcasetrue\stepcounter{case}\fi
  \par
  \refstepcounter{proofcase}
  \everypar=\expandafter{\the\everypar{\setbox0=\lastbox}\everypar{}Case \theproofcase\ }%
}

\begin{document}

\begin{abstract}
    We study the obstacle problem associated with the Kolmogorov operator $\mcL \coloneqq \Delta_v - \partial_t - v\cdot\nabla_x$, which arises from the theory of optimal control in Asian-American options pricing models. This problem presents a significant departure from the elliptic and parabolic obstacle problems due to the highly degenerate hypoelliptic nature and the non-commutative Galilean group structure underlying the operator $\mcL$.

    Our first main contribution is to improve the known regularity of solutions, from $C^{0,1}_t \cap C^{0,2/3}_x \cap C^{1,1}_v$ to $C^{0,1}_{t,x} \cap C^{1,1}_v$. The previous result in the literature, which has been called optimal, corresponds to $C^{1,1}$ regularity with respect to the Kolmogorov distance. This is the expected regularity for solutions to obstacle problems. Our unexpected improvement of regularity in the $x$ variable is obtained using Bernstein's technique and an approach drawing on ideas from Evans-Krylov theory.

    We then use this improvement in regularity of the solution to prove the first known free boundary regularity result. We show that under a standard thickness condition, the free boundary is differentiable with respect to the kinetic distance and is a $C^{0,1/2}_{t,x} \cap C^{0,1}_v$ regular surface. This result constitutes the first step in the program of free boundary regularity. Critically, our arguments rely on a new monotonicity formula and a commutator estimate that are only made possible by the solution's enhanced regularity in $x$.

\end{abstract}

\title{Hypoelliptic Regularization in the Obstacle Problem for the Kolmogorov Operator}
\author{David Bowman}
\address[David Bowman]{Department of Mathematics, University of Chicago,  Chicago, Illinois 60637, USA}
\email{dbowman@uchicago.edu}
\maketitle

\section{Introduction}
\subsection{Background}
This paper is concerned with the obstacle problem associated with the Kolmogorov operator $\mcL \coloneqq \Delta_v - \partial_t - v\cdot\nabla_x \coloneqq \Delta_v - Y$. Let $\mathcal{U} \subset \R^{1+2n}$ be a bounded, open set, and denote by $\partial_p\mathcal{U}$ its parabolic boundary. Given a smooth function $\psi:\overline{\mathcal{U}} \to \R$ and a continuous function $g:\partial_p \mathcal{U} \to \R$ such that $g\ge \psi$, the obstacle problem associated to $\mcL$ is formulated as

\begin{equation}
    \label{eq: formulation1}
    \begin{cases} \min\{-\mcL f, f-\psi\} =0 & \text{ on } \mathcal{U},\\ f = g & \text{ on } \partial_p \mathcal{U}.
    \end{cases}
\end{equation}

The solution $f$ may be interpreted as the least supersolution $\mcL f \le 0$ that lies above the obstacle $\psi$, with boundary data $g$. The equation \eqref{eq: formulation1} arises from the study of optimal control for Asian-American option pricing models. In short, these models involve pay-offs $\psi$ which depend on the current price of the underlying asset as well as some type of historical average of the asset price. Such options are often traded in commodity, interest rate, and foreign exchange markets. To explain the financial motivation for \eqref{eq: formulation1}, it is simplest to consider the case $\mathcal{U}=[0,T] \times \R^n \times \R^n$. In this case, we consider a stochastic control model where the asset price $V_t$ and its historical average $X_t$ are evolving according to the dynamics
$$
\begin{cases} dX_t = V_t \ dt, \\
&  (X_{t_0},V_{t_0}) = (x_0,v_0) \in \R^n \times \R^n, t\in [t_0,T]\\
             dV_t = \sqrt{2} \ dB_t,
\end{cases}$$
where $B_t$ is an $n$-dimensional Brownian motion. At time $t_0$ and position $(x_0,v_0)$, the contract holder seeks to maximize the expected value of the payoff function $\psi = \psi(t,x,v)$ at some time prior (but no later) than the pre-agreed upon time $T>0$. That is, the value function for this problem is given by $$\tilde{f}(t_0,x_0,v_0) = \sup_{\tau \in [t_0,T]} \mathbb{E}[\psi(\tau, X_\tau, V_\tau)],$$ where the supremum is taken over all stopping times $\tau$ with values in $[t_0,T]$.  After a time reversal $t\mapsto T-t$ and the reflection $x \mapsto -x$, the modified value function $f(t,x,v) = \tilde{f}(T-t, -x, v)$ solves the obstacle problem \eqref{eq: formulation1}. Mathematically, one seeks to understand the regularity of the value function $f$ and the interface between the exercise set $\{f=\psi\}$, wherein it is optimal to take the option immediately, and the set $\{f>\psi\}$, in which it is optimal to wait before exercising the option. See \cite{ANCESCHI}, \cite{RebucciAnceschi}, \cite{MONTIPASCUCCI}, \cite{FRENTZPOLIDOROOPTIMAL1} and the references therein for a far more in depth exploration of the applications to financial mathematics.

The purpose of this paper is to explore the optimal regularity of $f$, as well as the regularity of the interface where $f$ touches the obstacle $\psi$. As is customary in the study of free boundary regularity for obstacle problems, we subtract the obstacle and set $u= f -\psi$ to find, at least formally, $\mcL u = -(\mcL\psi) \chi_{u>0}$. In order to have any hope of free boundary regularity, it is standard to impose the nondegeneracy condition $-\mcL\psi \ge c_0>0$ (see \cite{XFRXRO} for a discussion of this condition), and for simplicity, we set $-\mcL \psi \equiv 1$. We also restrict ourselves to the unit kinetic cylinder $Q_1 = \{(t,x,v) \in (-1,0] \times B_1 \times B_1\}.$ This brings us to the second formulation

\begin{equation}
    \label{eq: formulation2}
    \begin{cases} \mcL u = \chi_{u>0}  & \text{ in } Q_1\\
    u  \ge 0 & \text{ in } Q_1,\\
    u=g-\psi \coloneqq h \ge 0 & \text{ on } \partial_p Q_1\end{cases}
\end{equation}
To fix ideas, let us consider this equation for the rest of the discussion, with the understanding that the goal is to study the regularity of $u$ and the boundary of the set $\{u>0\}$, which we call the free boundary. Mathematically, this is a generalization of the parabolic problem studied by Caffarelli, Petrosyan, and Shahgholian in \cite{CPS}. However, there are several particularities of the structure of the operator $\mcL$ that make the analysis considerably different. First of all, the operator $\mcL$ is highly degenerate, with diffusion in only half of the variables. Obtaining enough regularity for $u$ in the spatial variable $x$ to be able to deduce anything about the regularity of the free boundary is one of the central difficulties. Standard arguments in the theory of obstacle problems do not seem to give enough smoothness in $x$, which has made new arguments imperative. We discuss this in more depth further on. Another critical problem is that the operator $\mcL$ fails to commute with $\partial_{v_i}$ and the material derivative $Y= \partial_t + v\cdot\nabla_x$. Indeed, $[\mcL, \partial_{v_i}] = \partial_{x_i}$, and $[\mcL, Y] = 2\nabla_v \cdot \nabla_x$. The fact that these operators do not commute precludes the direct transference of methods and results from \cite{CPS}. On the other hand, $\mcL$ does commute with the operators $\partial_t$ and $\partial_{x_i}$, but the former lacks the necessary translation invariance and the latter the necessary order of scaling for us to apply the standard methods. Finally, there is considerable difficulty posed by the oblique geometry of the cylinders naturally associated with the Galilean group structure of the problem. 

\subsection{Main Results}
The two main results of this paper are an improvement of the known regularity in the $x$-variable, and the first known regularity result for for the free boundary. Our first contribution is the perhaps surprising improvement on the known regularity in the spatial variable $x$ from $C_x^{0,2/3}$ to $C_x^{0,1}$, at least for smooth enough obstacles, with the following theorem.
\begin{thm}
\label{thm: optimalreg}
    Let $f$ solve \eqref{eq: formulation1} with continuous boundary data $g:\partial_pQ_1 \to \R$ such that $g \ge \psi$, and assume that $\psi \in C^4_{t,x,v}(Q_1).$ Then there exists $C(\|\psi\|_{C^4_{t,x,v}(Q_1)}, \|g\|_{L^\infty(\partial_pQ_1)}, n)$ such that the following estimate holds:
    \begin{align}
        \label{eq: Lipxestf}
        \|f\|_{L^\infty(Q_1)} + \|\nabla_x f\|_{L^\infty(Q_{1/2})} + \|\partial_t f\|_{L^\infty(Q_{1/2})} + \|D^2_v f\|_{L^\infty(Q_{1/2})} \le C.
    \end{align}
    
    For $u$ solving \eqref{eq: formulation2}, there exists $C=C(n)$ such that $u$ enjoys the estimate
    \begin{equation}
        \label{eq: Lipxestu}
       \|\nabla_x u\|_{L^\infty(Q_{1/2})}+ \|\partial_t u\|_{L^\infty(Q_{1/2})} +\|D^2_vu\|_{L^\infty(Q_{1/2})}\le C(1+\|u\|_{L^\infty(Q_1)}).
    \end{equation}
\end{thm}
We make no attempt to optimize the regularity of the obstacle $\psi$, since at least for a first attempt at studying the free boundary it is typical to take $\mcL \psi=-1$. To fix ideas as we discuss the result, let us work with the second formulation \eqref{eq: formulation2}. That $D^2_vu$ and $\partial_t u$ are locally bounded is not a new fact, and was obtained many years ago in the paper \cite{FRENTZPOLIDOROOPTIMAL1}, in which Frentz, Nyström, Pascucci and Polidoro apply a scaling-based blow-up argument to obtain a $C^{1,1}$ estimate with respect to the kinetic distance, which corresponds to $C^{0,1}_t\cap C^{0,2/3}_x \cap C^{1,1}_v$ local regularity. We discuss their result more in Section \ref{historical}, but we remark for now that the regularity they obtain is a consequence of $f$ separating from the obstacle quadratically with respect to the radius of the kinetic cylinders naturally associated with the Galilean group structure of the problem. This is the typical regularity obtained in obstacle-type problems, see for instance, \cite{OBSPROBREVISITED}, \cite{CPS}, \cite{SUBELLIPTIC} for the elliptic, parabolic, and sub-elliptic cases, respectively. It is therefore very surprising that in the degenerate hypoelliptic equation \eqref{eq: formulation2}, the solution $u$ enjoys an improvement of regularity in the $x-$variable that is considerably stronger than that which comes from the scaling. As we shall see, the $\nabla_x u$ estimate of \eqref{eq: Lipxestu} is what really enables the study of the free boundary. For instance, it allows us to establish a monotonicity formula and conclude that blow-ups for the problem \eqref{eq: formulation2} coincide with the blow-ups in the parabolic problem. Additionally, it provides the control needed to circumvent the difficulty posed by $\partial_{v_i}$ and $\mcL$ failing to commute. 

The proof of Theorem \ref{thm: optimalreg} is based on the penalization method and Bernstein's technique. However, Bernstein's method is considerably more difficult in the hypoelliptic setting than in the elliptic or parabolic settings. Indeed, considering a solution $u$ to $\mcL(u)=0$, there holds
\begin{align} \mcL(|\nabla_xu|^2) &= 2|\nabla_v\nabla_x u|^2,\\
\mcL(|\nabla_v u|^2) &= 2(|D^2_vu|^2 + \nabla_v u \cdot \nabla_xu).
\end{align}
One typically finds Lipschitz estimates in the diffusion variable the easiest to obtain, but in this setting $|\nabla_vu|^2$ is not a subsolution, while $|\nabla_xu|^2$ is. In this case, we find the invariance of $\mcL$ under translation in the $x$-variable to our benefit, and the \textit{lack of} invariance under translations in the $v-$variable to our detriment. This is a core difficulty that one must contend with in order to derive interior estimates. Continuing with the regularity of solutions to \eqref{eq: formulation2}, the first step is to obtain estimates on the lower-order terms $|\nabla_vu|$ and $|\nabla_xu|$. Doing so involves strengthening an estimate of Imbert and Mouhot in \cite{IMBERTMOUHOT}. From this follows an estimate for  $|\partial_t u|$, using also the observation that $|\partial_t u| \le 1 + |D^2_vu| + |\nabla_xu|$ on $Q_1$. Finally, we establish a lower bound on $\partial_{v_e v_e} u$ for any direction $e$, and then use the first order estimates and the ellipticity of $\mcL$ in the variable $v$ to obtain an upper bound on $\partial_{v_ev_e}u$. The idea of using ellipticity to compare  $(\partial_{v_ev_e}u)^-$ and $(\partial_{v_ev_e}u)^+$ hearkens back to the Evans-Kyrlov theorem, and here we are making great use of the fact that obstacle problems are \textit{concave}. We point out here also that the idea of using Bernstein's technique for obstacle problems also appeared in \cite{ROSOTONBERNSTEIN}, \cite{CAFFARELLITHINOBSTACLE}, \cite{KOIKEBERNSTEIN}, but their approaches differ from ours. 

Lastly, we briefly discuss the theory of hypoellipticity for the operator $\mcL$. Despite being highly degenerate parabolic, with diffusion only in half the variables, the operator $\mcL$ enjoys the regularizing property that if $\mcL u \in C^\infty$, then $u \in C^\infty$. The diffusion in $v$ smoothens the solution out in the velocity variable, and this smoothness is propagated to the space-time variables $(t,x)$ through the transport term $\partial_t + v\cdot \nabla_x$. This interaction between diffusion and transport is an instance of what is known as hypoelliptic regularization. This is realized through the commutator $[\nabla_v, \partial_t + v\cdot \nabla_x] = \nabla_x$, which implies that the operator $\mcL$ satisfies the H{\"o}rmander condition for hypoellipticity (see \cite{FRENTZPOLIDOROOPTIMAL1} and the references therein for a more detailed discussion). In the Hilbert space setting, the hypoelliptic regularization mechanism of $\mcL$ may be realized through velocity averaging as in \cite{Bouchut}, for example, where it is demonstrated that if $\mcL u \in L^2(\R^{1+2n})$, then $ (-\Delta_x)^{1/3} u \in L^2(\R^{1+2n})$. Or, in the H{\"o}lder setting, the Schauder theorem of \cite{IMBERTMOUHOT} gives that if $\mcL u \in C^{0,\alpha}_\ell$, which in particular corresponds to $\alpha/3$ H{\"o}lder regularity in $x$, then $u \in C^{2,\alpha}_\ell$, which corresponds to $(2+\alpha)/3$ H{\"o}lder regularity in $x$. In general, it is typically expected that $u$ gains regularity in the spatial variable $x$ on the order of $2/3$. Therefore, and for simplicity considering the alternative formulation \eqref{eq: formulation2} where $\mcL u = \chi_{u>0} \in L^\infty$, the initial result of \cite{FRENTZPOLIDOROOPTIMAL1} would really seem optimal. The $C^{1,1}_\ell$ regularity is in keeping with traditional results in obstacle problems, and the $2/3$ order of regularity in $x$ is in keeping with the typical hypoelliptic regularization effect of the operator $\mcL$. Our Theorem \ref{thm: optimalreg} is particularly surprising then, since the regularity in $x$ is considerably stronger than that which is expected from the kinetic scaling and that which is typical of the regularity theory for $\mcL$. 

 Our second main result is, to our knowledge, the first result on the regularity of free boundaries for \eqref{eq: formulation2}. In order to state our result, we quickly introduce the notation $\Omega_u \coloneqq \{u>0\}$ for the region where $u$ is above the zero-obstacle, $\Lambda_u \coloneqq \{u=0\}$ for the contact set, and $\Gamma_u \coloneqq \partial_p\Omega_u$ for the free boundary. As per usual, we study the regularity of $\Gamma_u$ near points where $\Lambda_u$ is thick enough. To be more precise, we measure the density of $\Lambda_u$ at a point $z_0=(t_0,x_0,v_0) \in \Gamma_u$ and at scale $r>0$ through the thickness function
\begin{equation}
\label{eq: densityfunction}
    \delta_r(u, z_0) \coloneqq \sup_{x\in B_{r^3}(x_0-r^2v_0)} \frac{\text{m.d.}(\Lambda_{u(t_0-r^2,x,\cdot)} \cap B_r(v_0))}{r},
    \end{equation}
where $\text{m.d.}(E)$ is the minimal diameter of a set $E \subset \R^{n}$, which is defined to be the smallest distance between two parallel planes containing $E$, and $\Lambda_{u(t,x,\cdot)} =\{v \ : \ u(t,x,v) =0\}$. When $z_0 =(0,0,0)$, we often just write $\delta_r(u)$. The definition of $\delta_r$ will be made more clear after the discussion of kinetic cylinders in Section \ref{backgroundonkinetic}. 

\begin{thm}
\label{thm: freeboundaryregularity}
Let $u$ solve \eqref{eq: formulation2}, and suppose that $z_0 \in \Gamma_u \cap Q_{1/2}$. Then there exists a universal modulus of continuity $\sigma:(0,\infty) \to (0,\infty)$ such that if $\delta_{r_0}(u,z_0) \ge \sigma(r_0)>0$ for a single value $r_0>0$, then 
$\Gamma_u \cap Q_{cr_0}(z_0)$ is a $C^{0,1/2}_{t,x} \cap C^{0,1}_v$ graph. 
\end{thm}
The same result is true if we replace $\delta_r$ by the weaker thickness function
$$\delta_r^*(u,z_0) \coloneqq \sup\left\{\frac{\text{m.d.}(\Lambda_{u(t,x,\cdot)}\cap B_r(v_0))}{r} \ : \ t \in [t_0-r^2,t_0-r^2/2],\ x \in B_{r^3(x_0+(t-t_0)v_0)}\right\}.$$

Additionally, we have the following useful results which are not directly implied by the previous theorem. The discussion of what it means for a graph to be differentiable with respect to the kinetic distance is postponed to section \ref{FBregularitysec}.
\begin{prop}
    \label{prop: graphdiff}
    Let $u$ solve \eqref{eq: formulation2}, and suppose that $z_0 \in \Gamma_u \cap Q_{1/2}$. Let $\sigma = \sigma(r)$ be the modulus of continuity for the thickness function $\delta_r$ as in the last theorem. If $\delta_{r_0}(u, z_0) \ge \sigma(r_0)>0$ for some $r_0>0$, then $\Gamma_u \cap Q_{cr_0}(z_0)$ is a surface which is everywhere differentiable with respect to the kinetic distance, and satisfies the so-called corkscrew condition: there exists $\kappa=\kappa(n,M)>0$ such that for any  $0<r<r_0$, there exist cylinders $Q_{\kappa r}(z_1) \subset \Lambda_u \cap Q_r(z_0)$ and $Q_{\kappa r}(z_2) \subset \Omega_u \cap Q_r(z_0).$ 
\end{prop}

 In proving these results, we follow the standard program of free boundary regularity. There are several critical junctures that seem to have only been passable in light of the enhanced regularity of $u$ in the $x-$variable. The first step made possible by the Lipschitz estimate is an almost-monotonicity formula in the vein of the Weiss monotonicity formula utilized in \cite{CPS}. In essence, near the free boundary point $(0,0,0)$, we are able to consider the rescalings $u_r(t,x,v)\coloneqq r^{-2}u(r^2t, r^3x, rv)$ as being almost solutions to the parabolic obstacle problem, in the sense that $(\Delta_v - \partial_t)u_r \approx \chi_{u_r>0}$. The fact that $\nabla_x u_r= r\nabla_x u = O(r)$ on compact sets, and that this estimate is Galilean invariant, is used to our advantage repeatedly. With this monotonicity formula we prove that the rescalings $u_r = u_r(t,x,v)$ converge in a strong sense to globally defined functions $u_0= u_0(t,v)$, called \textit{blow-ups} of $u$, that solve the parabolic obstacle problem and are independent of $x$. In fact, we show that blow-ups for the problem \eqref{eq: formulation2} coincide with blow-ups for the parabolic problem. The next step is to transfer information about the free boundary from $u_0$ to $u_r$. Using the monotonicity formula, we prove that the material derivative $Yu = \partial_tu + v\cdot \nabla_xu$ vanishes continuously at free boundary points near which the contact set $\Lambda_u$ is thick enough at a single scale. This should be seen as analogous to the fact that the time derivative vanishes at regular free boundary points in the parabolic problem, as in \cite{CPS}. We also manage to prove that the blow-ups at such points are unique, and are half-space solutions of the form $\frac12(v\cdot e)_+^2$ for some direction $e \in \mathbb{S}^{n-1}$. The next critical step is to establish that at regular points $u$ exhibits directional monotonicity for a cone of velocities $v$. This involves handling the commutator $[\partial_{v_e},\mcL]u_r = \partial_{x_e}u_r$, which only seems to be possible with the estimate on $\nabla_x u_r$. Finally, we use the vanishing of $Yu$ near regular points and the aforementioned monotonicity to establish a pointwise lower bound $u(0,0,v) \gtrsim v_n^2$ on a cone of velocities. With the regularity estimates of Theorem \ref{thm: optimalreg} we are able to trap the free boundary between $C^{0,1/2}_{t,x}\cap C^{0,1}_v$ H{\"o}lder cones on both sides.

It should be noted that this result implies, and in fact is stronger than, the fact that $\Gamma_u$ may be written locally as the graph of a function $f$ which is \textit{Lipschitz} with respect to the kinetic (also called Kolmogorov) distance. The latter corresponds to $C^{0,1/2}_t \cap C^{0,1/3}_x \cap C^{0,1}_v$ regularity with respect to the Euclidean distance. Our extra regularity for the free boundary comes from the extra regularity on $u$. Finally, we discuss the final claim of Proposition \ref{prop: graphdiff}, which says that if the contact set $\Lambda_u$ satisfies a thickness condition near $z_0\in \Gamma_u$ at a single scale, then at \textit{all} small scales $0<r<r_0$, the sets $\Omega_u \cap Q_r(z_0)$ and $\Lambda_u \cap Q_r(z_0)$ contain cylinders of proportional radius. This fact is \textit{not} implied by the H{\"o}lder regularity. Indeed, take the domain $\mathcal{U} \coloneqq \{(t,x,v) \ : \ x>0\}$. Only the set $\{(t,0,0)\}$ satisfies the two-sided cylinder condition, despite the boundary of $\mathcal{U}$ being as flat as can be. Here we see that the geometry of the class of cylinders induced by the action of the Galilean group poses a considerable challenge, in contrast with the geometry in the elliptic and parabolic problems. We discuss the geometry of the kinetic cylinders, the Galilean transformation, and the kinetic distance and scaling at length in the next subsection. At any rate, this extra condition could potentially be useful in the theory of Boundary Harnack principles for the operator $\mcL$.

We would be remiss not to mention that the result of Theorem \ref{thm: freeboundaryregularity} constitutes only the first step in the general program for free boundary regularity. As it stands, there does not exist a Boundary Harnack principle that is suitable for our problem. And even if a result analogous to those for elliptic and parabolic operators were to be established for $\mcL$, there would still be a significant bridge to cross before reaching higher regularity for the free boundary $\Gamma_u$. First of all, and as stated previously, it seems quite difficult to establish that $\partial_t u$ and $\nabla_x u$ vanish continuously near regular points, since the former does not satisfy Galilean invariance and the latter does not have the right order of scaling. The quantities that we \textit{can} show vanish on the regular part of $\Gamma_u$ are $Yu = \partial_t u + v\cdot\nabla_x u$ and $\partial_{v_e}u$. While $\mcL(\partial_{v_e}u) = \partial_{x_e}u$ can be made small by scaling, the quantity $\mcL(Yu) = 2 \nabla_v \cdot \nabla_x u$ is not expected to be controlled. Broadly speaking, we would like to apply Lemma \ref{lem: improvementmin}, which is already a generalization of the canonical lemma in the literature, for a cone of directions $\nu_{t,x}$ in the $(t,x)$ variables to deduce that $\partial_{\nu_{t,x}} u + \partial_{v_e}u \ge 0$ near regular points. For the reasons stated, there is a major difficulty in taking $\nu_{t,x} = (1,e)$, or $\nu_{t,x}=(1,v)$, and so it seems to be quite challenging to even show that $u$ is monotone on $(t,x,v)$-cones centered at regular points. In sum, the kinetic scaling and the non-commutativity of the Galilean group present significant obstacles toward the obtaining of higher regularity for $\Gamma_u$.
\subsection{Background on Kinetic Equations}
\label{backgroundonkinetic}
The operator $\mcL$ is often referred to as the \textit{Kolmogorov operator} and is the prototypical example of a \textit{kinetic Fokker-Planck operator}. As discussed in Section $1.2$, the operator $\mcL$ is prominent in the theory of hypoellipticity, with it and its fractional-diffusion variants being commonly-used testing grounds for the study of the spatially inhomogeneous Landau and Boltzmann equations coming from kinetic theory. While the obstacle problem \eqref{eq: formulation1} and its variant \eqref{eq: formulation2} are not \textit{kinetic} in the sense that they model systems of particles coming from physics, their analysis is still well-suited for the recently developed framework of spatially inhomogeneous kinetic equations (see, for instance, \cite{LUISKINETIC}, \cite{LUISSCHAUDER}, \cite{IMBERTMOUHOT} and the references therein). For this reason, we still refer to $x$ as the spatial variable and $v$ as the velocity variable, even though this is not really the case with the motivations we have in mind. It seems pertinent to provide a brief overview of the inherent invariances and geometries associated with the operator $\mcL$. 

The natural translation that preserves the equation $\mcL u = f$ is the inertial change of variables
$$z_0 \circ z\coloneqq (t_0,x_0,v_0)\circ (t,x,v) \coloneqq (t_0+t, x_0+x+tv_0, v_0+v).$$
The operation $\circ$ is called the Galilean transformation, and $\R^{1+2n}$ and gives rise to the non-commutative \textit{Galilean Group.} It is easy to see that it is really a group, as elements $z_0=(t_0,x_0,v_0)$ of the Galilean group $(\R^{2n+1}, \circ)$ come with inverses $z_0^{-1}=(-t_0,-x_0 + t_0v_0, -v_0).$ The operator $\mcL$ is Galilean invariant in the sense that if one sets $u^{z_0}(\cdot) \coloneqq u(z_0 \circ \cdot)$, there holds $\mcL u^{z_0} = f^{z_0}$. It is important to point out that translations in $v$ \textit{do not} preserve the equation, because the transport derivative $Y = \partial_t + v\cdot \nabla_x$ has coefficients depending on $v$. As for scaling, the transformation 
$$S_r(t,x,v) \coloneqq (r^2t, r^3x, rv)$$ leaves the class of equations $\mcL u = f$ invariant, in the sense that if we set $u_r(\cdot) \coloneqq u(S_r(\cdot))$, then  $\mcL u_r = r^2 f_r$. With respect to these two invariances we define the kinetic cylinders of radius $r>0$ centered at $z_0=(t_0,x_0,v_0)$ by
$$Q_r(z_0) \coloneqq \{(t,x,v) \ : \ t_0-r^2<t\le t_0, |x-x_0-(t-t_0)v_0|<r^3, |v-v_0|<r\} = z_0 \circ Q_r,$$
where here $Q_r$ denotes the cylinder of radius $r$ centered at $(0,0,0)$. Given a solution $u$ on $Q_r(z_0)$, we may always translate by $z_0^{-1}$ and rescale appropriately to the situation of a solution on $Q_1$. The geometry of these cylinders is considerably more complicated than in the Euclidean or parabolic setting: they are oblique in the spatial direction, as the centers of the $x$-balls travel in time with velocity $v_0$. 

The differential operators $\partial_t + v\cdot \nabla_x =Y, \nabla_x,$ and  $\nabla_v$ arise from the right-action of the group:
\begin{align}
   Yu(t,x,v) &= \lim_{h \to 0} \frac{u((t,x,v)\circ(h,0,0)) - u(t,x,v)}{h},\\
    \partial_{x_i} u(t,x,v) &= \lim_{h \to 0} \frac{u((t,x,v)\circ(0,he_i,0)) - u(t,x,v)}{h},\\
    \partial_{v_i} u(t,x,v) &= \lim_{h \to 0} \frac{u((t,x,v)\circ(0,0, he_i)) - u(t,x,v)}{h}.
\end{align}
These derivatives are invariant under the left-action of the group, in the sense that, with $u^{z_0}$ the left-translation defined above, one has $Yu^{z_0}(z) = Yu(z_0\circ z)$, and likewise for $\partial_{x_i}$ and $\partial_{v_i}$. For this reason, we call them left-invariant differential operators. On the other hand, these derivatives do not necessarily commute with the operator $\mcL$. 

The derivatives $\partial_t, \nabla_x,$ and $t\nabla_x + \nabla_v$, which arise from the left-action of the group in the sense that
\begin{align}
    \partial_t u(t,x,v) &= \lim_{h \to 0} \frac{u((h,0,0) \circ (t,x,v)) - u(t,x,v)}{h},\\
    \partial_{x_i} u(t,x,v) &=  \lim_{h \to 0} \frac{u((0,he_i,0) \circ (t,x,v)) - u(t,x,v)}{h},\\
(t\partial_{x_i} +\partial_{v_i}) u(t,x,v) &= \lim_{h \to 0} \frac{u((0,0,he_i)\circ(t,x,v)) - u(t,x,v)}{h},
\end{align}
\textit{do} commute with $\mcL$, and are invariant by the \textit{right}-action of the group, but not necessarily the \textit{left}-action. Typically, when performing a blow-up analysis in free boundary problems, one re-centers and zooms in around a free boundary point $z_0$. In order to preserve the equation, one must translate with the left action. But this means that information about $\partial_t$, for instance, may be lost. In addition, the failure of $Y$ and $\partial_{v_i}$ to commute with $\mcL$ means that even if $\mcL u =0$, it may not be true that $\mcL(\partial_{v_i} u)=0$ or $\mcL(Yu)=0$. Both of these particularities of the hypoelliptic structure of the operator $\mcL$ must be contended with when studying the obstacle problem \eqref{eq: formulation2}.

Next, we briefly discuss the connections of our problem to the general theory of influx boundary conditions for kinetic equations. For a generic smooth bounded domain $\mathcal{U} \subset \R\times \R^n \times \R^n$ with exterior normal $n = (n_t, n_x, n_v)$, one may partition the boundary $\partial \mathcal{U}$ into three pieces:
\begin{align}
    &\partial_+\mathcal{U} \coloneqq \{(t,x,v) \in \partial \mathcal{U} \ : \ (1,v,0) \cdot n >0\}  \ \ \ \text{ (outgoing) }\\
    &\partial_-\mathcal{U} \coloneqq \{(t,x,v) \in \partial \mathcal{U} \ : \ (1,v,0) \cdot n <0\}  \ \ \ \text{ (incoming) }\\
    &\partial_0\mathcal{U} \coloneqq \{(t,x,v) \in \partial \mathcal{U} \ : \ (1,v,0) \cdot n =0\}  \ \ \ \text{ (tangential) }
\end{align}
 Broadly speaking, particles flow into $\mathcal{U}$ at the incoming boundary, graze the domain at the tangential boundary, and exit through the outgoing boundary. In analogy to the parabolic setting, $\partial_-\mathcal{U}$ is akin to the initial time, $\partial_0\mathcal{U}$ the lateral boundary, and $\partial_+\mathcal{U}$ the terminal time. In the context of physical kinetic equations, one typically considers domains only in $x$, i.e., $\mathcal{U} = (0,\infty) \times U_x \times \R^d$. In this case, away from the initial time, the exterior normal $n= (0, n_x, 0)$ has no $t$ or $v$ components. It was observed in \cite{LUISKINETIC} that in this case there holds
$$\lim_{r \to 0} \frac{|Q_r(z_0) \cap \mathcal{U}|}{|Q_r|} = 0 \text{ for } z_0 \in \partial_+ \mathcal{U}$$
and likewise
$$\lim_{r \to 0} \frac{|Q_r(z_0) \cap \mathcal{U}^c|}{|Q_r|} = 0 \text{ for } z_0 \in \partial_- \mathcal{U}.$$

In other words, the geometry of the slanted cylinders implies that cylinders centered along the outgoing boundary become increasingly contained in the complement of $\mathcal{U}$ as they get smaller, with the opposite being true for incoming points. Therefore, in the case of $\mathcal{U}$ restricting only the spatial variable $x$, it is only those cylinders centered at tangential points that see some fixed proportion of $\mathcal{U}$ and $\mathcal{U}^c$. We highlight this fact because this intuition does not transfer over to general bounded domains in all variables, and because at first thought, this might seem a bit concerning for our free boundary regularity result. As is customary in free boundaries for obstacle problems, the ``regular" part of the free boundary $\Gamma_u$ for which we will prove a regularity result is the set of $z_0 \in \Gamma_u$ such that $\Omega_u = \{u>0\}$ and $\Lambda_u = \{u=0\}$ comprise some fixed proportion of $Q_r(z_0)$ for every $r>0$ small, see Proposition \ref{prop: graphdiff}. This might lead one to worry that our result is only for the tangential part of the boundary, which could be very small. Indeed, consider $\mathcal{U} = \{(t,x,v)  \in \R^3 \ : \ x>0\}$. Then $\partial_0\mathcal{U} = \{(t,0, 0) \ : \ t\in \R\}$, a very small part of the boundary. A regularity theory that holds only for the tangential boundary would be very restrictive. But, as already said, this density characterization does not hold for domains $\mathcal{U}$ bounded in all variables $(t,x,v)$, and we point this out because such domains are not typically studied in the context of \textit{kinetic} equations. Indeed, one may follow the methods of \cite{LUISKINETIC} to find that, for any $C^1$ bounded domain $\mathcal{U} \subset \R^{1+2n}$, and any point $z_0 \in \partial \mathcal{U}$ with outer normal $n_{z_0} = (n_{t_0}, n_{x_0}, n_{v_0})$ such that $n_{v_0} \neq 0$, there holds
$$\liminf_{r \to 0} \frac{|Q_r(z_0) \cap \mathcal{U}|}{|Q_r|} >0,$$
$$\liminf_{r \to 0} \frac{|Q_r(z_0) \cap \mathcal{U}^c|}{|Q_r|} >0.$$
Therefore it is reasonable to expect that generically, the free boundary consists mostly of regular points. 

Finally, we touch briefly on the kinetic H{\"o}lder spaces developed in \cite{LUISSCHAUDER} and \cite{IMBERTMOUHOT}. It is simplest to define the kinetic H{\"o}lder spaces through Taylor approximations. In our context,  the quantity $p_1(z;z_0) \coloneqq u(z_0) + (v-v_0)\cdot \nabla_vu(z_0)$ is the first order Taylor polynomial of $u$ at $z_0$ with respect to the kinetic scaling,  and likewise $\displaystyle p_2(z;z_0) \coloneqq u(z_0) + (t-t_0)Yu(z_0) +(v-v_0)\cdot \nabla_v u(z_0) + \frac12 (v-v_0)\cdot D^2_vu(z_0)(v-v_0)$ is the second order approximation. Then we may formulate the following definition. 
\begin{defn}
    Let $0<\alpha \le 1$. A bounded function $u: Q_1 \to \R$ is said to belong to $C^{0,\alpha}_{\ell}(Q_1)$ if

    $$\|u\|_{C^{0,\alpha}_\ell(Q_1)} \coloneqq \|u\|_{L^\infty(Q_1)} + \sup\left\{ \frac{\|u-u(z_0)\|_{L^\infty(Q_r(z_0))}}{r^\alpha} \ : \ Q_r(z_0) \subset Q_1\right\} <\infty,$$
    
    and likewise we say $u\in C^{1,\alpha}_{\ell}(Q_1)$ if 
    $$\|u\|_{C^{1,\alpha}_\ell(Q_1)} \coloneqq \|u\|_{L^\infty(Q_1)} + \sup\left\{\frac{\|u-p_1(z;z_0)\|_{L^\infty(Q_r(z_0))}}{r^{1+\alpha}} \ : \ Q_r(z_0)\subset Q_1\right\} <\infty$$

    and finally $u \in C^{2,\alpha}_{\ell}(Q_1)$ if
    
      $$\|u\|_{C^{2,\alpha}_\ell(Q_1)} \coloneqq \|u\|_{L^\infty(Q_1)} + \sup\left\{\frac{\|u-p_2(z;z_0)\|_{L^\infty(Q_r(z_0))}}{r^{2+\alpha}} \ : \ Q_r(z_0)\subset Q_1\right\} <\infty.$$
\end{defn}
 As mentioned in \cite{LUISSCHAUDER} and \cite{IMBERTMOUHOT},  for $0<\alpha\le 1$, an equivalent formulation of $C^{0,\alpha}_{\ell}$ regularity is H{\"o}lder continuity with respect to the norm $\|z\| = |t|^{1/2} + |x|^{1/3} + |v|$, in the sense that
 $$|u(z_1)-u(z_2)| \le C\|z_1^{-1}\circ z_2\|^\alpha.$$

 The distance $d(z_0, z_1) = \|z_1^{-1} \circ z_0\|$ with $\|\cdot\|$ defined above is often called the \textit{Kolmogorov distance}, and we also refer to it by the \textit{kinetic distance}. 

 In \cite{IMBERTMOUHOT}, the following version of the local Schauder estimate is proven. 

 \begin{thm}
 \label{thm: Schauderthm}
     Let $g, S: Q_2 \to \R$ be such that $\mcL g = S \in C^{0,\alpha}_{\ell}(Q_2)$. Then $g$ enjoys the estimate 
     \begin{equation}
         \label{eq: Schauderest}
         \|g\|_{C^{2,\alpha}_{\ell}(Q_1)} \le C(n,\alpha)(\|S\|_{C^{0,\alpha}_\ell(Q_2)} + \|g\|_{L^\infty(Q_2)}).
     \end{equation}
     Also, by Lemma $2.7$ of \cite{LUISSCHAUDER}, it is shown that for any function $g$ there holds
     $$\|D^2_v g\|_{C^\alpha_\ell(Q)} +\|Yg\|_{C^\alpha_\ell(Q)} \lesssim \|g\|_{C^{2,\alpha}_\ell(Q)} $$
     for any kinetic cylinder $Q$. In particular, $D^2_v g$ and $Yg$ are classically defined and H{\"o}lder continuous (of mixed order) with respect to the Euclidean distance.
 \end{thm}
The previous theorem is proven in \cite{IMBERTMOUHOT} under more general conditions, allowing for operators with H{\"o}lder continuous and elliptic coefficients as well as lower order terms. See also the references therein for earlier versions of Schauder estimates for Kolmogorov-type equations. We elect to use the cited version because its choice of norms fit better into our approach. We will use this estimate often when performing compactness arguments.  
\subsection{Historical Results and Future Directions}
\label{historical}
The problem \eqref{eq: formulation1} began receiving attention nearly two decades ago. In \cite{PolidoroEXISTENCE}, existence and uniqueness results are proven via the penalization method. Shortly after, a series of papers \cite{FRENTZPOLIDOROOPTIMAL1}, \cite{POLIDOROREG2}, \cite{POLIDOROREG3} address the regularity of solutions in a variety of contexts. Focusing on \cite{FRENTZPOLIDOROOPTIMAL1}, we restate their result in our context as follows:
\begin{thm}
\label{thm: polidorooptimal}
    Let $f$ solve the obstacle problem \eqref{eq: formulation1} on $Q_1$ with obstacle $\psi \in C^{2,\alpha}_\ell(Q_1)$ and continuous boundary data $g$. Then $f$ separates from the obstacle at most quadratically, in the sense that there exists $C=C(n, \|g\|_{L^\infty(\partial_p Q_1)}, \|\psi\|_{C^{2,\alpha}_\ell(Q_1)}))$ such that for any $z_0 \in \{f=\psi\} \cap Q_{1/2},$ there holds
    $$\sup_{Q_r(z_0)}(f-\psi) \le Cr^2.$$
As a consequence,
$$\|D^2_v f\|_{L^\infty(Q_{1/2})} + \|Yf\|_{L^\infty(Q_{1/2})} \le C.$$
In the language of kinetic H{\"o}lder spaces, there holds the estimate
    $$\|f\|_{C^{1,1}_\ell(Q_{1/2})} \le C.$$
\end{thm}
The last estimate is not explicitly written in \cite{FRENTZPOLIDOROOPTIMAL1}, but it is an immediate consequence of the fact that the kinetic first order Taylor expansions of $(f-\psi)$ vanish identically when centered at free boundary points. The $C^{1,1}_\ell$ regularity follows then from the quadratic growth of $f-\psi$. Notice that this estimate implies that, in the Euclidean framework, $f$ belongs to  $(C^{0,1}_t \cap C^{0,2/3}_x \cap C^{1,1}_v)(Q_{1/2})$. Indeed, supposing for instance that $(0,0,0) \in \{f=\psi\}$, the quadratic growth implies that $(f-\psi)(t,x,v) \le C(|t|+ |x|^{2/3} + |v|^2)$. This result is claimed to be optimal, and it is indeed optimal within the class of kinetic H{\"o}lder spaces $C^{k,\alpha}_\ell$. However, in the case of obstacles $\psi$ which are $C^4_{t,x,v}$ with respect to the Euclidean distance, our Theorem \ref{thm: optimalreg} provides the surprising result that the regularity of $f$ in the spatial variable $x$ can be improved to $C^{0,1}_x,$ that is, Lipschitz with respect to the Euclidean distance. This is surprising because, at least for obstacle problems for local differential operators and smooth obstacles, $C^{1,1}$ regularity with respect to the inherent distance is typically the best regularity possible. Indeed, solutions to the elliptic obstacle problem $\min\{-\Delta f, f-\psi\}=0$ enjoy $C^{1,1}$ local regularity with respect to the Euclidean distance, and solutions to the parabolic problem  $\min\{(\partial_t-\Delta)f, f-\psi\} =0$ enjoy $C^{1,1}$ local regularity with respect to the parabolic distance. Historically, the optimal regularity for obstacle problems is always a consequence of proving that solutions separate from the obstacle at most quadratically with respect to the distance from the free boundary. Therefore, our Theorem \ref{thm: optimalreg} is very surprising in that the regularity in $x$ can be improved beyond that which is inherited from quadratic growth. To our knowledge, this is the first instance of such a phenomenon, at least in the case of local differential operators. We want to emphasize again that Theorem \ref{thm: polidorooptimal} is indeed optimal within the scale of kinetic H{\"o}lder spaces, as well as that the authors in that paper treat more general operators and obstacle problems. 

Finally, we remark that the study of the free boundary regularity for \eqref{eq: formulation1}-\eqref{eq: formulation2} is a longstanding unresolved problem. In addition to the aforementioned series of papers analyzing the regularity of solutions to  \eqref{eq: formulation1} and \eqref{eq: formulation2}, over the past couple of decades there has also been a series of papers pursuing Boundary Harnack regularity estimates for Kolmogorov operators in domains which are Lipschitz with respect to the Kolmogorov distance, as defined in the previous subsection. See \cite{POLIDOROBOUNDARYHARNACK}, \cite{POLIDOROCARLSEON}, \cite{NYSTROMPOLIDORO}, \cite{LISTGARDNYSTROM} and the references therein for more work in this challenging area. In \cite{FRENTZPOLIDOROOPTIMAL1} and \cite{POLIDOROCARLSEON} it was stated that their long term goal was to study the regularity of the free boundaries for solutions to the obstacle problem associated with Kolmogorov operators. This problem proves to be very challenging, for the reasons discussed throughout this introduction. It is our belief that the $C^{0,2/3}_x$ regularity in the spatial variable $x$ obtained in the first regularity result (Theorem \ref{thm: polidorooptimal}) was not enough to analyze the free boundary using the existing methods. We hope that our improved regularity in Theorem \ref{thm: optimalreg} opens the door to a complete theory of free boundary regularity for the obstacle problem \eqref{eq: formulation2}, and we take the first step with Theorem \ref{thm: freeboundaryregularity}. 

We also highlight recent interest in the obstacle problem \eqref{eq: formulation1} found in the paper \cite{RebucciAnceschi}, wherein Rebucci and Anceschi investigated the variational structure of the problem and its associated variational inequalities.

\subsection{Notation}
\begin{itemize}
    \item Points in $\R^- \times \R^n \times \R^n$ are denoted by $z=(t,x,v)$ -- time, space, velocity, respectively. 
    \item $Y\coloneqq \partial_t - v\cdot\nabla_x$ is the material derivative
    \item $\mcL \coloneqq \Delta_v - Y$ is the Kolmogorov operator
    \item $\mathcal{H} \coloneqq \Delta_v - \partial_t$ is the Heat operator
    \item $\Lambda_u = \{z \in Q_1 \ : \ u(z)=0\}$ is the contact set of $u$.
    \item $\Omega_u \coloneqq \{u>0\} \cap Q_1$ is the positivity set of $u$
    \item $\Gamma_u \coloneqq \partial_p \{u>0\} \cap Q_1$ is the free boundary of $u$
    \item $\partial_{v_e} = e \cdot \nabla_v , \partial_{x_e} = e \cdot \nabla_x$
    \item $Q_r \coloneqq \{(t,x,v) \ : \ -r^2 < t\le 0, |x|<r^3, |v|<r\}$ is the kinetic cylinder of radius $r>0$ centered at the origin
    $Q_r(z_0)= z_0 \circ Q_r  \coloneqq \{(t,x,v) \ : \ -r^2 <t-t_0\le 0, |x-x_0 -(t-t_0)v_0|<r^3, |v-v_0|<r\}$ is the kinetic cylinder of radius $r>0$ centered at $z_0=(t_0,x_0,v_0)$.
    \item $\partial_p Q_r(z_0)$ is the parabolic boundary of $Q_r(z_0)$, that is, $\partial Q_r(z_0) \setminus (\{t_0\} \times B_{r^3}(x_0)\times B_r(v_0)).$
    \item $|\cdot|$ may be used to denote the $\R^m$ Euclidean norm of a vector, or the $2n+1$-dimensional Lebesgue measure of a set $A\subset \R^{1+2n}$. There will never be any ambiguity.   \item $C^k_{t,x,v}$ denotes the class of functions which are $k$-times continuously differentiable in all variables $(t,x,v)$ with respect to the Euclidean distance.
    \item $(C^{0,\alpha}_{t,x} \cap C^{1,\alpha}_v)(Q_1)$ (and its variants) denote the class of functions $u$ which satisfy
    $$\|u\|_{L^\infty(Q_1)} + [u]_{C^\alpha(Q_1)} + \sup_{(t,x,v) \in Q_1} [\nabla_v u(t,x,\cdot)]_{C^\alpha(B_1)} <\infty$$
    where the $[\cdot]_{C^\alpha}$ semi-norms are the standard Euclidean H{\"o}lder semi-norms. 
    \item $A(t) \coloneqq  \{(x,v) \ : \ (t,x,v) \in A \subset \R^{1+2n}\}$ is the $t-$section of a set $A \in \R^{1+2n}$. An analogous definition holds for $(t,x)-$sections.

\end{itemize}
\subsection{Local and Global Solutions}
\begin{defn}
    Given $r,M>0$ and $z_0=(t_0,x_0,v_0) \in Q_1$, we define 
    $\mathcal{P}_r(M;z_0)$ to be the class of functions $u$ solving
    \eqref{eq: formulation2} in $Q_1$, with $\|u\|_{L^\infty(Q_1)} \le M$ and $z_0 \in \Gamma$. In the case $z_0=(0,0,0)$, we denote the class by $\mathcal{P}_r(M)$.
\end{defn}
Solutions of the equation \eqref{eq: formulation2} in $C^{1,1}_\ell$ satisfy the equation in the weak (distributional) sense:
$$\int_{Q_1} u(\Delta_v \eta + \partial_t \eta + v\cdot \nabla_x\eta) \ dv \ dx \ dt = \int_{\Omega_u} \eta \ dv \ dx \ dt$$
for all smooth test functions $\eta$ compactly supported in $Q_1$. By standard regularity theory, $u$ is smooth and classically differentiable on $\Omega_u$, and by Theorem \ref{thm: optimalreg}, the derivatives $D^2_v u, \partial_tu$, and $\nabla_x u$ exist almost everywhere and are locally bounded in the $L^\infty$ sense. Finally, we clarify the definition of $\Gamma_u = \partial_p\Omega_u$ as the set of all points $z_0 \in \{u=0\}=\Lambda_u$ such that $Q_r(z_0) \cap \{u>0\} = Q_r(z_0) \cap \Omega_u \neq \emptyset$ for all $r>0$. 

We will also need the following definition.

\begin{defn}
    Given $M>0$, denote by $\mathcal{P}_\infty (M)$ the class of functions $u$ on $\R^- \times \R^n \times \R^n$ such that
    \begin{itemize}
    \item
    $\mcL u = \chi_{u>0} \text{ in } \R^- \times \R^n \times \R^n$
    \item $|u(t,x,v)| \le M(1+|t| + |x|^{2/3} + |v|^2)$,
    \item $(0,0,0) \in \Gamma_u$.
    \end{itemize}
 In the case $u$ is independent of $x$, we refer to by $\mathcal{P}_\infty^p(M)$ the class of global solutions to the parabolic obstacle problem in the same exact way. See Definition $1.2$ of \cite{CPS}.

\end{defn}
The superscript $p$ in the above definition stands for \textit{parabolic}. We will see later on in fact that $\mathcal{P}_\infty(M) = \mathcal{P}_\infty^p(M)$.
\subsection{Blow-ups}
We will follow the traditional approach to the study of free boundary problems by freezing a free boundary point $z_0 \in \Gamma_u$, re-centering, and zooming in 
\begin{equation}
\label{eq: rescaling}
    u_r^{z_0}(z) \coloneqq \frac{u(z_0 \circ S_r(z))}{r^2} = \frac{u(t_0 + r^2 t, x_0 + r^3x + r^2tv_0, v_0 + rv)}{r^2} 
\end{equation}
As usual, we have that $\mcL u_r^{z_0} = \chi_{u_r^{z_0}>0}$ thanks to the scaling and left-invariance of the operator $\mcL$. Accordingly we may assume that $z_0 =(0,0,0)$. Owing to the $C^{1,1}_{\ell}$ estimates and the Lipschitz estimates of Theorem \ref{thm: optimalreg}, we will see that if $u \in \mathcal{P}_1^k(M)$, then as $r \to 0$ the rescalings $u_r$ converge in $ C^{0,\alpha}_t \cap C^{0,1}_x \cap C^{1,\alpha}_v$ to a global solution $u_0$ of the parabolic problem. This will be discussed further in Section \ref{blowups}. The function $u_0$ is said to be a \textit{blow-up} of $u$ at $z_0$. For global solutions, there is also the concept of \textit{shrink-down}, sending $r \to \infty$. 

Notice that $\partial_t u_r^{z_0}(\cdot) = (\partial_tu + v_0\cdot \nabla_xu)(z_0 \circ S_r(\cdot))$ and $\nabla_x u_r^{z_0}(\cdot) = (r\nabla_xu)(z_0\circ S_r(\cdot))$. The fact that $\partial_t$ is not preserved by Galilean translation, and that $\nabla_x$ scales like $r^3$, mean that some information about $\partial_t u$ and $\nabla_x u$ is lost when performing blow-up arguments. See, for instance Lemma $3.4$ of \cite{SUBELLIPTIC} and Theorem $8.1$ of \cite{CPS}. 
\subsection{Outline of the Paper}
In Section \ref{regularitysec}, we use Bernstein's method and a penalization procedure to prove Theorem \ref{thm: optimalreg}. The method is quite robust and we believe it could lend itself well to other obstacle problems. In Section \ref{nondegensec}, we develop the basic properties of solutions of \eqref{eq: formulation2}, such as nondegeneracy on kinetic cylinders centered at free boundary points. Even though we don't make use of it, we also discover a form of \textit{parabolic} nondegeneracy, which is a curious and rare property to find in an obstacle problem with no extra monotonicity assumptions. Section \ref{monotonicitysec} introduces a new monotonicity formula for the problem \eqref{eq: formulation2}, making critical use of the new estimate on $\nabla_xu$ in order to view the rescalings $u_r$ as being almost constant in $x$. In Section \ref{blowups}, we find that blow-ups for the problem \eqref{eq: formulation2} are independent of $x$ and solve the parabolic obstacle problem. We use this to classify the free boundary in three ways. In Section \ref{balanced} we study free boundary points at which the contact set $\Lambda_u$ is sufficiently thick at a single scale, as measured by the thickness function $\delta_r(u)$. We show that at such free boundary points, the transport derivative $Yu$ vanishes continuously, as well as that solutions are monotone in velocity cones. We conclude in the final section with the proof of Theorem \ref{thm: freeboundaryregularity} after establishing a \textit{pointwise} lower bound on velocity cones. In every single section, we make critical use of the new Lipschitz estimate $\|\nabla_xu\|_{L^\infty(Q_{1/2})} \le C$ coming from Theorem \ref{thm: optimalreg}.
\subsection{Acknowledgements}
The author would like to thank Luis Silvestre for suggesting this problem and for many helpful discussions.
\section{Regularity Estimates}
\label{regularitysec}
Since the estimate on $\nabla_xf$ is perhaps surprising, we begin by explaining why there is some reason to expect it. To fix ideas, consider $f$ solving 
$$\begin{cases} \mcL f\le 0\text{ and } f \ge \psi & \text{ on }Q_1,\\ 
f>\psi & \text{ on } \partial_p Q_1.\end{cases}$$

Assuming that $\psi$ is smooth and that the boundary datum for $f$ is smooth, the assumption that $f>\psi$ on $\partial_p Q_1$ implies that $f>\psi$ on $Q_1 \setminus Q_{1-\delta}$ for some $\delta>0$. In other words, the contact set is compactly contained within $Q_1$. Therefore, taking $e_i, e_j$ to be arbitrary standard basis vectors of $\R^n$ and parameters $\alpha^2 +\beta^2 + \gamma^2=1$, there exists $C>0$ depending on $\|f\|_{C^{0,1}(\partial_pQ_1)_{t,x,v}}$ and $\|\psi\|_{C^{0,1}(Q_1)_{t,x,v}}$ such that 
\begin{align}
    f^h((t,x,v))  &\coloneqq f((\alpha h, \beta he_i, \gamma he_j) \circ (t,x,v)) + C|h|\\
    &\ge  \psi((\alpha h, \beta he_i, \gamma he_j)\circ(t,x,v)) + C|h| \ge \psi(t,x,v)
    \end{align}
for any $|h|<\delta$.

Crucially, the operator $\mcL$ remains invariant under the left-action of the group: for any smooth $g$, there holds $\mcL(g(z_0 \circ \cdot)) = (\mcL g)(z_0\circ \cdot)$. Therefore $\mcL f^h \le 0$ where it is defined and so, interpreting $f$ as the least supersolution above the obstacle on $Q_r$ for any $0<r\le 1$ in the sense that
$$f = \inf\{\tilde{f} \ : \ \mcL \tilde{f}  \le 0, \tilde{f} \ge \psi \text{ in } Q_1, \tilde{f} \ge f \text{ on } \partial_p Q_r\},$$
we arrive at $f \le f^h.$ Taking $h \to 0$, one arrives at
$$|(\alpha \partial_t + \beta \partial_{x_i} + \gamma(t\partial_{x_i} + \partial_{v_i}))f(z)| \le C \ \text{ for } z \in Q_{1-\delta}.$$
In this way we derive Lipschitz estimates on $\partial_t f, \nabla_x f$ and $(t\nabla_x + \nabla_v)f$. The same approach applied to second order-difference quotients implies lower bounds of the form
$$\partial_{tt}f, \ \  \partial_{x_ex_e}f, \ \ (t\partial_{x_e}+\partial_{v_e})^2f \ge -M$$
for some constant $M$ depending on the regularity of $\psi$ and the boundary data of $f$. A version of this approach may be found in \cite{QUASICONVEXITY} for the parabolic case. It is quite common in the theory of nonlocal obstacle problems, see for example, the recent paper \cite{SERRA}.

To summarize, there is hope for an estimate on $\nabla_x f$ because the operator $\mcL$ is invariant under left-translation in the $x$-variable. This makes the difference $f(t, x+he, v) - f(t,x,v)$ a supersolution on the set where $f$ lies above the obstacle, and therefore it may not take an interior minimum. However, this approach relies on stronger boundary regularity on $f$ than we would like to impose. It is our aim to derive interior estimates making lax requirements on the boundary data. Still, the observation that the maximum principle should imply Lipschitz regularity of $f$ in the $x-$variable naturally leads one to consider Bernstein's approach.

Indeed, our approach to the optimal regularity estimates is to use Bernstein's method and the penalization technique. We treat first the original obstacle formulation \eqref{eq: formulation1}. The penalized problem is formulated as
\begin{equation}
\label{eq: penalizedprob2}
\begin{cases} \mcL f^\eps = \beta_\eps(f^\eps-\psi) & \text{ in } Q_1, \\ f^\eps = g & \text{ on } \partial_p Q_1.
\end{cases}
\end{equation}
where $g:\partial_pQ_1 \to \R$ is continuous and satisfies $g(z)>\psi(z)$ for all $z \in \partial_p Q_1$, and the nonlinearity $\beta_\eps(\cdot)$ is a smooth, increasing, concave function which is non-positive and such that 
$$\begin{cases}\beta_\eps (s) =0 & \text{ for } s\ge 0, \\ 
\beta_\eps(s) =\displaystyle \frac{s}{\eps} & \text{ for } s<-\eps,\\
\beta'(s) \ge 0, \text{ and }\beta''(s) \le 0 & \text{ for } s \in \R.
\end{cases}$$
This is the same penalization function utilized in \cite{PSUBOOK}. As usual, the idea is that the nonlinearity penalizes $f^\eps$ from lying too far below $\psi$. Let us note that the existence and uniqueness of smooth solutions to \eqref{eq: penalizedprob2} is standard. Our purpose is to obtain estimates independent of $\eps$ that allow us to pass to a limit $f^\eps \to f$ solving the original obstacle problem \eqref{eq: formulation1}. It will soon be convenient for us to subtract off the obstacle and set $u^\eps \coloneqq f^\eps - \psi$. The boundary data becomes $g-\psi \coloneqq h$, where $h: \partial_pQ_1 \to (0,\infty)$ is positive and continuous. Setting $\vphi \coloneqq -\mcL\psi$, we arrive at 

\begin{equation}
    \label{eq: penalizedprob}
    \begin{cases} \mcL u^\eps = \beta_\eps(u^\eps) + \vphi & \text{ in } Q_1, \\
    u^\eps =h & \text{ on } \partial_p Q_1.
    \end{cases}
\end{equation}

First we record the following elementary $L^\infty$ estimate for $f^\eps$ and $u^\eps$. 
\begin{lem}
\label{lem: linfty}
Let $f^\eps$ solve \eqref{eq: penalizedprob2} with $g$ and $\psi$ as stated in Theorem \ref{thm: optimalreg}. Then
$$\|f^\eps\|_{L^\infty(Q_1)} \le \max\{\|g\|_{L^\infty(\partial_pQ_1)}, \|\psi\|_{L^\infty(Q_1)}\} .$$
As a consequence, with $u^\eps$ as in \eqref{eq: penalizedprob}, there holds on $Q_1$ the two-sided bound
\begin{equation}
\label{eq: L^inftyestimate}
-\max\{1, \|\vphi\|_{L^\infty(Q_1)}\} \eps \le u^\eps \le \max\{\|g\|_{L^\infty(\partial_pQ_1)}, \|\psi\|_{L^\infty(Q_1)}\} + \|\psi\|_{L^\infty(Q_1)}.
\end{equation}

For future purposes we also record the estimate
\begin{equation}
\label{eq: beta(u)}
-2\|\vphi\|_{L^\infty}\le  \beta_\eps(u^\eps)+\vphi \le \|\vphi\|_{L^\infty}.
\end{equation}

\end{lem}
\begin{proof}
     As $\beta_\eps(\cdot) \le 0$, $f^\eps$ is a supersolution and thus
    $$\min_{Q_1} f^\eps = \min_{\partial_p Q_1} g.$$
    And then noting that
    $$\mcL f^\eps(z) = 0 \text{ for } z \in \{z \in Q_1 \ : \ f^\eps(z)> \psi(z)\}$$
    we may apply the maximum principle to determine
    $$\max_{\{f^\eps > \psi\}} f^\eps \le \max\{\|\psi\|_{L^\infty(Q_1)}, \|g\|_{L^\infty(\partial_p Q_1)}\},$$
    since the parabolic boundary of the set $\{f^\eps >\psi\}$ consists of either points where $f^\eps = \psi$ or points where $f=g$ on $\partial_p Q_1$. We conclude that
    $$\min_{\partial_p Q_1} g \le f^\eps \le \max\{\max_{\partial_pQ_1} g, \|\psi\|_{L^\infty(Q_1)}\}$$
    and both claims follow.

    Finally, since $u^\eps = h \ge 0$ on $\partial_p Q_1$, if $u^\eps$ is not nonnegative then it must obtain a negative interior minimum at some $z_0\in Q_1$. Necessarily, $\mcL u^\eps(z_0) \ge 0$, implying that $\beta_\eps(u^\eps(z_0)) \ge-\vphi(z_0)$. From the form of $\beta_\eps(\cdot)$ we obtain that either 
    $$u^\eps(z_0) \ge -\eps, \text{ or } u^\eps(z_0) \ge - \eps \vphi(z_0) \ge - \eps \|\vphi\|_{L^\infty(Q_1)}.$$ So, the amount that the penalized solutions may dip below $0$ is on the order of $O(\eps)$. And using that $\beta_\eps'(s) \ge 0$ for all $s$ and that $\beta_\eps(s)=0$ for $s\ge 0$, we find that 
    $$\min_{Q_1} \beta_\eps(u) \ge -\|\vphi\|_{L^\infty(Q_1)}.$$
\end{proof}
In order to have cleaner estimates, let us reformulate the previous estimates as follows.
\begin{cor}
    There exists $M=M(\|\psi\|_{C^2(Q_1)}, \|g\|_{L^\infty(\partial_pQ_1)} >0$ such that
    \begin{equation} \label{eq: Linftyestimate2}
    -M\eps \le u^\eps \le M \text{ and } |\beta(u^\eps) + \vphi| \le M \text{ on } Q_1.
    \end{equation}
\end{cor}
In most of what follows, we drop the $\eps$-indexing for aesthetic reasons and derive estimates for $u= u^\eps$ independent of $\eps$.
The $L^\infty$ estimate will now, in part, enable us to bound first order derivatives by the zero'th order term $u^2$. Before continuing, let us record the facts
\begin{align}
\label{eq: identities}
\mcL(|\nabla_xu|^2) &= 2(\beta'(u)|\nabla_x u|^2 + |\nabla_v\nabla_x u|^2 + \nabla_xu \cdot \nabla_x \vphi),\\
 \mcL(|\nabla_v u|^2) &= 2(\beta'(u)|\nabla_v u|^2 + |D^2_vu|^2 +\nabla_x u \cdot \nabla_v u + \nabla_v u \cdot \nabla_v \vphi),\\
\mcL(\nabla_x u \cdot \nabla_v u) &= |\nabla_xu|^2 + 2(\nabla_x u \cdot \nabla_v u \beta'(u) + 2\nabla_v \nabla_x u : D^2_v u ) + \nabla_x u \cdot \nabla_v \vphi + \nabla_v u \cdot \nabla_x \vphi.
\end{align}

These follow from simple applications of the identities $\mcL(fg) = f\mcL(g) +g\mcL(f) + 2\nabla_v f \cdot \nabla_v g$, as well as the commutators $[\partial_{x_i}, \mcL] =0, [\partial_{v_i}, \mcL] = -\partial_{x_i}.$

 The next proposition is a refinement of Proposition $3.1$ of \cite{IMBERTMOUHOT}, where in our case we must take advantage of the right-hand side having smooth and increasing nonlinear dependence on $u$. The estimate from \cite{IMBERTMOUHOT}, as stated, can not be applied directly, since it implies only that 
$$\|\nabla_{x,v} u^\eps\|_{L^\infty(Q_{1/2})} \lesssim \|u^\eps\|_{L^\infty(Q_1)} + \|\beta_\eps(u^\eps) + \vphi\|_{L^\infty(Q_1)} +\|\nabla_{x,v}(\beta_\eps(u^\eps) + \vphi)\|_{L^\infty(Q_1)}.$$
Here $\nabla_{x,v} = (\nabla_x, \nabla_v)$. This bound will not work for us because $\beta_\eps' = O(1/\eps)$, and we are seeking estimates uniform in $\eps$. Let us give a brief heuristic explanation for why we may improve the above estimate using the structure of the nonlinearity. Suppose $|\nabla_xu^\eps|^2$ takes an interior maximum at $z_0 \in Q_1$, so that $\mcL(|\nabla_xu^\eps|^2)(z_0)\le 0$. Using \eqref{eq: identities}, we see that if $u^\eps(z_0)>0$, then $z_0$ belongs to the set where $\beta_\eps(u^\eps)\equiv 0$, and consequently we would expect that $|\nabla_xu^\eps|^2$ is controlled by $\vphi$ and its derivatives. On the other hand, if $-M\eps < u^\eps(z_0)<0$, then $\beta_\eps'(u^\eps(z_0)) \approx \eps^{-1}$. Using \eqref{eq: identities}, we would arrive at $\beta_\eps'(u^\eps(z_0))|\nabla_x u^\eps(z_0)|^2 \le |\nabla_xu^\eps(z_0)| |\nabla_x \vphi(z_0)|$, implying that $|\nabla_xu^\eps(z_0)| =O(\eps \|\nabla_x \vphi\|_{L^\infty})$. This is not rigorous because, while $\beta_\eps'(s) = \eps^{-1}$ for $s<-\eps$, and $\beta_\eps' \approx \eps^{-1}$ on average for $-\eps<s<0$, it holds that $\beta_\eps'(s) \to 0$ as $s\to 0$. Still, it is reason to believe that the nonlinearity will work in our favor. We now turn to giving the interior estimate.

\begin{prop}[Estimate on $|\nabla_v u| + |\nabla_xu|$]
\label{prop: estimate1}
Under the same conditions as the previous lemma, there exists $C = C(M, \|\vphi\|_{C^1_{t,x,v}(Q_1))})$ such that the solution $u$ to \eqref{eq: penalizedprob} satisfies 
$$\|\nabla_v u\|_{L^\infty(Q_{3/4})} + \|\nabla_x u\|_{L^\infty(Q_{3/4})} \le C(1+\|u\|_{L^\infty(Q_1)}).$$

In particular, the estimate is independent of $\eps$, and holds for any smooth and increasing nonlinearity $\beta(\cdot)$.
\end{prop}
\begin{proof} Set 
$$\Psi \coloneqq \nu u^2 + [a^2 \eta^4 |\nabla_xu|^2 + c\eta^3 \nabla_x u \cdot \nabla_v u + b^2 \eta^2 |\nabla_v u|^2] + A(1-t),$$ where $\eta$ is a cutoff function such that  $\eta \equiv 1$ on $Q_{3/4}$, $\eta$ is supported in $Q_1$, and $\sqrt{\eta} \in C^\infty$. The parameters $A>0$, $\nu>0$ and $0<c<ab$, $0<a<b$ are to be decided. We will show there exist choices of parameters such that $\Psi$ is a subsolution, that is, that $\mcL(\Psi) \ge 0$. From this the conclusion follows in the standard way.

Before embarking on the proof, we note that the condition that $\sqrt{\eta} \in C^\infty$ ensures that $|\nabla_v \eta|^2 \lesssim \eta$, pointwise. Therefore

$$|\mcL(\eta^{k+1})| =|(k+1)\eta^k \mcL(\eta) + k(k+1)\eta^{k-1}|\nabla_v \eta|^2| \lesssim \eta^k \text{ for any } k \ge 0.$$ 

Using \eqref{eq: identities}, we compute that
\begin{align} \label{eq: L(psi)1} \mcL \Psi  &= A+ 2\nu u(\beta(u)+\vphi) + |\nabla_v u|^2)  \\
&+a^2\left(\mcL(\eta^4)|\nabla_x u|^2 + 16\eta^3 \nabla_v \eta \cdot \nabla_v\nabla_x u \nabla_x u + 2\eta^4(\beta'(u)|\nabla_x u|^2 + |\nabla_v\nabla_x u|^2 + \nabla_xu \cdot \nabla_x \vphi)\right)\\
&+ c\left(\mcL(\eta^3)\nabla_x u \cdot \nabla_v u + 6\eta^2 \nabla_v \eta \cdot(D^2_v u \nabla_x u + \nabla_v \nabla_x u \nabla_vu)\right)\\
&+ c\left(2\eta^3\left(\frac12|\nabla_xu|^2 + \nabla_x u\cdot \nabla_v u \beta'(u) + \nabla_v\nabla_xu : D^2_vu+ \nabla_xu \cdot \nabla_v \vphi + \nabla_v u \cdot \nabla_x \vphi\right)\right)\\
&+ b^2\left(\mcL(\eta^2)|\nabla_vu|^2 + 8\eta^2\nabla_v \eta \cdot D^2_v u \nabla_v u + 2\eta^2(\beta'(u)|\nabla_vu|^2 + \nabla_xu\cdot \nabla_v u + |D^2_vu|^2 + \nabla_v u \cdot \nabla_v \vphi)\right).
\end{align}

The first observation we make is that we may collect all of the $\beta'$ terms in \eqref{eq: L(psi)1} and group them together as
\begin{equation}
    \label{eq: betaeq}
    2\beta'(u)\eta^2(a^2\eta^2|\nabla_xu|^2 + c\eta\nabla_x u \cdot \nabla_v u + b^2|\nabla_vu|^2).
\end{equation}
For $0<c<ab$, it holds that
$$|c\eta \nabla_x u \cdot \nabla_vu| \le |a\eta \nabla_x u| |b\nabla_vu| \le\frac12( a^2\eta^2|\nabla_xu|^2  + b^2|\nabla_vu|^2).$$
As $\beta' \ge0$, the term \eqref{eq: betaeq} is nonnegative and we may discard it. At this point, the proof proceeds essentially the same as in \cite{IMBERTMOUHOT}. We are going to apply Cauchy's inequality $|xy| \displaystyle\le \eps |x|^2 + \frac{C}{\eps}|y|^2$ repeatedly. 

Let us first handle the $a$ term. With the stated bounds on $\mcL(\eta^4)$, we may control
\begin{align}
|\mcL(\eta^4)||\nabla_xu|^2 & \le C \eta^3 |\nabla_xu|^2,\\
16\eta^3 |\nabla_v \eta \cdot \nabla_v\nabla_x u \nabla_x u| &\le \eta^4 |\nabla_v\nabla_x u|^2 +  C\eta^3|\nabla_x u|^2\\
2\eta^3|\nabla_x u \cdot \nabla_x\vphi| &\le  \eta^3 |\nabla_xu|^2 +  |\nabla_x\vphi|^2,
\end{align}
using again that $|\nabla_v \eta|^2 \lesssim \eta \le 1$. We remark here that the value of $C=C_\eta$ is chosen at the end to be as large as it needs to be to satisfy all of the Cauchy inequalities involving $\eta$, and may be changing from line to line. We find that the $a^2$ term is bounded below (excluding the $\beta'$ term) by
\begin{align}
\label{eq: aterm}
a^2(-C \eta^3|\nabla_xu|^2 +\eta^4 |\nabla_v\nabla_xu|^2 -|\nabla_x\vphi|^2).
\end{align}
Note we have absorbed the term $\eta^3|\nabla_xu|^2$ into $C\eta^3|\nabla_xu|^2$. 

Now we continue with the $c$ term. Again disregarding the $\beta'$ term, we introduce a new parameter $\delta>0$ to bound below the $c$ term by
\begin{align}
\label{eq: cterm}
c\left((1-3\delta)\eta^3|\nabla_x u|^2 - \delta \eta^4 |\nabla_v\nabla_x u|^2 - \frac{C}{\delta}(\eta |\nabla_vu|^2+ \eta^2|D^2_vu|^2) - \frac{C}{\delta}(|\nabla_{x,v}\vphi|^2)\right).
\end{align}
Here $\nabla_{x,v} = (\nabla_x, \nabla_v)$. This lower bound is a consequence of the following estimates which hold for all $\delta>0$ and $C>0$ large:

\begin{align}
|\eta^{1/2}\eta^{3/2} \nabla_v u \cdot \nabla_x u | &\le \delta \eta^3 |\nabla_xu|^2 + \frac{C}{\delta}\eta |\nabla_vu^3\\
|6 \eta^2 \nabla_v \eta \cdot D^2_vu \nabla_xu| &\le \delta \eta^3|\nabla_xu|^2 + \frac{C}{\delta}\eta^2|D^2_vu|^2,\\
|6\eta^2 \nabla_v \eta \cdot(\nabla_v\nabla_xu \nabla_vu)| &\le \delta \eta^4 |\nabla_v\nabla_x u|^2 + \frac{C}{\delta}\eta |\nabla_vu|^2,\\
|\nabla_v\nabla_x u : D^2_v u| &\le \delta \eta^4 |\nabla_v\nabla_xu|^2 + \frac{C}{\delta}\eta^2|D^v_u|^2,\\
\eta^3|(\nabla_xu, \nabla_v u)\cdot(\nabla_v\vphi, \nabla_x\vphi)| &\le \delta \eta^3|\nabla_xu|^2 + \eta |\nabla_vu|^2 +\left(|\nabla_x \vphi|^2 +\frac{C}{\delta} |\nabla_v\vphi|^2)\right).
\end{align}
Note that we have used $\eta^2 \nabla_v \eta = 2\eta^{3/2} \eta^{1/2}\nabla_v\sqrt{\eta}$. Note also that we may absorb the $\eta|\nabla_vu|^2$ term from the fifth inequality into the $\displaystyle \frac{C}{\delta}\eta|\nabla_vu|^2$ term. This explains \eqref{eq: cterm}. 

To handle the $b^2$ coefficient, we introduce another parameter $\mu$ for our Cauchy inequalities. Proceeding as usual, we bound below by 
\begin{align}
\label{eq: bterm}
b^2\left(-\frac{C}{\mu} \eta |\nabla_vu|^2 -\mu \eta^3 |\nabla_xu|^2 + 2\eta^2 |D^2_vu|^2 - |\nabla_v\vphi|^2\right).
\end{align}

We highlight that we have bounded $2\eta^2| \nabla_v u \cdot \nabla_v \vphi| \le \eta|\nabla_vu|^2 + |\nabla_v\vphi|^2$, using that $0\le \eta \le 1$. Then, we absorb the $-\eta|\nabla_v u|^2$ term into the $\displaystyle -\frac{C}{\mu}\eta|\nabla_vu|^2$ term. Finally, we remark that the term $2\nu u(\beta(u)+\vphi)$ from \eqref{eq: L(psi)1} may be bounded below by $-CM^2\nu $, where $M$ is as in \ref{eq: Linftyestimate2}.

At this point, we put together \eqref{eq: L(psi)1}, \eqref{eq: aterm}, \eqref{eq: cterm}, and \eqref{eq: bterm}. Grouping by derivatives,
\begin{align}
\label{eq: finalestimate}
    \mcL \Psi &\ge A -CM^2\nu + \left((1-3\delta)c - \mu b^2- C_\eta a^2\right)\eta^3|\nabla_xu|^2 + \left(2\nu - \frac{C}{\mu}b^2 - \frac{C}{\delta}c\right)\eta |\nabla_vu|^2 \\
    &+ \left(2b^2- \frac{C}{\delta}c\right) \eta^2|D^2_vu|^2 +\left(a^2 - \delta c\right)\eta^4|\nabla_v\nabla_x u|^2 - C\|\vphi\|_{C^1}.
    \end{align}
First, take $a=1$, and take $c>> C_\eta$, say $c=100C_\eta$. Then take $\delta$ small enough so that the coefficient of the $|\nabla_v\nabla_xu|^2$ term is positive, as well as so that $(1-3\delta)c > \displaystyle \frac{1}{2}$. With this choice of $c$ and $\delta$, choose $b>c$ large so that the coefficient of the $|D^2_vu|^2$ term is positive. Then choose $\mu$ small so that $\mu b^2 < C_\eta$, which ensures the positivity of the $|\nabla_xu|^2$ coefficient. Next, choose $\nu$ as large as necessary to guarantee the $|\nabla_vu|^2$ term is positive. Finally, choose $A= CM^2\nu + C\|\vphi\|_{C^1}$. 

The estimate now follows in the usual way with the maximum principle: 
\begin{align}
a^2 \sup_{Q_{3/4}}( |\nabla_x u|^2 + |\nabla_v u|^2 ) &\le \sup_{Q_{3/4}} a^2|\nabla_x u|^2 + b^2|\nabla_v u|^2 \\
&\le \sup_{Q_{3/4}}2(a^2|\nabla_x u|^2 + 2c\nabla_x u \cdot \nabla_v u + b^2|\nabla_v u|^2) \\
&\le2 \sup_{Q_1} \Psi \le 2(A + \sup_{\partial_p Q_1} \nu u^2.)
\end{align}
\end{proof}
With the above proposition in hand, we may now estimate the time derivative. 
\begin{prop}[$|\partial_tu|$ estimate] Under the same conditions as the previous proposition, the solution $u$ to \eqref{eq: penalizedprob} satisfies 
$$\|\partial_tu\|_{L^\infty(Q_{2/3})} \le C(1+ \|u\|_{L^\infty(Q_1)})$$
for some $C=C(M, \|\vphi\|_{C^1_{t,x,v}(Q_1)})$.
As in the previous proposition, the estimate is independent of $\eps$ and holds for any smooth, increasing nonlinearity $\beta$.

\end{prop}
\begin{proof} 
We use Bernstein's technique again. As mentioned in the introduction, the key point for this step is that
$$|\partial_tu| \le M + |D^2_v u| + |\nabla_x u| \text{ on }Q_1.$$ 

This follows thanks to \eqref{eq: Linftyestimate2}. Now, observe $$\mcL((\partial_t u)^2) = 2(\beta'(u)(\partial_t u)^2 + |\nabla_v \partial_tu|^2+ 2\partial_t u \partial_t \vphi)$$ and, again using \eqref{eq: identities}
$$\mcL(e^{-\lambda t} |\nabla_x u|^2) = e^{-\lambda t}(2\beta'(u) |\nabla_x u|^2 + 2 |\nabla_x \nabla_v u|^2 +\nabla_x u \cdot \nabla_x\vphi+ \lambda |\nabla_x u|^2).$$ Considering smooth $\eta$ such that $\eta \equiv 1$ on $Q_{2/3}$ and vanishing near $\partial_p Q_{3/4}$, set 
$$\Psi \coloneqq \eta^2(\partial_t u)^2 + A|\nabla_v u|^2 + Ae^{-\lambda t}|\nabla_x u|^2 +\nu(1-t)$$ for $\nu,A,\lambda >0$ to be decided. Similarly to before, we prove that under a certain choice of parameters, $\Psi$ is a subsolution on $Q_{3/4}$.  Indeed,
\begin{align} \label{eq: Lpsi2}
\mcL \Psi &= \mcL(\eta^2) (\partial_t u)^2 + 8\eta \partial_tu \nabla_v \eta \cdot \nabla_v \partial_tu + 2\eta^2(\beta'(u)(\partial_t u)^2
+ |\nabla_v \partial_tu|^2+ \partial_tu \partial_t\vphi)\\
&+ 2A(\beta'(u)|\nabla_v u|^2 + |D^2_vu|^2 + \nabla_v u \cdot \nabla_v \vphi)\\
&+ 2Ae^{-\lambda t}(\beta'(u)|\nabla_xu|^2 + |\nabla_v \nabla_x u|^2 + \nabla_x u \cdot \nabla_x \vphi+\frac{\lambda}{2} |\nabla_xu|^2)+\nu.
\end{align}

The strategy is the same. Any term involving $\beta'(u)$, we may discard, as they are all nonnegative. The right-hand side of the first line in \eqref{eq: Lpsi2} may be bounded below with Cauchy's inequality by 
$$-C (\partial_t u)^2 + \eta^2 |\nabla_v \partial_tu|^2- |\partial_t \vphi|^2 \ge -C(M^2+ |\nabla_x u|^2 + |D^2_vu|^2) + \eta^2 |\nabla_v \partial_tu|^2 -|\partial_t \vphi|^2.$$
As usual, introduce a parameter $\eps>0$ and use Cauchy's inequality to find
\begin{align}
    2A |\nabla_v u \cdot \nabla_v\vphi| \ge -\eps A |\nabla_vu|^2 - \frac{A}{\eps} |\nabla_v \vphi|^2,
    2A |\nabla_x u \cdot \nabla_x \vphi| \ge-\eps A |\nabla_xu|^2 - \frac{A}{\eps} |\nabla_x \vphi|^2.
\end{align}
It is then clear we may bound

\begin{align}
\mcL \Psi &\ge -C(M^2+|D^2_vu|^2 + |\nabla_x u|^2) + \eta^2|\nabla_v \partial_tu|^2  + 2A|D^2_vu|^2 + Ae^{-\lambda t}\lambda |\nabla_xu|^2\\
&- \eps A |\nabla_xu|^2 - \eps A |\nabla_vu|^2 -\frac{CA}{\eps}\|\vphi\|_{C^1_{t,x,v}}^2 + \nu.
\end{align}
Again we remark that the value of $C = C_\eta$ is changing, but is still absolute and depends only on $\eta$. We take $\lambda=1$, and note that $e^{- t} \ge 1$ on $Q_1$. Taking $A =2C$, the coefficient of $|D^2_vu|^2$ is $(2A-C) \ge 0$, and the coefficient of $|\nabla_xu|^2$ is at least $(1-\eps)A - C \ge 0$ if we choose $\eps = \displaystyle \frac12$. We then choose any 
$$\nu> 4C^2 \|\nabla \vphi\|_{C^1_{t,x,v}}^2 + CM^2 + C\|\nabla_v u\|_{L^\infty(Q_{3/4})}^2$$
in order to ensure $\mcL \Psi \ge 0$ on $Q_{3/4}$. The proposition follows then from the maximum principle and Proposition \ref{prop: estimate1}, as $\|\nabla_vu\|_{L^\infty(Q_{3/4})} \le C(M,\|\vphi\|_{C^1_{t,x,v}(Q_1)}).$
\end{proof}
So far, we have only used that the nonlinearity $\beta(\cdot)$ is increasing. In the following step, we use concavity in a critical way. 
\begin{prop}
    Under the same conditions as the previous proposition, let $u$ solve \eqref{eq: penalizedprob}. Then there exists $C = C(M, \|\vphi\|_{C^2_{t,x,v}(Q_1)})$ such that $u$ enjoys the semiconcavity estimate 
    $$\|(D^2_vu)^-\|_{L^\infty(Q_{1/2})}\le C(1+\|u\|_{L^\infty(Q_1)}).$$
    The estimate holds for any smooth, increasing, and concave nonlinearity $\beta(\cdot)$. 
\end{prop}
\begin{proof} 
Let $$\partial_{v_ev_e}^-u \coloneqq \begin{cases}- \partial_{v_ev_e}u & \partial_{v_ev_e}u<0,\\ 0 & \text{ otherwise}\end{cases}$$
for an arbitrary direction $e \in \mathbb{S}^{n-1}$. 
Set $$\Psi \coloneqq \eta^2(\partial_{v_ev_e}^-u)^2 + a|\nabla_v u|^2 + be^{-t}|\nabla_x u|^2 + cu^2 + \nu(1-t)$$ for $a,b,c,\nu>0$ to be decided, $\eta$ a smooth cutoff function which is identically $1$ on $Q_{1/2}$ and vanishes near $\partial_p Q_{2/3}$. Our goal is to show that $\mcL \Psi \ge 0$ on $Q_{2/3}$ under a suitable choice of parameters. First, we show that $\Psi$ is a subsolution on the open set $\{\partial_{v_ev_e}u>0\} \cap Q_1$. Indeed, on this set
\begin{align}
\label{eq: Lpsi4}
\mcL \Psi &= 2a(\beta'(u)|\nabla_vu|^2 + |D^2_vu|^2 +\nabla_v u \cdot\nabla_xu + \nabla_vu \cdot \nabla_v \vphi)\\
&+be^{-t}(2|\nabla_x u|^2\beta'(u) + 2|\nabla_v\nabla_x u|^2 + |\nabla_x u|^2 + \nabla_x u \cdot \nabla_x \vphi) + 2c(u(\beta(u)+\vphi) + |\nabla_vu|^2) + \nu.
\end{align}

As usual, we may disregard the $\beta'$ terms as they are nonnegative. On $Q_1$ we may bound below 
\begin{align}
    2a\nabla_v u \cdot \nabla_x u &\ge -a|\nabla_vu|^2 -a|\nabla_x u|^2,\\
    2a \nabla_v u \cdot \nabla_v \vphi &\ge -a|\nabla_v u|^2 - a|\nabla_v\phi|^2,\\
    be^{-t}(|\nabla_xu|^2 + \nabla_x u \cdot \nabla_x\vphi) &\ge \frac{b}{2} |\nabla_xu|^2 - \frac{b}{2} |\nabla_x \vphi|^2,\\
    2c u(\beta(u)+\vphi) &\ge -2cM^2.
\end{align} 

Note that we have again used \eqref{eq: Linftyestimate2}. Therefore, it suffices to choose $b= 2a$ and $c=a$, and then $\nu >\displaystyle 2a M^2 + a\|\nabla_{x,v}\vphi\|_{L^\infty(Q_1)}^2$ in order to ensure $\mcL \Psi \ge 0$ on $\{\partial_{v_ev_e}u>0\} \cap Q_1$. In showing that $\mcL \Psi \ge 0$ on $\{\partial_{v_ev_e}u<0\}$, we will have to choose $a$ large. Before proceeding with that computation, we note that on $\{\partial_{v_ev_e}u<0\},$ one has
$$\mcL(\partial_{v_ev_e}^-u) = -\mcL(\partial_{v_ev_e}u) = -\left(\beta'(u)\partial_{v_ev_e}u + \beta''(u)(\partial_{v_e}u)^2 + 2\partial_{v_ex_e}u + \partial_{v_ev_e}\vphi \right).$$
In the hypoelliptic case, we must account for the commutator $[\mcL, \partial_{v_e}] = \partial_{x_e}.$

Proceeding as usual,
\begin{align} \label{eq: Lpsi3}
\mcL \Psi &= \mcL(\eta^2)(\partial_{v_ev_e}^-u)^2 + 8\eta \partial_{v_ev_e}^-u \nabla_v \eta \cdot \nabla_v \partial_{v_ev_e}^-u\\
&-2\eta^2\partial_{v_ev_e}^-u(\beta'(u)\partial_{v_ev_e}u + \beta''(u)(\partial_{v_e}u)^2 + 2\partial_{v_ex_e}u + \partial_{v_ev_e}\vphi)\\ & +2\eta^2|\nabla_v \partial_{v_ev_e}^-u|^2
+ 2a(\beta'(u)|\nabla_vu|^2 + |D^2_vu|^2 +\nabla_v u \cdot\nabla_xu + \nabla_v u \cdot \nabla_v \vphi)\\ &+ be^{-t}(2|\nabla_x u|^2\beta'(u) + 2|\nabla_v\nabla_x u|^2 + |\nabla_x u|^2 + \nabla_x u \cdot \nabla_x \vphi) + 2c( u(\beta(u) + \vphi) + |\nabla_vu|^2) + \nu.
\end{align}

Let us bound the first two lines on the righthand side of \eqref{eq: Lpsi3} as follows:
\begin{equation}
\label{eq: firsttwoconvexitbds1}
    \mcL(\eta^2)(\partial_{v_ev_e}^-u)^2 + 8\eta \partial_{v_ev_e}^-u \nabla_v \eta \cdot \nabla_v \partial_{v_ev_e}^-u \ge -C |D^2_vu|^2 + \eta^2 |D^3_v u|^2,
\end{equation}
and
\begin{align}
    \label{eq: firsttwoconvexitbds2}
    -2\eta^2\partial_{v_ev_e}^-u(\beta'(u)\partial_{v_ev_e}u + \beta''(u)(\partial_{v_e}u)^2 + 2\partial_{v_ex_e}u + \partial_{v_ev_e}\vphi) & \ge -C\eta^2(|D^2_vu|^2 +|\nabla_v\nabla_x u|^2 + |D^2_v \vphi|^2)\\&-2\eta^2\partial_{v_ev_e}^-u(\beta'(u)\partial_{v_ev_e}u + \beta''(u)(\partial_{v_e}u)^2).
\end{align}
Crucially, the parenthetical term $\beta'(u) \partial_{v_ev_e}u + \beta''(u) (\partial_{v_e}u)^2$ is negative thanks to the facts that $\partial_{v_ev_e}u<0$ at the point of evaluation, $\beta'\ge0$, and $\beta'' \le 0$. We deduce then that $$-2\eta^2\partial_{v_ev_e}^-u(\beta'(u)\partial_{v_ev_e}u + \beta''(u) (\partial_{v_e}u)^2) \ge 0,$$ and therefore may be discarded. To proceed, let us bound
\begin{align}
\label{eq: secondtwoconvexitybds}
    be^{-t} \nabla_x  u \cdot \nabla_x \vphi &\ge -\frac{b}{2}e^{-t}|\nabla_x u|^2 - C|\nabla_x\vphi|^2,\\
    2a\nabla_v u \cdot \nabla_v \vphi &\ge -a|\nabla_v u|^2 - a|\nabla_v\vphi|^2.
\end{align}
Employing our observations, throwing away the remaining extraneous terms, and combining \eqref{eq: Lpsi3}-\eqref{eq: secondtwoconvexitybds}, we find 
\begin{align}
\mcL \Psi &\ge (2a - C)|D^2_vu|^2 + (2be^{-t} - C)|\nabla_v\nabla_xu|^2\\
&+ \left(\frac{be^{-t}}{2}-a\right)|\nabla_xu|^2 + (2c-a)|\nabla_vu|^2 + \nu -2cM^2 - (C+a)\|\vphi\|_{C^2(Q_1)}^2 .
\end{align}

Here, as always, $C = C_\eta$ has been chosen large enough to satisfy each of the above inequalities. Taking $a = C= c$ and $b =2a$, and then finally $\nu > 2aM^2 +2a \|\vphi\|_{C^2(Q_1)}^2$, we conclude that $\Psi $ is a subsolution on $Q_{2/3}$. We finish the proof with the maximum principle, noting that by the previous steps, the quantities $|\nabla_xu|_{L^\infty(Q_{2/3})}$ and $|\nabla_v u|_{L^\infty(Q_{2/3})}$ are controlled by $C(\|\psi\|_{C^1_{t,x,v}(Q_1)},\|u\|_{L^\infty(Q_1)})$. 
\end{proof}
\begin{cor} Under the same conditions as the previous proposition, solution $u$ to \eqref{eq: penalizedprob} satisfies the full $C^{1,1}_v$ estimate $$\|D^2_vu\|_{L^\infty(Q_{1/2})} \le C(1+\|u\|_{L^\infty(Q_1)})$$
for some $C = C(M, \|\vphi\|_{C^2_{t,x,v}(Q_1)}).$
\end{cor}
\begin{proof} The proof uses ellipticity of the Kolmogorov operator in the velocity variable. Write $\Delta_v = \text{tr}(D^2_vu)^+ - \text{tr}(D^2_vu)^-$. The previous step gives an upper bound on $(D^2_vu)^-$, and we may write
$$ \text{tr}(D^2_vu)^+ = (\beta(u)+\vphi) +  \text{tr}(D^2_vu)^- + \partial_tu + v\cdot\nabla_x u \le M + \text{tr}(D^2_vu)^- + |\partial_tu| + |\nabla_xu|,$$
thanks to the fact that $|v| \le 1$. By the three previous propositions and lemmas, all of the terms on the righthand side are bounded on $Q_{1/2}$ by $C(1+\|u\|_{L^\infty(Q_1)})$ for $C = C(M, \|\vphi\|_{C^2_{t,x,v}(Q_1)})$. We therefore have an upper bound on the magnitude of every eigenvalue and may conclude. 
\end{proof}
Finally, we arrive at the proof of Theorem \ref{thm: optimalreg}.
\begin{proof}[Proof of Theorem \ref{thm: optimalreg}.]
For any $0<\alpha,\delta<1$, we may apply the preceding estimates to extract a subsequence $u^{\eps_j} \to u$ strongly in $(C^{1,\alpha}_v \cap C^{0,\alpha}_{t,x})(Q_{1-\delta})$ to a function $u$ satisfying $\|D^2_vu\|_{L^\infty(Q_{1-\delta})} + \|\nabla_{t,x,v}u\|_{L^\infty(Q_{1-\delta})} \le C(1+ \|u\|_{L^\infty(Q_1)})$ for some $C=C(\delta, M, \|\psi\|_{C^2_{t,x,v}(Q_1)}).$

Now, as in \eqref{eq: Linftyestimate2}, we recall that $u_{\eps_j} \ge -M\eps_j$, and therefore in the limit obtain that $u \ge 0$. Next, for any $z_0\in Q_{1-\delta}$ such that $u(z_0)= c>0$, one certainly has $u^{\eps_j}(z_0)>c/2$ for $\eps_j$ small enough. By the uniform estimates on $\|u^{\eps_j}\|_{C^{1,1}_\ell(Q_{1-\delta})}$, there exists a uniform $\rho>0$ such that $u^{\eps_j} > c/4$ on $Q_{\rho}(z_0)$ for all $\eps_j$ small. Therefore, on $Q_{\rho}(z_0)$, there holds
$$\mcL u^{\eps_j} = \vphi\text{ for } \eps_j \text{ small}.$$ 

Since the righthand side $\vphi \in C^2_{t,x,v}(Q_1) \subset C^{0,\alpha}_\ell(Q_1)$, we may apply the Schauder estimate of Theorem \ref{thm: Schauderthm} to extract a further subsequence $u^{\eps_{j_k}} \to u$ in $C^{2,\alpha}_{\ell}(Q_{\rho/2}(z_0))$ for any $0<\alpha<1$. This allows us to conclude that $\mcL u (z_0)=\vphi(z_0)$, and therefore that $\mcL u = \vphi$ on $\{u>0\}$. Finally, we note that $\mcL u^\eps \le \vphi$ for all $\eps$, implying $\mcL u \le \vphi$.  Setting $f = u+\psi$, so that $\{f>\psi\} = \{u>0\}$, we obtain that $f$ solves
$$
\begin{cases}
    f \ge \psi \text{ and }  \mcL f \le 0 , & \text{ in } Q_1,\\
    \mcL f = 0 & \text{ in } Q_1 \cap \{f>\psi\},\\
    f = g & \text{ on } \partial_p Q_1
\end{cases}
$$
and obeys the estimate 
  $$\|D^2_v(f-\psi)\|_{L^\infty(Q_{1/2})} + \|\nabla_{t,x,v}(f-\psi)\|_{L^\infty(Q_{1/2})} \le C(1+\|f-\psi\|_{L^\infty(Q_1)}).$$
Here $C =  C(\|g\|_{L^\infty(\partial_p Q_1)}, \|\psi\|_{C^4_{t,x,v}(Q_1)})$ as claimed. By the uniqueness theorem of \cite{PolidoroEXISTENCE} we conclude, since $f$ is the only such solution to \eqref{eq: formulation1}. The estimates on $u$ in the second formulation \eqref{eq: formulation2} can be obtained with marginally less difficulty by setting $-\mcL \psi =1 = \vphi$, meaning one need not contend with derivatives of $\vphi$.
\end{proof}
\begin{remark}
    There is no real loss of generality in assuming that $f$ is continuous on $\partial_p Q_1$. Indeed, by the results of \cite{PolidoroEXISTENCE} and \cite{FRENTZPOLIDOROOPTIMAL1}, the solution $f$ will be continuous and above the obstacle on the interior. Since our derivative estimates do not extend to the fixed boundary $\partial_p Q_1$, we could always just use that $f$ is continuous on $\partial_p Q_{1-\delta}$ and step inside. For this reason we will forget about the boundary data in the rest of the paper. 
\end{remark}
\begin{remark}
    As for the regularity assumption on the obstacle, we make a slightly stronger assumption on $\psi$ than that made in \cite{FRENTZPOLIDOROOPTIMAL1}. This is possibly just an artifact of the proof. For the purposes of free boundary regularity, it is customary to take $\psi$ smooth with $\mcL(\psi)=-1$ as a first approach. 
\end{remark}
\begin{remark} 
The regularity of Theorem \ref{thm: optimalreg} is likely optimal. That $C^{0,1}_t$ and $C^{1,1}_v$ are optimal is known from the study of the parabolic problem. We can at least argue that $\partial_t u$ and $\nabla_xu$ cannot both be continuous up to the entire free boundary, unless the free boundary consists entirely of regular points as defined in Section \ref{blowups}. We show there that in general, $\Gamma_u = \Gamma_u^{\text{reg}} \sqcup \Gamma_u^{\text{sing}}$, where, in particular, $\Gamma_u^{\text{sing}}$ consists of those points $z_0 \in \Gamma_u$ for which the rescalings $u_r^{z_0}$ converge as $r\to 0$ to a quadratic polynomial $u_0(t,v) = mt + Av \cdot v$, for some $m\in \R$ and $A\in \mathbb{R}^{n\times n}$. Therefore, at singular points with $m \neq 0$, it cannot be that $\partial_t u_r^{z_0}(z) = Yu(z_0 \circ S_r(z))$ converges to $0$ as $z_0 \circ S_r(z) \to z_0 \in \Gamma$. In other words, at least one of $\partial_tu$ or $\nabla_xu$ fail to vanish continuously at singular points. We suspect that both derivatives will have jumps across $\Gamma_u^{\text{sing}}$, but for the reasons outlined in the introduction, it is not easy to work with these derivatives directly. 
\end{remark}
\section{Nondegeneracy}
\label{nondegensec}
From this point forward, we work with solutions $u$ to \eqref{eq: formulation2} that have continuous boundary data and are such that $\|u\|_{L^\infty(Q_1)} \le M$ for some $M>0$. Then Theorem \ref{thm: optimalreg} implies that 
$$\|\nabla_x u\|_{L^\infty(Q_{1/2})} + \|\partial_tu\|_{L^\infty(Q_{1/2})} + \|D^2_vu\|_{L^\infty(Q_{1/2})} \le CM$$
for some dimensional constant $C= C(n)$. As remarked in the previous section, there is no real loss in generality in assuming continuity on the boundary. Any boundary conditions which can ensure interior regularity will necessarily ensure continuity, at which point we may apply our interior regularity estimates by stepping inside to $Q_{1-\delta}$. Since our estimate is an interior estimate, we make no further reference to the boundary data, and occasionally use $h$ and $\vphi$ to denote functions that have nothing to do with the initial formulations \eqref{eq: formulation1} and \eqref{eq: formulation2}.

While some results are easily transferable to the classical formulation \eqref{eq: formulation1}, for the purpose of studying free boundary regularity it is preferable to work with the second formulation. We begin with the usual nondegeneracy of solutions, which says that solutions grow quadratically on kinetic cylinders. This is an important fact that ensures that the property of being on the free boundary is preserved by blow-up limits. 
\begin{prop}\label{prop: nondegen} Let $u$ be a solution of \eqref{eq: formulation2} and let $z_0 \in \Gamma_u$. Then there exists an absolute constant $c_n>0$ depending only on dimension such that for any $r>0$ such that $Q_r(z_0) \subset Q_1$ there holds
\begin{equation} \label{eq: nondegen} \sup_{Q_r(z_0)} u \ge c_n r^2.\end{equation}
In fact, for any $z_1 \in \Omega_u$ such that $Q_r(z_1) \subset Q_1$, one has
\begin{equation}
   \label{eq: nondegenpositive} u(z_1) +c_nr^2 \le \sup_{Q_r(z_1)} u
\end{equation}

\end{prop}
\begin{proof} Let $r>0$ be arbitrary and set $\vphi^r(t,x,v) \coloneqq \ds  \frac{1}{4n+2}(|v|^2 -t) + \frac{|x|^2}{4r^4}$. Now, take $z_1 \in \{u>0\} \cap Q_1= \Omega_u \cap Q_1$ arbitrary and set $\tilde{\vphi}^r(z) \coloneqq \vphi^r(z_1^{-1} \circ z)$. By Galilean invariance, we compute
$$\mcL \tilde{\vphi}^r(z) = (\mcL\vphi^r)(z_1^{-1} \circ z) = \frac12 - \frac{(v-v_1)\cdot(x-x_1-(t-t_1)v_1)}{2r^4}.$$

Since $|v-v_1|<r$ and $|x-x_1-(t-t_1)v_1|<r^3$ for $z=(t,x,v)\in Q_r(z_1)$, we find that $0 \le \mathcal{L}(\tilde{\vphi}^r)\le 1$ there. Now, set $f(z) \coloneqq u(z)  - \tilde{\vphi}(z)- u(z_1)$. On $\Omega_u \cap Q_r(z_1)$ we have that $\mcL f = 1- \mcL \tilde{\vphi}^r \ge 0$. It is clear by construction that $f(z_1)=0$. Therefore, the maximum principle implies that $f$ takes its nonnegative maximum on $\partial_p(\Omega_u \cap Q_r(z_1))$. However, $f<0$ on $\Gamma_u \cap Q_r(z_1)$, because $u$ vanishes there and $\tilde{\vphi}^r$ and $u(z_1)$ are positive there. This implies that $f$ takes its maximum on $\Omega_u \cap \partial_pQ_r(z_1)$. Note that on $\partial_pQ_r(z_1)$, the function $\tilde{\vphi}^r$ satisfies  $\displaystyle \tilde{\vphi}^r \ge \frac{1}{4n+2} r^2$, and therefore we obtain
$$0 \le \sup_{\partial_p Q_r(z_1)} u- \tilde{\vphi}^r - u(z_1) \le \sup_{\partial_p Q_r(z_1)} u - \frac{1}{4n+2}r^2 - u(z_1)$$
and we rearrange to find that \eqref{eq: nondegenpositive} holds with $c_n = \ds \frac{1}{4n+2}.$

Sending $z_1 \to z_0 \in \Gamma_u $ we arrive at \eqref{eq: nondegen}. Note that, since we are sending $z_1 \to z_0$, we may as well assume that we are working with a point $z_1 \in \{u>0\}$ such that $Q_r(z_1) \subset Q_1$. Else we may replace $r$ by an arbitrary $0<\rho<r$ for $z_1$ close enough to $z_0$, and then send $\rho \to r$. 
\end{proof}

We will also use the inequality \eqref{eq: nondegenpositive} once later on. The previous proof differs in a small way from the usual proofs of nondegeneracy (see, e.g., \cite{OBSPROBREVISITED}, \cite{PSUBOOK}, \cite{CPS}) in that we use an auxiliary function $\vphi^r$ that depends on $r$. It is not clear whether there is an explicit function $\vphi$ such that $\mcL \vphi = 1$, $\vphi(0,0,0)=0$, and $\displaystyle \inf_{\partial_p Q_r} \vphi \gtrsim r^2.$ 

As a consequence of the nondegeneracy, we obtain the following result on the ``porosity" of the free boundary.
\begin{prop} \label{prop: porosity}
    Let $u$ solve \eqref{eq: formulation2} with $\|u\|_{L^\infty(Q_1)} \le M$, and let $z_0 \in \Gamma_u \cap Q_{1/2}$ be arbitrary. There exists $\mu = \mu(n, M)$ such that 
    \begin{equation}
    \label{eq: porosity}
        |Q_r(z_0)\cap \Omega_u| =  |Q_r(z_0) \cap \{u>0\}| > \mu |Q_r|  \ \text{ for all } r>0 \text{ sufficiently small}.
    \end{equation}
    In particular, $Q_r(z_0) \cap \Omega_u \supset Q_{\delta r}(z_1)$ for some $\delta = \delta(n,M) \in (0,1)$. 
\end{prop}
\begin{proof}
    The proof essentially follows from the $C^{1,1}_{\ell}$ regularity of $u$. Let $z_1\in Q_r(z_0)$ be a point such that $u(z_1) \ge cr^2$, as per the nondegeneracy lemma. By the regularity estimates of Theorem \ref{thm: optimalreg}, $u$ separates from its first order Taylor polynomial $p_1(z; z_1) = u(z_1) + (v-v_1) \cdot \nabla_v u(z_1)$ quadratically, that is,  there exists an absolute $C>0$ such that for all $0<\rho<1$,
    $$\|u-p_1\|_{L^\infty(Q_\rho(z_1))} \le CM \rho^2.$$ 
    Taking $C>0$ even larger if need be and applying the Lipschitz regularity estimate of $\nabla_v u$  from Theorem \ref{thm: optimalreg}, as well as the fact that $\nabla_v u(z_0)=0$ continuously, it holds that $|\nabla_v u(z_1)| \le CMr$ for some absolute constant $C$. Now, take $\rho = \delta r$ for $0<\delta = \delta(n,M)<1$ such that $\displaystyle CM\delta(1-\delta) = \frac{c}{2}$. Then there holds 
    $$u(z) \ge u(z_1) -CMr^2 \rho - CM\rho^2= r^2(c - CM \delta - CM \delta^2)\ge \frac{c}{2}r^2 > 0 \text{ for } z \in Q_{\delta r}(z_1).$$
\end{proof}
Let us recall the following version of the Lebesgue differentiation theorem for kinetic cylinders, from \cite{IMBERTSILVESTRE}.
\begin{thm}[Theorem $10.3$ from \cite{IMBERTSILVESTRE}]
\label{thm: luislebesgue}
Let $\Omega \in \R^{1+2n}$ be an open set, and $f \in L^1(\Omega)$ with respect to the $2n+1$-dimensional Lebesgue measure $\ dt \otimes \ dx \otimes \ dv.$ Then 
$$\lim_{r \to 0 } \frac{1}{|Q_r|}\int_{Q_r(t_0,x_0,v_0)} |f(t,x,v) - f(t_0,x_0,v_0)| \ dt \ dx \ dv=0$$
for almost every $(t_0,x_0,v_0) \in \Omega$.
\end{thm}
As a corollary, we obtain that the free boundary $\Gamma_u$ is of measure zero. 
\begin{lem}
\label{lem: fbmeasurezero}
    Let $u$ solve \eqref{eq: formulation2} with $\|u\|_{L^\infty(Q_1)} \le M$. Then $|\Gamma_u| =0$.
\end{lem}
\begin{proof}
    We apply Theorem \ref{thm: luislebesgue} and Proposition \ref{prop: porosity} to $f = \chi_{\Gamma_u}$ to find that 
    $$\chi_{\Gamma_u}(t_0,x_0,v_0) = \lim_{r \to 0} \frac{|Q_r(t_0,x_0,v_0) \cap \Gamma_u|}{|Q_r(t_0,x_0,v_0)|}<1-\mu$$
    for almost every $(t_0,x_0,v_0) \in Q_1$, where here $\mu = \mu(n,M)$ is as in \eqref{eq: porosity}. Hence $\chi_{\Gamma_u} =0$ almost everywhere and therefore $\Gamma_u$ is measure zero. 
\end{proof}
As a corollary, we obtain the following fact which will be very useful in our study of blow-ups.
\begin{cor}
\label{cor: blowupssolveq}

    For any open set $\mathcal{U} \subset \R^{1+2n}$ and $f: \mathcal{U} \to \R$ such that
    \begin{itemize}
        \item $f \ge 0$,
        \item $f \in C^{1,1}_{\ell}(\mathcal{K})$ on any compact subset $\mathcal{K} \subset \subset \mathcal{U}$, 
        \item $\mcL f =1$ on $\{z \in \mathcal{U} \ : \ f(z)>0\}$, 
    \end{itemize}
    one obtains that $f$ solves the obstacle problem
    $$\mcL f = \chi_{f>0} \text{ almost everywhere on } \mathcal{U}.$$
\end{cor}
\begin{proof}
    The first and third property are enough to ensure the nondegeneracy  property \eqref{eq: nondegen}, and then the second will imply the porosity result of \eqref{eq: porosity} for some $\mu>0$ depending on $f$ and $\mathcal{K}$. Together these are enough to say that $|\partial_p\{f>0\}| =0$. Evidently, $\mcL f =0$ on the interior of $\{f=0\}$ and the result follows. Noting that the regularity on $f$ ensures that $\mcL f \in L^\infty$, we conclude that $\mcL = \chi_{f>0}$ almost everywhere. 
\end{proof}
Finally, although we do not find an application for it, we record that solutions also enjoy parabolic nondegeneracy as a consequence of the estimate on $\nabla_xu$ from Theorem \ref{thm: optimalreg}. We will see later that, near free boundary points where the contact set is sufficently thick, solutions actually enjoy elliptic nondegeneracy. Still, the following proposition is interesting in that nondegeneracy on certain cross-sections of the entire free boundary $\Gamma_u$ does not usually occur in obstacle problems without further assumptions. For instance, in the parabolic problem $\mathcal{H}u = \chi_{u>0}$, one would require $\partial_t u \ge 0$ in order to get elliptic nondegeneracy along $t-$sections of the free boundary. Here, in our setting, we have quadratic growth on parabolic cylinders centered at \textit{every} free boundary point, making no further assumptions on $u$. We use the notation $\mathcal{Q}_r'(t_0,v_0) = \{(t,v) \ : \ t_0-r^2 <t \le t_0, |v-v_0|<r\}$ to denote the usual parabolic cylinder of radius $r>0$ centered at $(t_0,v_0)$. 
\begin{prop} \label{prop: parabolicnondegen} Let $u$ solve \eqref{eq: formulation2} with $\|u\|_{L^\infty(Q_1)} \le M$. There exists a uniform $r_0 =r_0(M)>0$ such that for any $z_0 = (t_0,x_0,v_0) \in \Gamma_u \cap Q_{1/2}$, there holds
$$\sup_{\mathcal{Q}_\delta'(t_0,v)} u(t,x_0,v_0) \ge c_n \delta^2$$
for all $0\le \delta \le r_0$. 
\end{prop}
\begin{proof} Without loss of generality consider $z_0=0$ and consider the rescaling $u_r(t,x,v) \coloneqq \ds \frac{u(r^2t, r^3x, rv)}{r^2}$ which solves $(\Delta_v - \partial_t)u_r = \chi_{u_r>0} + v\cdot \nabla_x u_r$. By the Lipschitz estimate of Theorem \ref{thm: optimalreg}, there exists an $r_0 = r_0(M)$ such that for any $0<r\le r_0$, one has 
$$(\Delta_v - \partial_t)u_r  \ge \frac12 \text{ on } \Omega_{u_r} \cap Q_1.$$ 

Now, since $(0,0,0) \in \Gamma_{u_{r_0}}$, for any $\eta>0$ there exists $z_\eta =(t_\eta,x_\eta,v_\eta)$ such that $u_{r_0}(z_\eta)>0$. On the parabolic cylinder $\mathcal{Q}_{1/2}'(0, 0)$, we define the function $f:\mathcal{Q}_{1/2}' \to \R$ by
$$f(t,v) \coloneqq u_{r_0}(t_\eta +t, x_\eta, v_\eta + v) - u_{r_0}(t_\eta, x_\eta, v_\eta) - \frac{1}{4n+2}((|v|^2 -t).$$
Following along the lines of Proposition \ref{prop: nondegen} we observe that
$$\mathcal{H} f = (\Delta_v - \partial_t)f \ge 0 \text{ on } \{(t,v) \in \mathcal{Q}_{1/2}'\ : \ u_{r_0}(t+t_\delta, x_\delta, v+v_\delta)>0\}.$$

Moreover, $f(0,0)=0$ and $f\le 0$ on $\{(t,v) \in \mathcal{Q}_{1/2}' \ : \ u_{r_0}(t+t_\eta, x_\eta, v+v_\eta)=0\}.$ Consequently, for any $0<\lambda<\displaystyle \frac12$, the parabolic maximum principle implies
$$0 \le \max_{\{u_{r_0}(t_\eta + \cdot, x_\eta, v_\eta + \cdot)>0\} \cap \mathcal{Q}_\lambda'} f (t,v) \le\max_{\{u_{r_0}(t_\eta + \cdot, x_\eta, v_\eta + \cdot)>0\} \cap \partial_p \mathcal{Q}_\lambda'} f(t,v).$$

Therefore $\displaystyle \max_{\mathcal{Q}_\lambda'(t_\eta, x_\eta, v_\eta)}u_{r_0}(t,x_\eta,v) \ge \frac{1}{4n+2} \lambda^2,$ and the result follows by sending $\eta \to 0$ and $z_\eta \to (0,0,0)$, and then undoing the scaling. 
\end{proof}
\section{A Monotonicity Formula}
\label{monotonicitysec}
The Lipschitz estimate on $\nabla_x u$ from Theorem \ref{thm: optimalreg} allows us to think of solutions $u$ as being almost solutions to the parabolic obstacle problem $\mathcal{H} u \approx \chi_{u>0}$ near free boundary points. Indeed, by rescaling around $(0,0,0) \in \Gamma_u$ with $u_r(t,x,v) \coloneqq r^{-2}u(r^2t, r^3x, rv)$, there holds $\mathcal{H}u_r = \chi_{u_r>0} + v\cdot \nabla_x u_r$. The error term eventually satisfies $|v\cdot \nabla_x u_r| \lesssim r$ on any given compact subset. This suggests that, in some sense, the rescalings $u_r$ should converge to a global solution $u_0=u_0(t,v)$ of the parabolic obstacle problem. We shall explore this matter more in the next section. In this section we focus on proving an almost monotonicity formula for $u$, from which we will derive a certain homogeneity for its blow-up limits. We emphasize again that our computations work only because of the spatial Lipschitz estimate.

In this section, we assume that $u \in \mathcal{P}_2(M)$, so, in particular, that the origin is a free boundary point. Let us introduce the Weiss functional, defined for smooth functions $f$ with polynomial growth by the formula
\begin{equation}\label{eq: Iformula} \mathcal{I}(r,f) \coloneqq \frac{1}{r^4}\int_{-4r^2}^{-r^2} \int_{\R^{2n}} \left(|\nabla_v f|^2 + 2f + \frac{f^2}{t}\right) K(t,x,v) \ dv \ dx \ dt
\end{equation}
where $K: (-\infty,0)\times \R^n \times \R^n$ is given by
$$\displaystyle K(t,x,v) \coloneqq \frac{C_n}{(t^2)^{n}} e^{\frac{|v|^2}{t} + \frac{|x|^2}{t^3}}$$
where $C_n = \pi^{-n/2}$. 

The kernel $K$ is not exactly the fundamental solution for the Kolmogorov equation $\mcL K= 0$, but it does bear a resemblance. The functional $\mathcal{I}$ is a modification of the Weiss functional analyzed in \cite{CPS} for the parabolic problem. The motivation, as suggested in the beginning, is to discount points far from the origin and think of $u_r$ as a solution to the parabolic problem plus a small error. 

It is clear that under the usual kinetic rescaling $f_r(t,x,v) = r^{-2}f(r^2t, r^3x, rv)$, there holds $I(r,f) = I(1, f_r)$, and therefore $I'(r,f) = \frac{d}{dr}I(1,f_r)$. It is convenient for us to introduce the notation 
$$\dot{f}_r \coloneqq \frac{d}{dr} f_r = \frac{v\cdot \nabla_v f_r +2t\partial_t f_r + 3x\cdot \nabla_x f_r - 2f_r}{r}.$$

A function $f: (-\infty,0) \times \R^{2n}\to \R$ is said to be homogeneous of kinetic degree $\kappa$ if $f(S_r(t,x,v)) = f(r^2t, r^3x, rv) = r^\kappa f(t,x,v)$ for all $(t,x,v) \in (-\infty,0) \times \R^{2n}$. It is easy to see for instance that the kernel $K$ is homogeneous of kinetic degree $-4n$. In the next section, we will prove that the rescalings $u_r$ converge to global solutions which are homogeneous of kinetic degree $2$. Finally, it is convenient to also introduce the notation $$L_\kappa f \coloneqq v \cdot \nabla_vf + 2t \partial_tf + 3x \cdot \nabla_x f -\kappa f.$$ In the standard way, one sees that function $f$ is homogeneous of kinetic degree $\kappa$ if and only if $L_\kappa  \equiv 0$. Note that  in the case $\kappa =2$, we have $\displaystyle \frac1r L_2f_r = \dot{f}_r$. 

\begin{prop}
\label{prop: monotonicityformula}
    Let $u \in \mathcal{P}_2(M)$, and let $\psi = \psi(x,v) \in C_c^\infty(B_1 \times B_1)$ be a nonnegative smooth cut-off which satisfies $\psi \equiv 1$ on $B_{1/2} \times B_{1/2}$. With $f \coloneqq u\psi$, there exists a constant $C(n,M)$ such that
    $$\mathcal{I}(r,f) + Cr^2 \text{ is monotone non-decreasing for } 0<r<1.$$
\end{prop}
\begin{proof}
    First, observe that
    $$f_r(t,x,v)  = \frac{f(r^2t, r^3x,rv)}{r^2} = u_r(t,x,v) \psi(r^3x,rv) = u_r(t,x,v)$$ on the set $\left\{|x|< \frac{1}{2r^3}, |v|< \frac{1}{2r}\right\}.$ Computing straightforwardly,
    
    \begin{align} \frac{d}{dr}\mathcal{I}(1, f_r) &= 2\int_{-4}^{-1} \int_{\R^{2n}} \left(\nabla_v f_r \cdot \nabla_v \dot{f}_r + \dot{f}_r + \frac{f_r}{t} \dot{f}_r\right)K \ dv \ dx \ dt\\
    & = 2\int_{-4}^{-1} \int_{\R^{2n}}\dot{f}_r\left(-\Delta_v f_r -\frac{v}{2t}\cdot \nabla_v f_r +1 + \frac{f_r}{t}\right) K \ dv \ dx \ dt\\
    &\coloneqq I_r + J_r,
    \end{align}
where 
$$I_r \coloneqq 2\int_{-4}^{-1} \int_{B_{\frac{1}{2r^3}} \times B_{\frac{1}{2r}}}\dot{f}_r\left(-\Delta_v f_r -\frac{v}{2t}\cdot \nabla_v f_r +1 + \frac{f_r}{t}\right) K \ dv \ dx \ dt$$
and
$$J_r  \coloneqq 2\int_{-4}^{-1} \int_{(B_{\frac{1}{2r^3}} \times B_{\frac{1}{2r}})^c}\dot{f}_r\left( -\Delta_v f_r -\frac{v}{2t}\cdot \nabla_v f_r +1 + \frac{f_r}{t}\right) K \ dv \ dx \ dt$$

Note that all we have done is integrated by parts on the term $\int_{\R^n} (\nabla_v f_r \cdot \nabla_v \dot{f}_r)K \ dv.$ As noted in the beginning, on the domain of integration for the term $I_r$, we simply have that $f_r = u_r$. On that domain, we can use the complementarity condition 
$$\dot{u}_r(-\Delta_v u_r + 1) = -\dot{u}_r(-\partial_t u_r - v\cdot \nabla_x u_r),$$
since $(\Delta_v - \partial_t - v\cdot \nabla_x)u_r = 1$ on the set $\Omega_{u_r} = \{u_r>0\}$, and $\dot{u}_r=0$ almost everywhere on the set $\Lambda_{u_r}=\{u_r=0\}$. Here we are also using that $|\Gamma_{u_r}|=0$, thanks to Lemma \ref{lem: fbmeasurezero}. Continuing,
\begin{align}
    I_r &=2\int_{-4}^{-1} \int_{B_{\frac{1}{2r^3}} \times B_{\frac{1}{2r}}} \dot{u}_r\left(-\partial_tu_r- v\cdot \nabla_x u_r -\frac{v}{2t}\cdot \nabla_v u_r  + \frac{u_r}{t}\right) K \ dv \ dx \ dt \\
    &=\int_{-4}^{-1} \int_{B_{\frac{1}{2r^3}} \times B_{\frac{1}{2r}}} \dot{u}_r\left(2t\partial_tu_r + 2tv\cdot \nabla_x u_r +v\cdot \nabla_v u_r  -2 u_r\right) \frac{K}{-t} \ dv \ dx \ dt\\
    &= \int_{-4}^{-1} \int_{B_{\frac{1}{2r^3}} \times B_{\frac{1}{2r}}}\dot{u}_r\left(r\dot{u}_r + [2tv-3x]\cdot \nabla_x u_r\right)\frac{K}{-t} \ dv \ dx \ dt\\
    &=\int_{-4}^{-1} \int_{B_{\frac{1}{2r^3}} \times B_{\frac{1}{2r}}}r (\dot{u}_r)^2\frac{K}{-t} \ dv \ dx \ dt  + \int_{-4}^{-1} \int_{B_{\frac{1}{2r^3}} \times B_{\frac{1}{2r}}}\dot{u}_r\left( [2tv - 3x]\cdot \nabla_x u_r\right)\frac{K}{-t}\ dv \ dx \ dt 
\end{align}
For the latter term we can write 
$$\dot{u}_r [2tv-3x]\cdot \nabla_x u_r = \sqrt{r} \dot{u}_r \frac{[2tv-3x]\cdot \nabla_x u_r}{\sqrt{r}} \ge -\frac12r( \dot{u}_r)^2 - \frac{|2tv-3x|^2|\nabla_xu_r|^2}{2r}.$$

Consequently, and using that $|\nabla_x u_r| \le CMr$ on the domain of integration, we find
$$I_r \ge -C Mr \int_{-4}^{-1} \int_{\R^{2n}} |2tv-3x|^2 \frac{K(t,x,v)}{-t} \ dv \ dx \ dt \ge - CMr$$
for some dimensional constant $C$. It is crucial here that we have Theorem \ref{thm: optimalreg}. Continuing with $J_r$, we are going to change variables back to unit scale. Using that $\displaystyle \dot{f}_r = \frac1r L_2f_r$ and that $K$ is homogeneous of kinetic degree $-4n$, we find
$$J_r = \frac{2}{r^3} \int_{-4r^2}^{-r^2} \int_{(B_1\times B_1) \setminus (B_{1/2}\times B_{1/2})} L_2f\left(-\Delta_v f -\frac{v}{2t}\cdot \nabla_v f + 1 + \frac{f}{t}\right) K \ dv \ dx \ dt.$$

Here we used that $\psi$ vanishes off $B_1\times B_1$. It is clear that for $-4r^2<t<-r^2$ and $(x,v) \in (B_1\times B_1)\setminus(B_{1/2}\times B_{1/2})$, one may use Theorem \ref{thm: optimalreg} to bound
$$ L_2f\left(-\Delta_v f -\frac{v}{2t}\cdot \nabla_v f + 1 + \frac{f}{t}\right)  \ge - \frac{CM}{r^2}$$
for some $C =C(\psi,n,M)$. And on the same domain, one may bound as well
$$K(t,x,v) \gtrsim \frac{1}{r^{4n}}e^{\frac{-1}{256r^6}} e^{-\frac{1}{16r^2}}.$$

Since the domain of integration has volume comparable to $r^2$, we find that
$$J_r \ge - \frac{C}{r^{4n+3}} e^{\frac{-1}{256r^6} -\frac{1}{16r^2}}.$$

Since $\displaystyle \int_0^r s^{-4n-3} e^{\frac{-1}{256s^6}-\frac{1}{16s^2}} \ ds \le C_nr^2$ for $r \in (0,1)$, we conclude. Note that the regularity estimates on $u$ and the localization of $\psi$ make all of the computations justified. 
\end{proof}
We do not use much about the structure of $K$ in the above arguments, other than that $K$ has sufficient polynomial decay in all variables, satisfies $\displaystyle \nabla_v \log K = \frac{v}{2t}$, and has kinetic homogeneity of degree $-4n$. As is traditional, we will use this monotonicity formula to deduce a certain homogeneity for blow-ups of $u$ in the next section.

We conclude with the following useful corollary.
\begin{cor}
\label{cor: homogconst}
    Let $u_0 \in \mathcal{P}_\infty(CM)$. Then $u_0$ is homogeneous of kinetic degree $2$ if and only if $\mathcal{I}(u_0, \cdot)$ is constant in $r>0$. 
\end{cor}
\begin{proof}
    The manipulations of Proposition \ref{prop: monotonicityformula} show that
    $$\frac{d}{dr} \mathcal{I}(u, r) = \int_{-4}^{-1} \int_{\R^{2n}} r (\dot{u}_r)^2 \frac{K}{-t} \ dv \ dx \ dt.$$
    Therefore $\mathcal{I}(r,u)$ is constant if and only if $\dot{u}_r \equiv 0$ on $[-4,-1] \times \R^{2n}$, and consequently on $(-\infty,0] \times \R^{2n}.$
\end{proof}
\section{Blow-ups and the Free Boundary}
\label{blowups}
The purpose of this section is to use a blow-up procedure to classify the free boundary into a ``regular" set and a ``singular" set, first in terms of the behavior of $u_r^{z_0}$ for $r$ small. Roughly speaking, we initially define regular points as those points $z_0$ for which $u_r^{z_0} \approx \frac12(v\cdot e)_+^2$ for some direction $e \in \mathbb{S}^{n-1}$, and singular points as those for which $u_r^{z_0}\approx mt + Av\cdot v$ for suitable $m \in \R$ and symmetric $A \in \R^{n \times n}$. Using the estimate on $\nabla_xu$ and the monotonicity formula, we will see first that the blow-ups of $u$ at free boundary points coincide with the blow-ups in the parabolic problem, and characterize the free boundary according to blow-ups. Then, we give an alternative characterization of the free boundary points in terms of what we define as the \textit{balanced energy}, obtained from limits of the energy functional $\mathcal{I}$ from the last section. Finally, there is a third characterization through the Lebesgue density of the contact set $\Lambda_u$: regular points are those points $z_0 \in \Gamma_u$ for which the contact set $\Lambda_u = \{u=0\}$ takes up a non-negligible portion of small cylinders $Q_r(z_0)$, and the singular points are those for which the contact set is too sparse in the cylinders $Q_r(z_0)$ as $r\to 0$. All three of these characterizations are important. We conclude with a characterization of global solutions $u \in \mathcal{P}_\infty(CM)$.
\begin{prop}
\label{prop: blow-upindx}
    Let $u \in \mathcal{P}_1(M;z_0)$ solve \eqref{eq: formulation2} and consider the rescalings $$u_r^{z_0} = \frac{u(z_0 \circ S_r(z))}{r^2}.$$

    For any subsequence $r_k \to 0$, one may extract a further subsequence $u_{r_{k_j}}$
    $$u_{r_{k_j}} \to u_0 \text{ locally in } C^{1,\alpha}_{\ell}((-\infty,0]\times \R^n \times \R^n).$$
    
    The globally defined function $u_0$ is said to be a \textit{blow-up} of $u$ at $z_0$, and it satisfies the following properties:
    \begin{itemize}
        \item $\nabla_x u_0 \equiv 0$ on $(-\infty,0] \times \R^n \times \R^n$
        \item $u_0$ is a global solution to the parabolic obstacle problem
        \begin{equation}
            \label{eq: parabolicprob}
            (\Delta_v - \partial_t)u_0 = \mathcal{H}u_0 = \chi_{u_0>0} \text{ on } (-\infty,0] \times \R^n
        \end{equation}
        \item $0\le u_0(t,x,v) \le CM(1+|t| + |v|^2).$
        \item $z_0 \in \Gamma_{u_0}$.
    \end{itemize}. 

\end{prop}
\begin{proof}
The sequence $u_r^{z_0}$ solves 
\begin{equation}
    \label{eq: rescalingequation}
    \mcL(u_r^{z_0}) = \chi_{u_r^{z_0}>0}  \text{ on } Q_{1/r}(z_0^{-1}).
\end{equation}

Then the stated convergence along a subsequence of $u_{r_j}$ to a globally defined function $u_0$ follows by Arzela-Ascoli and the estimates of Theorem \ref{thm: optimalreg}. To prove that $u_0$ is spatially independent, we make use of the fact that $\nabla_x$ scales like $r^3$, while $u_r$ is on the order of $r^{-2}$. More precisely, given any points $z_1 = (t,x_1,v)$ and $ z_2= (t,x_2,v) \in (-\infty,0] \times \R^n \times \R^n$, for $r_j \coloneqq r_{k_j}$ small enough we have
    $$|u_{r_j}(z_1)- u_{r_j}(z_2)| \le CM r_j \to 0 \text{ as } r_j \to 0.$$

   In order to prove the second statement, we observe that $u_0$ inherits the $C^{0,1}_{t,x} \cap C^{1,1}_v$ and $C^{1,1}_\ell$ local regularity from sequence $u_{r_j}$. Then taking $z_0 \in \{u_0>0\}$ such that $u(z_0)=c>0$ and using the uniform estimates on $\|u_{r_j}\|_{C^{1,1}_\ell(Q_1(z_0))}$ for $r_j$ small enough, there exists a uniform $\rho>0$ proportional to $c^{1/2}$ such that $u_{r_j} > \displaystyle \frac{c}{2}$ on $Q_{\rho(z_0)}$, and therefore $\mcL u_{r_j}  =1 $ there. Passing to a further subsequence with the Schauder estimate of Theorem \ref{thm: Schauderthm}, we find $\mcL u_0(z_0)=1$. Therefore, using the regularity of $u_0$ and Corollary \ref{cor: blowupssolveq}, we arrive at the fact that $\mcL u_0= \chi_{u_0>0}$. Applying the fact that $u_0$ is independent of $x$, we conclude that $\mathcal{H} u_0 = \mcL u_0 = \chi_{u_0>0}$.
   
   The final two statements are simple consequences of nondegeneracy and quadratic growth. 
\end{proof}
\begin{remark}
    Any regularity estimate on $u \in C^{0,\beta}_x$ for $\beta>2/3$ would provide the result of the previous proposition. To classify blow-ups even further, we will make full use of the \textit{Lipschitz} estimate.
\end{remark}
We may characterize the blow-ups further as parabolically homogeneous of degree $2$, thanks to the monotonicity formula from Proposition \ref{prop: monotonicityformula}.
\begin{prop}
\label{prop: homog}
    Let $u \in \mathcal{P}_1(M)$ and suppose that $u_0$ is a blow-up of $u$ at the origin. Then $u_0$ is \textit{parabolically} homogeneous of degree $2$, in the sense that $u_0(r^2t,rv) = r^2u_0(t,v)$ for any $r>0$ and $(t,v) \in (-\infty,0) \times \R^n$.
\end{prop}
\begin{proof}
    The proof is standard. Suppose that $u_{r_j} \to u_0 = u_0(t,v)$ locally in $C^{1,\alpha}_\ell$ as $r_j \to 0$. Let $\psi$ be a cutoff function as in Proposition \ref{prop: monotonicityformula}, and set $f = u\psi$. Then, by construction, $f_{r_j} \to u_0$ locally in $C^{1,\alpha}_\ell$, since $\psi(r^3x, rv) \to 1$. The almost monotonicity formula from Proposition \ref{prop: monotonicityformula} implies that 
    $\lim_{r\to 0^+} \mathcal{I}(f,r) = \lim_{r\to 0^+} \mathcal{I}(1, f_r)$
    exists, since the correction term $Cr^2$ tends to zero as $r\to 0$. Necessarily one has $\mathcal{I}(u_0, 1) = \displaystyle \lim_{r\to 0^+} \mathcal{I}(f,r) = \lim_{r_j \to 0} \mathcal{I}(f, r_j)$. By scaling, for any $\rho>0$ there holds

    $$\mathcal{I}(u_0, \rho) = \lim_{r \to 0^+} \mathcal{I}(f_r, \rho) = \lim_{r\to 0^+} \mathcal{I}(f, r\rho) =\mathcal{I}(u_0, 1).$$

    We therefore deduce that $\mathcal{I}(u_0, \rho)$ is constant in $\rho$ and the claim follows by Corollary \ref{cor: homogconst} and \ref{prop: blow-upindx}.
\end{proof}
We are now in the position to classify the possible blow-ups of $u$, borrowing a result from \cite{CPS}. As in the parabolic case, we will see momentarily that every blow-up for our problem belongs to one of the following two classes.
\begin{defn}
\label{def: blow-ups}
    \begin{itemize}
    \item[]
        \item We say that $u_0\in \mathcal{P}^p_\infty(CM)$ is a \textit{half-space solution} of \eqref{eq: parabolicprob} if
        \begin{equation}
        \label{eq: halfspace} u_0(t,v) = \frac12(v \cdot e)_+^2 \end{equation}
       for some $e \in \mathbb{S}^{n-1}$.
        
        \item We say that $u_0\in \mathcal{P}^p_\infty(CM)$ is a \textit{polynomial solution} of \eqref{eq: parabolicprob} if 
        \begin{equation}
            \label{eq: polynomial}
       u_0(t,v) =mt + Av\cdot v
        \end{equation}
        for some constant $m \in \R$ and symmetric matrix $A$ with $\frac12\text{tr}(A) = m+1$.
    \end{itemize}
\end{defn}
For the global parabolic obstacle problem \eqref{eq: parabolicprob}, the following results were proven in \cite{CPS}. 
\begin{prop}
\label{prop: caffrecap}
    Let $f: \R^- \times \R^n$ be a global solution to the parabolic obstacle problem \eqref{eq: parabolicprob} belonging to $\mathcal{P}^p_\infty(CM)$. Then the following claims hold:
    \begin{itemize}
        \item $f$ is monotone decreasing in time and convex in $v$, that is,  
        $$\partial_t f(t,v) \le 0 \text{ and }D^2_v f(t,v) \ge 0 \text{ for all } (t,v) \in \R^- \times \R^n,$$
        \item $\mathcal{I}(r,f)$ is monotone non-decreasing in $r$,
        \item $f$ is parabolically homogeneous of degree $2$, in the sense that $f(r^2t, rv) = r^2f(t,v)$ for all $r>0$ and  $(t,v) \in \R^- \times \R^n$, if and only if $\mathcal{I}(r,f)$ is constant in $r$. For homogeneous solutions, the possible values of $\mathcal{I}(1,f)$ are $\omega$ and $2\omega$, where $\omega \coloneqq \frac{15}{4}$. The former occurs if and only if $f$ is a half-space of the form \eqref{eq: halfspace}, and the latter if and only if $f$ is a paraboloid of the form \eqref{eq: polynomial}.  
    \end{itemize}
\end{prop}
At this point, it is convenient to discuss the energy functional $\mathcal{I}$ centered at other free boundary points. Let $\psi = \psi(x,v) \in C_c^\infty(B_1 \times B_1)$ be any cutoff function satisfying $0\le \psi \le 1$ everywhere and $\psi \equiv 1$ on $B_{1/2} \times B_{1/2}$. Set $f \coloneqq u \psi$. Given $z_0 \in \Gamma_u \cap Q_{1/2}$, the almost monotonicity formula in Proposition \ref{prop: monotonicityformula} implies that there is a constant $C=C(n,M,\psi)$ such that the map $r \mapsto \mathcal{I}(1, f_r^{z_0}) + Cr^2$ is monotone increasing in $r$. Seeing as $\psi(x_0 + r^3x + r^2tv_0, v_0 + rv) \to 1$ as $r\to 0$, if $u_0$ is any blow-up of $u$ at $z_0$ along the sequence $r_j \to 0$, then the rescalings $f_{r_j}^{z_0} \to u_0$ and $\mathcal{I}(1,u_0) = \lim_{r\to 0^+} \mathcal{I}(1, f_r^{z_0}).$ In particular, this limit is independent of the choice of $\psi$. With $f$ as above, we define the energy at scale $r>0$ at $z_0$ by
$$\mathcal{E}(r,z_0;u) \coloneqq \mathcal{I}(1,f_r^{z_0}).$$
Our previous considerations allow us to define the \textit{balanced energy} of $u$ at $z_0$ by 
$$\mathcal{E}(z_0;u) = \lim_{r \to 0} \mathcal{E}(r,z_0;u).$$

When evaluating the energy at the origin, we often write $\mathcal{E}(r;u)$. Necessarily, $\mathcal{E}$ is upper semi-continuous on $\Gamma_u$ and takes values in $\{\omega, 2\omega\}$, since the sequence of continuous functions $\mathcal{E}_r(z_0)+Cr^2= \mathcal{E}(r,z_0;u)+Cr^2$ are monotone decreasing in $r$ as $r\searrow 0$,  with limit $\mathcal{E}(z_0) = \mathcal{I}(1,u_0) \in \{\omega, 2\omega\}$. Here $u_0$ is any blow-up of $u$ at $z_0$; the value is independent of the exact blow-up, which, a priori, need not be unique. We will see at the end of the next section, however, that blow-ups at regular points are unique. It is likely true in the singular case as well, but we do not pursue that question here.

Using the preceding sequence of propositions, we may give a classification of points on the free boundary $\Gamma_u$ for solutions $u$ to \eqref{eq: formulation2} in terms of the possible blow-ups.  
\begin{prop}
\label{prop: fbclassification} For $u$ solving \eqref{eq: formulation2}, we partition the free boundary $\Gamma_u$ into two disjoint sets, $\Gamma_u^{\text{reg}} \cup \Gamma_u^{\text{sing}}$, where   
$$\Gamma_u^{\text{reg}}\coloneqq \{z_0 \in \Gamma \ : \ \text{there exists a blow-up of $u$ at $z_0$ that is a half-space blow-up of the form \eqref{eq: halfspace}}\},$$
and $\Gamma_u^{\text{sing}} \coloneqq \Gamma_u \setminus \Gamma_u^{\text{reg}}$. Then every blow-up of $u$ at a point of $\Gamma_u^{\text{sing}}$ is a quadratic polynomial solution of the form \eqref{eq: polynomial}, and consequently every blow-up at $z_0 \in \Gamma_u^{\text{reg}}$ is a half-space of the form \eqref{eq: halfspace}.
\end{prop}
\begin{proof}
    The proof is standard. By Proposition \ref{prop: blow-upindx}, Proposition \ref{prop: homog}, and Proposition \ref{prop: caffrecap}, we find that blow-ups of $u$ at $z_0 \in \Gamma_u$ are either of the form \eqref{eq: halfspace} or \eqref{eq: polynomial}. This dichotomy partitions the free boundary, thanks to the monotonicity formula of Proposition \ref{prop: monotonicityformula}, which implies that if $u_0$ is a given blow-up at $(0,0,0) \in \Gamma_u$, then the full sequence $\mathcal{E}(r, z_0; u)$ converges to either $\omega$ or $2\omega$. This precludes the possibility of two blow-up sequences at a given free boundary point converging to global homogeneous solutions of a different form. 
\end{proof}
As has already been discussed, the balanced energy $\mathcal{E}$ is upper-semicontinuous on $\Gamma_u$ and takes values in $\{\omega, 2\omega\}$. Therefore, the set $\{z_0 \in \Gamma_u \ : \
 \mathcal{E}(z_0) = \omega\}$ is relatively open in $\Gamma_u$. We refer to such points as \textit{low energy points}, and their complement by \textit{high energy points}. The following corollary is immediate. 
\begin{cor}
    \label{cor: regularopen}
    Let $u\in \mathcal{P}_1(M)$ solve \eqref{eq: formulation2}. Then $\Gamma_u^{\text{reg}}$ is relatively open in $\Gamma_u$, and consists precisely of the \textit{low energy points}.
\end{cor}
There is actually a third characterization of the free boundary according to Lebesgue density. This characterization will be useful later on.
\begin{prop}
\label{prop: FBdensityclassification}
    Suppose that $u$ solves \eqref{eq: formulation2} and $z_0 \in \Gamma_u$ satisfies the density condition
    \begin{equation}
        \label{eq: lebesguedensitycondition}
    \lim_{r \to 0} \frac{|\Lambda_u \cap Q_r(z_0)|}{|Q_r(z_0)|}>0\end{equation}
    Then $z_0 \in \Gamma_u^{\text{reg}}$. On the other hand, if \eqref{eq: lebesguedensitycondition} fails, that is, if $$\lim_{r \to 0} \frac{|\Lambda_u \cap Q_r(z_0)|}{|Q_r(z_0)|} =0,$$
    then $z_0 \in \Gamma_u^{\text{sing}}.$

    Consequently, we have the three following equivalent characterizations of the regular free boundary points:
    \begin{align} \Gamma_u^{\text{reg}} &= \{z_0 \in \Gamma_u \ : \ \text{ there exists a half-space blow-up of the form \eqref{eq: halfspace}}\}\\ 
    &= \{z_0 \in \Gamma_u \ : \ \mathcal{E}(z_0) = \omega\}\\
    &= \left\{z_0 \in \Gamma_u \ : \ \limsup_{r\to 0} \frac{|\Lambda_u \cap Q_r(z_0)|}{|Q_r(z_0)|}>0\right\}.
    \end{align}
\end{prop}
\begin{proof}
    We begin by showing that free boundary points $z_0$ for which $\Lambda_u$ has positive Lebesgue density are regular points. Let $r_j\to 0$ be any sequence for which the above limit is positive, and let $u_0$ be any blow-up along the sequence $u_{r_j}^{z_0}$. Passing to a limit in the identity
    $$\frac{|\Lambda_{u_{r_j}^{z_0}} \cap Q_1|}{|Q_1|} = \frac{|\Lambda_u \cap Q_{r_j}(z_0)|}{|Q_{r_j}(z_0)|},$$
    and using that $\limsup \Lambda_{u_{r_j}^{z_0}} \subset \Lambda_{u_0}$, we obtain that 
    $$|\Lambda_{u_0} \cap Q_1|>0.$$
    Consequently, $u_0$ cannot be a polynomial solution of the form \eqref{eq: polynomial}, and must therefore be a half-space solution of the form \eqref{eq: halfspace}. The first claim follows from Proposition \ref{prop: fbclassification}.

    To prove the second claim, we follow \cite{XFRXRO}. By assumption, $|\Lambda_{u_r^{z_0}} \cap Q_1| \to 0$. Now, take any $\eta\in C_c^\infty(Q_1)$ vanishing near $\{t=0\}$, and let $u_0$ be any blow-up of $u$ at $z_0$, attained along the sequence $r_j \to 0$. Passing to a limit in the weak formulation of \eqref{eq: formulation2},
    $$\int_{Q_1} - \nabla_v u_{r_j}^{z_0} \cdot \nabla_v \eta + \int_{Q_1} u_{r_j}^{z_0}(\partial_t \eta + v\cdot \nabla_x \eta) = \int_{Q_1} \eta \chi_{u_{r_j}^{z_0}>0}$$
    we obtain that
    $$\int_{Q_1} - \nabla_v u_0 \cdot \nabla_v \eta + \int_{Q_1} u_0 \partial_t \eta = \int_{Q_1} \eta.$$
    Here we have used the assumption, as well as the fact that $u_0$ is independent of $x$. Therefore $\mathcal{H}u_0 = 1$ on $Q_1$. Thanks to Proposition \ref{prop: caffrecap}, we find that $\Lambda_{u_0(t)}$ is a closed, convex set for each $t \in (-1,0)$. If $\Lambda_{u_0(-1/k)}  \cap B_1 = \{v \in B_1 \ : \ u_0(-1/k, v)=0\} $ had nonempty interior in $\R^n$, then the homogeneity of $u_0$ would imply that $\Lambda_{u_0}$ contained an open paraboloid, contradicting the fact that $\mathcal{H}u_0=1$ on $Q_1$. Therefore $\Lambda_{u_0(-1/k)} \cap B_1$ is contained in an $n-1$ dimensional hyperplane $H_k$. Since $\partial_tu_0 \le 0$, we deduce that $\{(t,x,v) \in \Lambda_{u_0} \cap Q_1 \ : \ t \le -1/k\} \subset \{(t,x,v) \in [-1,-1/k] \times  B_1 \times  H_k\}$. We conclude that $|\Lambda_{u_0} \cap Q_1|=0$, and therefore $u_0$ must be a polynomial solution of the form \eqref{eq: polynomial}.
\end{proof}
We remark that it is possible to prove free boundary regularity through the density condition \eqref{eq: lebesguedensitycondition}. However, this is weaker than the condition imposed in Theorem \ref{thm: freeboundaryregularity}, for we only impose that the contact set is thick enough at a single scale. Indeed, if $\delta_r(z_0;u)$ is too small for all $r \to 0$, it is easy to see that \eqref{eq: lebesguedensitycondition} fails at $z_0$. Lastly, we end with the following classification of global solutions to \eqref{eq: formulation2}.
\begin{prop}
    \label{prop: globalclassification}
    Suppose that $u \in \mathcal{P}_R(MR^2)$. Then 
    $$\|\nabla_x u\|_{L^\infty(Q_{R/2})} \le \frac{CM}{R}.$$

    Consequently, if $u \in \mathcal{P}_\infty(CM)$, then on all of $\R^- \times \R^{2n}$, there holds
    \begin{enumerate}
        \item $\nabla_xu\equiv 0$
        \item $\partial_tu \le 0$,
        \item $D^2_v u \ge 0.$
    \end{enumerate}
\end{prop}
\begin{proof}
    The claim follows from noting that $u_R \in \mathcal{P}_1$, applying Theorem \ref{thm: optimalreg}, and then undoing the scaling. The classification of global solutions is then immediate from Proposition \ref{prop: caffrecap}.
\end{proof}
\section{Balanced Energy and Further Classification of Blow-Ups}
\label{balanced}
In this section, we continue our analysis of the free boundary. In particular, we explore some connections between the thickness function $\delta_r(u)$ and the balanced energy $\mathcal{E}$. We work with ``almost global" solutions $u \in \mathcal{P}_R(MR^2)$ for $R>0$ large, and assume that the contact set $\Lambda_u$ has positive thickness at unit scale. If $R$ is large enough, then $u$ may be approximated by a global solution $u_0$ to the parabolic obstacle problem \eqref{eq: parabolicprob} which is monotone in $v$ near the origin for a certain cone of directions. We then deduce that the same is true of our local solution $u \in \mathcal{P}_R(MR^2)$, which involves overcoming the difficulty posed by the derivatives $\mcL$ and $\partial_{v_e}$ failing to commute. The second critical observation is that the transport derivative $Yu$ vanishes continuously near the origin. This should be seen as analogous to the fact that $\partial_t u \to 0$ near regular points in the parabolic problem, see \cite{CPS}. In proving this we perform a standard blow-up contradiction argument based on the balanced energy $\mathcal{E}$. We emphasize here that all of these critical steps strongly rely on the estimate on $\nabla_xu$ from Theorem \ref{thm: optimalreg}.

As described in the introduction, we measure the ``thickness" of the contact set $\Lambda_u$ near free boundary points $z_0 \in \Gamma_u$ with $\delta_r(u, z_0)$ as defined in \eqref{eq: densityfunction}. By construction, we have the Galilean and scaling invariance property
$$\delta_r(u,z_0) = \delta_1(u_r^{z_0},(0,0,0)) \coloneqq \delta_1(u_r^{z_0}).$$
Note that when evaluating the density at the origin, we suppress the notation. There also holds the following stability property. If $u_j \in \mathcal{P}_1(M)$ is converging in $C^{1,\alpha}_\ell(Q_1)$ to some $u_0$, then $(0,0,0) \in \Gamma_{u_0}$, and
$$\limsup_{j \to \infty} \Lambda_{u_j} \subset \Lambda_{u_0}.$$
As a consequence,
$$\delta_1(u_j) \ge \sigma \text{ for all } j \text{ implies }\delta_1(u_0) \ge \sigma.$$ 

For posterity, we record that the minimal diameter of a set $E \subset \R^{n}$ is given by 

$$\text{m.d.}(E) =\inf_{\xi \in \mathbb{S}^{n-1}} \osc_{v \in E} v \cdot \xi.$$
Let us also remark that, in the case of a function $f: (-\infty,0] \times \R^n \times \R^n \to [0,\infty)$ which is independent of $x$, there holds
$$\delta_r(f,(t_0,x_0,v_0)) =  \frac{\text{m.d.}(\Lambda_{f(t_0-r^2)} \cap B_r(v_0))}{r} \coloneqq \delta_r^p(f,(t_0,v_0)), $$
where $\delta_r^p$ is the thickness function considered in \cite{CPS}.

\begin{lem} \label{lem: approxlemma} Given $M, \sigma, \eps>0$, there exists $C=C(n)>0$ and $R_0 = R_0(n,\eps, M,\sigma)$ such that if $ u \in \mathcal{P}_R(MR^2)$ for $R \ge R_0$, and $\delta_1(u) \ge \sigma$, then there must exist a global solution $u_0 \in \mathcal{P}_\infty(CM)$ such that 
\begin{enumerate}
\item $\|u-u_0\|_{(C^0_t \cap C^0_x \cap C^1_v)(Q_1)} \le \eps;$
\item $u_0$ vanishes on a set $[-1,0] \times\R^n \times B_\rho(v_*)$ where $B_\rho(v_*) \subset B_1$ and $\rho = \frac{\sigma}{4n}$.
\item $u$ vanishes on $[-1/2,0] \times B_1 \times B_{\rho/2}(v_*)$. 
\end{enumerate}
\end{lem}
\begin{proof} Assume to the contrary that we may find a sequence $R_n \to \infty$ and $u_n \in \mathcal{P}_{R_n}(MR_n^2)$ such that at least one of the three desired properties fails. In the now-standard way, we apply the regularity estimates of Theorem \ref{thm: optimalreg} to extract a subsequence $u_{n_j} \coloneqq u_j$ converging locally in $C^{0,\alpha}_{t,x} \cap C^{1,\alpha}_v$ to a global solution $u_0 \in \mathcal{P}^p_\infty(CM)$ of the parabolic obstacle problem \eqref{eq: parabolicprob}. As usual, we apply Proposition \ref{prop: caffrecap} and Proposition \ref{prop: blow-upindx} to deduce that $u_0(t,x,v) = f_0(t,v) \in \mathcal{P}_\infty^p(CM)$ is independent of $x$ and also satisfies monotonicity and convexity conditions $\partial_t f_0 \le 0$ and $ D^2_v f_0 \ge 0$ on all of $(-\infty,0] \times \R^n$. Also, it is clear that $\delta_1^p(f_0, (0,0)) = \delta_1(u_0, (0,0,0))\ge \sigma$, since $\limsup \Lambda(u_j) \subset \limsup \Lambda(u_0).$ By construction, $u_0$ satisfies the first condition for $u_j$ with $j$ large. 

From these observations, we repeat the argument of \cite{CPS} verbatim and apply John's ellipsoid lemma to deduce that $\Lambda_{f_0(-1)}$ contains a ball $B_\rho(v_*) \subset B_1$ for $\rho =\displaystyle \frac{\sigma}{4n}$ and some $|v_*| \le 1-\rho$. Then, by the monotonicity $\partial_t f_0 \le 0$, we find $[-1,0] \times B_\rho(v_*) \subset \Lambda_{f_0}$. Therefore $u_0$ satisfies the second condition, since $\Lambda_{u_0}  = \{(t,x,v) \ : \ x\in \R^n, (t,v) \in \Lambda_{f_0}\}$. 

 We now claim that 
 $$\mathcal{U} \coloneqq [-1/2] \times B_1 \times B_{\rho/2}(v_*) \subset \Lambda_{u_j}$$ for sufficiently large $j$. To that end, note that if $\mathcal{U} \cap \Omega_{u_j} \neq \emptyset$, then for large $j$ the following dichotomy holds: either $\mathcal{U} \subset \Omega_{u_j}$, or $ \mathcal{U} \cap \Gamma_{u_j} \neq \emptyset$. In the former case, we would have $\mcL u_j =1$ on $\mathcal{U}$ for all $j$ large, implying $\mcL u_0=1$ there by the Schauder estimate of Theorem \ref{thm: Schauderthm}. This would contradict the fact that $u_0$ vanishes on $\mathcal{U}$. If latter case were to hold, we would be able to apply the nondegeneracy of Proposition \ref{prop: nondegen} at a point $z_j \in \Gamma_{u_j} \cap \mathcal{U}$ to see that
 $$\sup_{Q_r(z_j)} u_j \ge cr^2 \text{ for all } 0<r<\frac12.$$
 
After passing to a subsequence, we may assume the sequence $z_j$ converges to a point $z_0 \in \overline{\mathcal{U}} \cap \Gamma_{u_0}$, which is impossible since $u_0$ vanishes on $[-1,0] \times \R^n \times B_\rho(v_*)$. Therefore there must be an $R_0(n,\eps,M,\sigma)$ such that $(1)-(3)$ hold.
\end{proof}
Let us remark that in everything that follows, $\rho$ is always used to denote $\rho = \displaystyle \frac{\sigma}{4n}$. We now turn towards proving a result on the directional monotonicity of $u$ in the velocity variable. We begin with the following simple \textit{improvement of minimum} lemma, which is the main workhorse for proving the existence of cones of monotonicity near regular points for solutions to obstacle problems. 
\begin{lem}\label{lem: improvementmin} Let $u:Q_1 \to [0,\infty)$ solve \eqref{eq: formulation2} and let $h:Q_1 \to \R$ be such that 
\begin{enumerate}
    \item $h-u \ge -\eps_0 \text{ on } \Omega_u \cap Q_1,$
    \item $h \ge 0 \text{ on } \Gamma_u$, and 
    \item $\mcL h\le 1/2 \text{ on } Q_1.$
\end{enumerate} Then if $\eps_0>0$ is sufficiently small depending only on dimension, there holds that 
$$h-u \ge 0 \text{ on }\Omega_u \cap Q_{1/2}.$$ 
\end{lem}
This type of lemma is typically written under the assumption that $h$ belongs to the null space of the operator associated with the obstacle problem. See, for reference, \cite{CPS}, \cite{SUBELLIPTIC},\cite{OBSPROBREVISITED}. This is an extraneous assumption, and in fact, it suffices to bound $\mcL h$ from above for the standard argument to go through. Following those works, we will go on to set $h= \partial_{v_e} u_r$. In our setting, $\mcL(\partial_{v_e}u_r) = \partial_{x_e}u_r$, which unfortunately does not vanish. The spatial Lipschitz estimate (once again) of Theorem \ref{thm: optimalreg} along with the natural kinetic scaling are the saving graces which will allow us to deduce any kind of monotonicity of $u$ or regularity of $\Gamma_u^{\text{reg}}$.

\begin{proof}[Proof of Lemma \ref{lem: improvementmin}.] Suppose to the contrary that there exists $u,h$ satisfying the conditions and a point $(t_0,x_0,v_0)=z_0 \in Q_{1/2}$ such that $(h-u)(z_0)<0$. Let 
$$f(z) = h(z) - u(z) + \frac12 \vphi(z_0^{-1} \circ z)$$
with $\vphi$ as in Lemma \ref{prop: nondegen} with $r=1/2$. Then $f(z_0) <0$ by construction and $f \ge 0$ on $\Gamma_u \cap Q_1$. Furthermore, on $\Omega_u \cap Q_{1/2}(z_0)$ we have 
$$ \mcL f(z) = \mcL h(z) - 1 +\frac12 \mcL\vphi(z_0^{-1}\circ z) \le 0$$
thanks to the assumption $\mcL h \le \ds \frac12$ and the property $0 \le \mcL(\vphi) \le 1.$ So $f$ is a super-solution on $\Omega_u\cap Q_{1/2}(z_0)$, which, by the maximum principle, necessarily must attain its negative minimum on $\partial_pQ_{1/2}(z_0)$. Thus

$$0 \le \inf_{\partial_pQ_{1/2}(z_0) \cap \Omega_u} h-u + \frac12\vphi(z_0^{-1}\circ \cdot),$$
implying
$$-\eps_0 \le \inf_{\partial_pQ_{1/2}(z_0) \cap \Omega_u} (h-u) \le - \frac{1}{8n+4}.$$
We reach a contradiction taking any $\eps_0<\displaystyle\frac{1}{8n+4}$.
\end{proof}

The following is a useful result on global solutions to the parabolic obstacle problem from \cite{CPS}. 

\begin{thm}[Theorem $11.1$ in \cite{CPS}.]
\label{thm: caffarelliglobal} Let $f=f(t,v) \in \mathcal{P}^p_\infty(M)$ be such that $(0,0) \in \Gamma_f$ and suppose that $\Lambda_f$ contains a cylinder $[-1,0] \times B_\rho(-se_n)$ for some $\rho<s<1-\rho$. Set 
$$K(\delta, s, h) \coloneqq \{v \ : \ |v'|<\delta, -s \le v_n \le h\}$$ for any $\delta,h>0$. Then there exists $C_0 = C_0(n,M,\rho)$ such that for any direction $e \in \mathbb{S}^{n-1}$ such that $e\cdot e_n >\theta = \theta(\rho) \coloneqq \sqrt{1-\rho^2/64},$ there holds
$$C_0 \partial_{v_e}f - f \ge 0\text{ in }[-1/2,0] \times K(\rho/16, s, 1/2). $$
\end{thm}
We use the previous theorem to transfer monotonicity in the $v$-variable to a solution $u \in \mathcal{P}_R(MR^2)$ which is close enough to a global solution. The proof follows the outline of Theorem $13.1$ in \cite{CPS}, except we must account for the commutator $[\partial_{v_e}u, \mcL] = \partial_{x_e}u \neq 0$. Before doing so, let us quickly state an immediate generalization of Lemma \ref{lem: improvementmin}.
\begin{lem}
    \label{lem: improvementminGENERALIZE}
    Let $u$ solve \eqref{eq: formulation2} in 
    $$\mathcal{N}_\delta(E) \coloneqq \bigcup_{z \in E} Q_\delta(z)$$
    for a set $E \subset \R^- \times\R^{2n}$, and let $h$ satisfy
    \begin{enumerate}
        \item $h \ge 0 \text{ on } \mathcal{N}_\delta(E) \cap \Gamma_u$,
        \item $h-u \ge -\eps_0 \text{ in } \mathcal{N}_\delta(E)$
        \item $\mcL h \le 1/2  \text{ on } \mathcal{N}_\delta(E) \cap \Omega_u.$
    \end{enumerate}
    There exists $h-u \ge 0$ on $\mathcal{N}_{\delta/2}(E)$ if $\eps_0 = \eps_0(\delta,n)>0$ is small enough.
\end{lem}
\begin{proof}
    Fix $z \in E$ and the cylinder $Q_\delta(z)$, and then re-scale, recenter, and apply Lemma \ref{lem: improvementmin}. 
\end{proof}
\begin{prop}
\label{prop: vmonotoneglobal}
    For every $\sigma>0$ there exists $R_0=R_0(\sigma, n,M)>0$ such that if $u \in \mathcal{P}_R(MR^2)$ for some $R \ge R_0$, and $\delta_1(u)\ge \sigma$, then, possibly after a rotation of axes, there holds $$C_0\partial_{v_e} u -u \ge 0 \text{ on } [-1/4,0] \times B_{1/2} \times K(\rho/16, s, 1/4)$$
    for every $e \in \mathbb{S}^{n-1}$ such that $e\cdot e_n >\theta= \theta(\rho) = \sqrt{1-\rho^2/64}.$ Here $C_0 =C_0(n,M,\rho)$ is as in Theorem \ref{thm: caffarelliglobal}. 

\end{prop}
\begin{proof}
    Let $\eps = \eps(\sigma,n,M)>0$ be decided later, and let $R_0 = R_0(\sigma, n,M, \eps)$ be as in Lemma \ref{lem: approxlemma}. Under the present assumptions, we may apply that lemma to obtain the existence of a solution $u_0 \in \mathcal{P}^p_\infty(M)$ to the parabolic problem such that
    \begin{equation}
        \label{eq: closeness} \|u-u_0\|_{(C^0_{t,x} \cap C^1_v)(Q_1)} <\eps
    \end{equation}
    and such that $u_0$ satisfies the assumptions of Theorem \ref{thm: caffarelliglobal} with $\rho = \displaystyle \frac{\sigma}{4n}$, perhaps after rotating the space and velocity axes if necessary. With $C_0 = C_0(\rho,n,M)$, and $e$ any unit vector such that $e \cdot e_n >\sqrt{1-\rho^2/64}$, we have that $u_0 = u_0(t,v)$ satisfies
    $$C_0 \partial_{v_e} u_0 - u_0 \ge 0 \text{ on } [-1/2,0] \times K(\rho/8,s,1/2). $$

    By \eqref{eq: closeness}, there necessarily holds
    $$C_0 \partial_{v_e}u-u \ge -C_0\eps - \eps \ge -\eps_0 \text{ on } [-1/2,0] \times B_1 \times  K(\rho/8,s,1/2)$$
    where $\eps_0=\eps_0(n)$ is as in Lemma \ref{lem: improvementminGENERALIZE}, if $\eps$ is chosen small enough. Now, since $u \in \mathcal{P}_R(MR^2)$, we have $u_R(\cdot) = R^{-2}u(S_R(\cdot)) \in \mathcal{P}_1(M)$ and thus Theorem \ref{thm: optimalreg} implies that $\|\nabla_x u_R\|_{L^\infty(Q_{1/2})} \le CM$, from which we deduce that  $$\|\nabla_x u\|_{L^\infty(Q_{R/2})} \le \frac{CM}{R}.$$

    Consequently, taking $R_0 = R_0(\sigma,n,M)$ larger if need be so that $R_0>2CC_0M$, we may apply Lemma \ref{lem: improvementminGENERALIZE} with $h =C_0\partial_{v_e}u$ to deduce that for any $u \in \mathcal{P}_R(MR^2)$ with $R \ge R_0$ and $\delta_1(u) \ge \sigma$, there holds
    $$C_0\partial_{v_e} u -u \ge 0 \text{ on } [-1/4,0] \times B_{1/2} \times K(\rho/16, s, 1/4).$$
\end{proof}

 We already observed from the upper semi-continuity of the balanced energy $\mathcal{E}$ that $\Gamma_u^{\text{reg}}$ is open. The purpose of this next lemma is to prove that free boundary points near which the contact set is thick enough are necessarily regular points, and to quantify this openness. We show that it is enough for the contact set $\Lambda_u$ to be thick enough at a single scale near $z_0 \in \Gamma_u$ in order to deduce that $\Lambda_u$ has positive Lebesgue density at every free boundary point near $z_0$. In particular, we deduce that $\Lambda_u$ is large in cylinders of infinitely small scale, even centered at nearby points. Here we are using the equivalent characterization of regular points through the  Lebesgue density of the contact set, as observed in Proposition \ref{prop: FBdensityclassification}.
\begin{lem}
\label{lem: thicknessenergy}
    Given $\sigma>0$, there exists $R_0= R_0(\sigma,n,M)>0$ such that if $u \in \mathcal{P}_R(MR^2)$ for $R \ge R_0$, and $\delta_1(u) \ge \sigma$, then 
    $$\Gamma_u \cap Q_{1/2} \subset \Gamma_u^{\text{reg}}.$$
    As a consequence, the balanced energy is such that $\mathcal{E}(z_0;u) \equiv \omega$ for $z_0 \in \Gamma_u \cap Q_{1/2}$, and $\Lambda_u$ has positive Lebesgue density at every point $z_0 \in \Gamma_u \cap Q_{1/2}$. 
\end{lem}
\begin{proof}
   Suppose for the sake of contradiction that there exist sequences $R_k \to \infty$, $z_k \in Q_{1/2},$ and $u_k \in \mathcal{P}_{R_k}(MR_k^2)$ satisfying $\delta_1(u_k) \ge \sigma$, but 
   $$\mathcal{E}(1/k, z_k; u_k) \ge 2\omega - \frac1k.$$

   Recenter with $\tilde{u}_k(z) \coloneqq  u_k(z_k \circ z)$. Owing to the regularity estimates in Theorem \ref{thm: optimalreg}, we may pass to a further subsequence (still denoted by $\tilde{u}_k$) converging locally in $C^{1,\alpha}_\ell(\R^- \times \R^{2n})$ to a global solution $\tilde{u}_0 \in \mathcal{P}_\infty(CM)$. The fact that $u_0$ is a global solution of the parabolic obstacle problem \eqref{eq: parabolicprob} is a consequence of Corollary \ref{cor: blowupssolveq}. Then, for any $r>0$, the almost monotonicity formula from Proposition \ref{prop: monotonicityformula} furnishes
   $$\mathcal{E}(r;\tilde{u}_0)= \lim_{k \to \infty} \mathcal{E}(r, z_k; u_k) \ge \lim_{k \to \infty} \mathcal{E}(1/k, z_k;u_k) \ge 2\omega,$$
   since eventually $1/k<r$.

   It is clear then that 
   $$\displaystyle \lim_{r \to 0} \mathcal{I}(r,\tilde{u}_0) = \lim_{r \to 0} \mathcal{E}(r, \tilde{u}_0) \ge 2\omega.$$
   Since $r \mapsto \mathcal{I}(r,u_0)$ is monotone increasing in $r$, and $\lim_{R \to \infty} \mathcal{I}(R, u_0) = 2\omega$, we deduce that $\mathcal{I}(r,u_0) \equiv 2\omega$ for all $r>0$. Consequently, Proposition \ref{prop: caffrecap} ensures that $u_0$ is a polynomial solution of the form $\tilde{u}_0 = mt + Av\cdot v$, where $m+1 = \frac12 \text{tr}(A)$. 

   On the other hand, as soon as $R_k \ge R_0$ for $R_0=R_0(n,1,M,\sigma)$ as in Lemma \ref{lem: approxlemma}, the functions $u_k$ vanish on cylinders $[-1/2] \times B_1 \times  B_{\rho/2}(v^*_k) \subset Q_1$ for some $v_k^* \in B_1$. Recall here that $\rho = \displaystyle \frac{\sigma}{4n}$. Therefore, for all $k$, the translations $\tilde{u}_k$ vanishes on a translated cylinder contained in $Q_2$. We deduce that $\tilde{u}_0$ vanishes on a nondegenerate cylinder, contradicting that $\tilde{u}_0$ is a nonzero quadratic polynomial. Therefore we may take $R_0 = R_0(\sigma,n,M)$ large so that the conclusion of the lemma holds.
\end{proof}

We now turn to proving that the material derivative $Yu$ vanishes near regular points. We follow the approach of \cite{CPS} used to show that $\partial_t u$ vanishes at regular points in the parabolic setting.

\begin{prop}
\label{prop: Yvanish}
    Fix $\sigma>0$ and assume $u \in \mathcal{P}_R(MR^2)$ for some $R\ge R_0$, where $R_0 = R_0(n,\sigma,M)$ is as in Lemma \ref{lem: thicknessenergy}. Suppose also that $\delta_1(u) \ge \sigma$. Then for all $\eps>0$, there exists $r_0 =r_0(\sigma, n,M,\eps)>0$ such that
    $$\sup_{Q_{r_0}}|Yu|<\eps.$$
\end{prop}
\begin{proof}
    First of all, the thickness condition $\delta_1(u) \ge \sigma$ implies that $\Gamma_u \cap Q_{1/2}$ consists of only low energy points, using the previous lemma. Now, for all $r>0$, the functions $\mathcal{E}_r$ given by $z \mapsto \mathcal{E}(r,z; u)$ are continuous on $\Gamma_u \cap \overline{Q_{1/3}}$ and converge pointwise to the balanced energy $\mathcal{E}(z) \equiv \omega$, seeing as every point in $z\in \Gamma_u \cap \overline{Q_{1/3}}$ is low energy. Since $\mathcal{E}(z,r) + Cr^2$ is monotone increasing for some $C=C(M,n)$ uniform in $z \in \Gamma_u \cap Q_{1/2}$, we may apply Dini's monotone convergence theorem to find that the convergence $\mathcal{E}_r \to \omega$ is uniform on $\overline{Q_{1/3}}$. As a consequence, for any sequences $r_j \to 0$ and $z_j \in \Gamma_u \cap Q_{1/3}$ converging to $(0,0,0)$, there holds
    $$\mathcal{E}(r_j, z_j;u) \to \omega.$$

   Now, let   $$m \coloneqq \limsup_{r \to 0} \{Yu(z) \ : \ z \in \Omega_u \cap Q_r \}.$$

Of course, by the Lipschitz estimates of Theorem \ref{thm: optimalreg}, $|m|<\infty$. We will show that $m \le 0$. To that end, suppose for the sake of contradiction that $m>0$ and $z_j \in \Omega_{u_j} \cap Q_{\frac{1}{2j}}$ be a maximizing sequence. Setting $$d_j \coloneqq  \sup\{r>0 \ : \ Q_r(z_j) \subset \Omega_{u_j}\},$$
there necessarily exists a sequence $\tilde{z}_j \in \Gamma_u \cap \partial_p Q_{d_j}(z_j).$ By passing to a subsequence if necessary, we may assume that
$$u_j(z) \coloneqq \frac{u(z_j \circ S_{d_j}(z))}{d_j^2} \to u_0,$$
with the convergence happening in $(C^{0,\alpha}_{t,x} \cap C^{1,\alpha}_v)(\mathcal{K})$ on every compactly contained subset $$\mathcal{K} \subset \limsup_{j \to \infty} Q_{d_j^{-1}}(z_j^{-1}).$$

Now, $u_0$ inherits the nonnegativity and $C^{0,1}_{t,x} \cap C^{1,1}_v$ local regularity of the $u_j$, and by the same argument as in Proposition \ref{prop: blow-upindx}, we find that $\mcL u_0 =\chi_{u_0>0}$ on the domain on which it is defined. In particular, $\mcL u_0=1$ on $Q_1$. Our first concern is showing that $u_0$ is a global solution independent of $x$. To that end, we must show that $d_j \to 0$. Indeed, passing to another sequence if necessary, assume for the sake of contradiction that $d_j \to d_0>0$. Since $z_j \to (0,0,0)$, we would have that $Q_{d_0} \setminus \{t=0\} \subset \Omega_{u}$, contradicting the fact that the contact set $\Lambda_u$ must satisfy the Lebesgue density condition \eqref{eq: lebesguedensitycondition} at the origin. We are using here the equivalent characterizations of regular points observed in Proposition \ref{prop: FBdensityclassification}.

With this in hand, the same argument as in Proposition \ref{prop: blow-upindx} implies that $u_0$ is independent of $x$, and therefore $u_0(t,x,v) \coloneqq f_0(t,v)$ solves the parabolic obstacle problem
\eqref{eq: parabolicprob} and also satisfies the inequalities
    $$0\le f_0(t,v) \le CM(1+|v|^2 +|t|) \ \text{ for } (t,v) \in (-\infty,0]\times \R^n.$$
Now, since $\mcL(u_j)=1$ on $Q_1$ for all $j$, we apply the kinetic Schauder estimates from Theorem \ref{thm: Schauderthm} to obtain the convergence $u_j \to f_0$ in $C^{2,\alpha}_{\ell}(\overline{Q_{1-\delta}})$ for any $0<\alpha,\delta<1$, by possibly passing to a further subsequence. Accordingly, it is straightforward to see that
    $$\partial_t f_0(0,0) = Y u_0(0,0,0)=\lim_{j \to \infty} Y u_j(0,0,0) = \lim_{j\to \infty} Yu_j(z_j) =m$$
    and
    $$\partial_t f_0(t,v) = Yu_0(t,x,v) = \lim_{j \to \infty} Yu_j(z_j \circ S_{d_j}(t,x,v)) \le m$$
    for any $(t,x,v) \in Q_1$. 
The parabolic maximum principle furnishes that $\partial_t u_0 \equiv m$ on $Q_1$. The rest of the proof is verbatim the same as in Lemma $7.7$ of \cite{CPS}. For posterity, we repeat it.

Without loss of generality, assume $\xi_j = S_{d_j^{-1}}(z_j^{-1}\circ \tilde{z}_j) \to \xi_0 =(s_0,y_0,w_0) \in \partial_p Q_1$. Seeing as $\xi_j \in \Gamma_{u_j} \cap \partial_pQ_1$, the energy considerations at the beginning of the proof imply that
$$\mathcal{E}(r, \xi_0; u_0) = \lim_{j \to \infty} \mathcal{E}(r, \xi_j; u_j) = \lim_{j \to \infty} \mathcal{E}(d_jr, \tilde{z}_j; u) =\omega$$
for every $r>0$. Here we used the identity $\mathcal{E}(\rho, z_1, u_r^{z_0}) = \mathcal{E}(r\rho, z_0 \circ S_r(z_1);u)$, which holds in general. Hence, for some $e \in \mathbb{S}^{n-1}$, the classification of global homogeneous solutions as in \cite{CPS} implies that $f_0$ is a backwards homogeneous polynomial solution of the form
$$f_0(t,v)= \frac12((v-w_0)\cdot e)_+^2 \text{ in } (-\infty,s_0] \times \R^n = \R_{s_0}^- \times \R^n.$$

Recall the notation $\mathcal{Q}_1'$ for the unit parabolic cylinder, introduced in Proposition \ref{prop: parabolicnondegen}. We show the above form of $f_0$ on $\R^-_{s_0} \times \R^n$ is in contradiction with the derived property $\partial_t f_0 \equiv m\ge 0$ on $\mathcal{Q}'_1$, unless $m=0$. The contradiction if $s_0>-1$ is immediate. The next possibility is that $s_0=-1$ and $B_1 \cap \{(v-w_0)\cdot e>0\} \coloneqq E \neq\emptyset$, then $\partial_t f_0$ has a discontinuity across $E \times \{-1\} \subset \{f_0>0\}$, contradicting that $f_0$ is caloric there. Finally, if $s_0=-1$ and $u_0$ vanishes on $\{-1\} \times B_1$, we would solve the forward problem to have $f_0(t,v) = m(t+1)$ in $\mathcal{Q}'_1$. Then $\nabla_vf_0$ has a jump across $\{-1\} \times \partial E$, which again is not possible. Therefore the assumption that $m>0$ invariably leads to a conclusion. We didn't really use anything about the sign of $m$ in the above argument, and in the same way we obtain
$$\lim_{r \to 0} \inf \{Yu(z) \ : \ z \in \Omega_u \cap Q_r\} \ge 0.$$
This concludes the proof.
\end{proof}
\begin{remark}
    Under no assumptions on the free boundary, we may obtain $$\limsup_{\Omega_u \ni z \to  \Gamma_u} Yu(z)\le 0$$ by the same argument. The proof would proceed in the same way, extracting a subsequence converging to a global solution $u_0$ of the parabolic obstacle problem that obeys $\partial_t u_0 = m$, with $m$ as in the above proof. Then the first statement in Proposition \ref{prop: caffrecap} implies $m\le 0$. 
\end{remark}
\begin{remark}
    It does not seem straightforward to prove results about $\partial_t u$ or $\nabla_x u$ vanishing continuously on $\Gamma_u$ using the same methods. As mentioned in the introduction, information on $\partial_t u$ is lost by left-action of the Galilean group:
    $$\partial_t u_r^{z_0}(z) = (\partial_t + v_0\cdot \nabla_x)u(z_0 \circ S_r(z)).$$
    On the other hand, information on $\nabla_x u$ is lost due to the scaling
    $$\nabla_x u_r^{z_0}(z) = r\nabla_x u(z_0 \circ S_r(z)).$$ We really use here the Galilean and scaling invariance of the material derivative $Y= \partial_t + v\cdot \nabla_x$. 
\end{remark}
\begin{remark}
    We may choose $R_0 = R_0(n,\sigma,M)$ such that Proposition \ref{prop: Yvanish} and Proposition \ref{prop: vmonotoneglobal} hold simultaneously. 
\end{remark}
We conclude with a proof of the fact that blow-ups at regular points are unique. We do not make any assumptions on the thickness of the contact set. The proof follows along the lines of the proof of Proposition \ref{prop: vmonotoneglobal}. 
\begin{prop}
    Suppose that $u\in \mathcal{P}_1(M)$ and that $(0,0,0) \in \Gamma_u^{\text{reg}}$, as defined in Proposition \ref{prop: fbclassification}. Then there exists a unique blow-up $u_0$ of the form \eqref{eq: halfspace}.
\end{prop}
\begin{proof}
    After rotating if necessary, we may assume that at least one blow-up of $u$ is given by $u_0(t,x,v) = \frac12(v_n)_+^2.$ Now, for any direction $e \in \mathbb{S}^{n-1}$ with $e \cdot e_n>\delta$, one has
    $$\frac1\delta \partial_{v_e} u_0 -u_0 \ge 0 \text{ on } Q_1.$$

Consequently, owing to the spatial Lipschitz estimate of Theorem \ref{thm: optimalreg}, and the $C^{1,\alpha}_\ell(Q_1)$ convergence of $u_{r_j} \to u_0$ for some sequence $r_j \to 0$,  there exists an $r_\delta>0$ small such that $\|\nabla_x u_{r_\delta}\|_{L^\infty(Q_1)} \le \delta/2$, and 
$$\frac1\delta \partial_{v_e} u_{r_\delta} -u_{r_\delta} \ge -\eps_0 \text{ on } Q_1 \text{ for any direction }e \in \mathbb{S}^{n-1} \text{ with }e\cdot e_n >\delta.$$

Here $\eps_0$ as in Lemma \ref{lem: improvementmin}. Arguing in the same way as in Proposition \ref{prop: vmonotoneglobal}, we deduce $$\frac1\delta \partial_{v_e} u_{r_\delta} \ge u_{r_\delta} \text{ on } Q_{1/2} \text{ for any direction }e \in \mathbb{S}^{n-1} \text{ with }e\cdot e_n >\delta.$$

Consequently, $\partial_{v_e}u \ge 0 \text{ on } Q_{r_\delta/2}$ for any direction $e\cdot e_n >\delta$. Now, by the classification of blow-ups in Proposition \ref{prop: fbclassification}, the only other possible blow-ups of $u$ at the origin are of the form $u_1 =\frac12(v\cdot \nu)_+^2$ for some $\nu \neq e_n$. This dichotomy arises from the monotonicity formula in Proposition \ref{prop: monotonicityformula}. By the preceding analysis, $u_1$ necessarily would satisfy
$$\partial_{v_e} u_1 \ge 0 \text{ on } \R^- \times \R^{2n}, \text{ for any } e\cdot e_n \ge 0.$$

This is only possible if $\nu=e_n$.
\end{proof}
\section{H{\"o}lder Regularity of the Free Boundary}
\label{FBregularitysec}
This final section is dedicated to the proofs of Theorem \ref{thm: freeboundaryregularity} and Proposition \ref{prop: graphdiff}. We continue to work with solutions $u \in \mathcal{P}_R(MR^2)$ which are ``close" to global solutions, as in Lemma \ref{lem: approxlemma}. The strategy is to first establish directional monotonicity in the $v$-variable, and then to establish elliptic nondegeneracy of $u$ on $(t,x)$-sections of the free boundary as a consequence of the material derivative $Yu$ vanishing continuously at regular points of $\Gamma_u$. With the monotonicity and nondegeneracy in $v$, we find a quadratic \textit{pointwise} lower bound for $u$ on $v$-cones centered at free boundary points. From there, the result follows from the regularity of $u$. We emphasize again that, while we are referring to Theorem \ref{thm: freeboundaryregularity} as a H{\"o}lder regularity result with the Euclidean distance in mind, it also implies that $\Gamma_u$ is \textit{Lipschitz} with respect to the kinetic distance.

We begin with the following lemma which is an immediate consequence of Proposition \ref{prop: Yvanish}. Essentially, the fact that $Yu$ can be made arbitrarily small on a neighborhood of the free boundary implies elliptic nondegeneracy on $(t,x)$-sections of $\Gamma_u$. Notice that this lemma is Galilean and scaling invariant, as  $Yu_r^{z_0}(z) = (Yu)(z_0 \circ S_r(z))$. 
\begin{lem} \label{lem: ellipticnondegen}
For every $\sigma>0$ there exists $R_0 = R_0(\sigma,n,M)>0$ and $r_0(\sigma,n,M)>0$ such that the following holds: if $u \in \mathcal{P}_R(MR^2)$ for some $R \ge R_0$, and $\delta_1(u) \ge \sigma$, then 
$$\sup_{z \in Q_{2r_0}} |Yu| \le \frac12$$

and consequently,
$$\Delta_v u \ge \frac12 \text{ on }\Omega_u \cap  Q_{2r_0}.$$
Therefore $u$ enjoys \textit{elliptic} nondegeneracy at every free boundary point $z_0=(t_0,x_0,v_0) \in \Gamma_u \cap Q_{r_0}$, in the sense that
\begin{equation}
\label{eq: vnondegen}
\sup_{w \in B_\lambda} u(t_0,x_0, v_0 + w) \ge c_n \lambda^2 \text{ for all } 0<\lambda\le r_0
\end{equation}
for some dimensional constant $c_n>0$.
\end{lem}
\begin{proof}
The first statement is proven in Proposition \ref{prop: Yvanish}, and the second statement follows from the equation. For the elliptic nondegeneracy, fix a free boundary point $z_0=(t_0,x_0,v_0) \in \Gamma_u \cap Q_{r_0}$. For any $\delta>0$, there exists $z_\delta = (t_\delta, x_\delta, v_\delta) \in \Omega_u \cap Q_\delta(z_0).$ Now, define $f=f(v): B_{r_0-\delta} \to \R$ by 
$$f(v) \coloneqq u(t_\delta,x_\delta, v_\delta +v) - u(t_\delta, x_\delta, v_\delta) - \frac{1}{4n}|v|^2.$$

We claim that $f$ is subharmonic on the open set $\{v \in B_{r_0-\delta} \ : \ u(t_\delta,x_\delta, v_\delta+v)>0\}$. Indeed, on that set, our previous observations imply that 
$$\Delta_v f(v) = (\Delta_v - Yu)(t_\delta,x_\delta,v_\delta+v) + Yu(t_\delta,x_\delta,v_\delta+v) - \frac12 \ge 0.$$

Seeing as $f(0)=0$ and $f<0$ on $\{v \ : \ u(t_\delta,x_\delta,v_\delta+v)=0\}$, the maximum principle ensures that for any $0<\lambda<r_0-\delta $, there holds
$$0 \le \max_{\{u(t_\delta,x_\delta,v_\delta+\cdot)>0\}\cap B_\lambda} f = \max_{\{u(t_\delta,x_\delta,v_\delta+\cdot)>0\} \cap \partial B_\lambda} f.$$

Accordingly,
$$\max_{B_\lambda} u(t_\delta,x_\delta, v_\delta + v) \ge \frac{1}{4n}\lambda^2.$$

Sending $\delta \to 0$ and $(t_\delta, x_\delta, v_\delta) \to (t_0,x_0,v_0)$, we conclude. 
\end{proof}

The following is a very simple geometric lemma. Essentially it says that if we take a point $v= (v',v_n)$ in a cone centered at the origin in $\R^n$, and consider a cone emanating backwards from $v$ in the $v_n$ direction with slightly wider opening, the latter cone will contain a ball around the origin of radius comparable to $v_n$. Let us introduce the cones

\begin{equation}
    \label{eq: cone}
   \mathcal{C}(e, \theta)\coloneqq  \left\{v \in B_1 \ : \ \frac{v}{|v|} \cdot e >\theta\right\}, \ \ \text{ where } 0<\theta <1,\text{ and } e \in \mathbb{S}^{n-1}.
\end{equation}

For cones centered at points other than the origin, we write $v_0 + \mathcal{C}(e,\theta)$ to denote the cone of the same opening, orientation, and slant length centered at $v_0$. Likewise, $v_0 - \mathcal{C}(e, \theta)$ denotes the backwards cone. 
 
\begin{lem} \label{lem: cone}
    Fix $\delta, \theta \in (0,1)$ and $v = (v', v_n) \in\displaystyle \mathcal{C}\left(e_n, \theta \right)$. Then there exists $c=c(\delta)$ such that for any $w \in B_{cv_n}$, the decomposition $v = se + w$ holds some $s\ge 0$ and some direction $e \in \mathbb{S}^{n-1} \cap \displaystyle \mathcal{C}\left(e_n, (1-\delta)\theta \right)$. In other words, for any $v \displaystyle\in \mathcal{C}\left(e_n, \theta\right)$, one has
    $$B_{cv_n} \subset v - \mathcal{C}\left(e_n,(1-\delta)\theta\right).$$
\end{lem}

\begin{proof}
    Let $v_n>0$, otherwise there's nothing to prove. Fix $c=c(\delta)>0$ to be decided. We show that any $w \in B_{cv_n}$ may be connected to $v$ through some direction $e$ with $e \cdot e_n > (1-\delta)\theta$. To do so, we have to show that
    \begin{equation}
    \label{eq: conestuff}
        \frac{v-w}{|v-w|} \cdot e_n > (1-\delta)\theta
    \end{equation}
    Well,
    $$(v-w) \cdot e_n = v_n -w_n > (1-c)v_n > (1-c)\theta |v|,$$
    and
    $$|v-w| \le |v|+|w| \le (1+c)|v|.$$

    Evidently, the quotient in \eqref{eq: conestuff} is larger than $(1-\delta)\theta$ provided that
    $$\frac{(1-c)}{1+c}  > 1-\delta$$
    which holds as soon as 
    $$c< \displaystyle \frac{\delta}{2-\delta}.$$

\end{proof}

The previous steps combined imply a uniform pointwise lower bound on cones centered at regular points. Notice the same argument applies to the elliptic and parabolic problems.
\begin{prop}\label{prop: quadraticlowerbd}
    Given $\sigma>0$ small, there exists $r_0(\sigma,n,M)>0$ and $R_0(\sigma,n,M)>0$ such that if $u \in\mathcal{P}_R(MR^2)$ for $R \ge R_0$ and $\delta_1(u) \ge \sigma$, then for some $e \in \mathbb{S}^{n-1}$, $\psi = \psi(\sigma,n) \in (0,1)$ and $c=c(n,\sigma)>0$, there holds at any free boundary point $z_0 \in \Gamma_u \cap Q_{r_0/2}$ the point-wise lower bound
    \begin{equation}
    \label{eq: vlowerbd}
    u(t_0,x_0, v_0 + v) \ge c(v\cdot e)_+^2
    \end{equation}
    for $v$ belonging to the conical section $B_{r_0/2} \cap \mathcal{C}(e, \psi).$
\end{prop}
\begin{proof}
Recall that we may take $R_0 = R_0(\sigma,n,M)$ large so that Proposition \ref{prop: vmonotoneglobal} and Lemma \ref{lem: ellipticnondegen} hold simultaneously. Rotate axes if necessary to ensure that $u$ is as in Proposition \ref{prop: vmonotoneglobal}. We will show a pointwise lower bound on a cone in the $v_n$ direction. With $\rho = \displaystyle \frac{\sigma}{4n}$ as usual, Proposition \ref{prop: vmonotoneglobal} guarantees the existence of $s \in (\rho, 1-\rho)$ such that
$$\partial_{v_e}u \ge 0 \text{ on } [-1/4,0] \times B_1 \times K(\rho/16, s, 1/4)$$
for any $e \in \mathbb{S}^{n-1} \cap \mathcal{C}(e_n, \theta)$, where $\theta = \theta(\rho) = \sqrt{1-\rho^2/64}$. Note that the dependence of $\theta$ is actually on $\sigma$ and $n$. Now, as in Lemma \ref{lem: thicknessenergy}, we have that 
$$[-1/4,0] \times B_{1/2} \times K(\rho/16, s, 1/4) \supset Q_{c_n\sigma},$$
for some dimensional constant $c_n>0$, provided $\sigma$ is not arbitrarily large. 

Applying Proposition \ref{prop: vmonotoneglobal} and Lemma \ref{lem: ellipticnondegen}, there exists $r_0 = r_0(\sigma,n,M)$ such that there holds simultaneously the directional monotonicity 
\begin{equation}
    \label{eq: vmonotonefinal}
    \partial_{v_e}u \ge 0 \text{ on } Q_{r_0} \text{ for any } e \in \mathbb{S}^{n-1} \cap \mathcal{C}(e_n, \theta)
\end{equation} as well as the elliptic non-degeneracy
\begin{equation}\sup_{w \in B_\lambda} u(t_0,x_0, v_0+w) \ge c_n \lambda^2 \text{ for all } z_0 \in \Gamma_u \cap Q_{r_0}, 0<\lambda<r_0.
\label{eq: ellipticnondegenfinal}
\end{equation}

Now, fix a free boundary point $(t_0,x_0,v_0) \in \Gamma_u \cap Q_{r_0/2}$, and take an arbitrary velocity $v=(v',v_n) \in B_{r_0/2} \cap \mathcal{C}\left(e_n, \frac{1+\theta}{2}\right)$. Thanks to Lemma \ref{lem: cone}, there exists $c=c(\sigma,n)$ such that for any $w \in B_{c v_n}$, we may write 
\begin{equation}
    \label{eq: conerepresentation}
    v = w+ se
\end{equation}
for some $s\ge 0,$ and $e \in \mathbb{S}^{n-1} \cap \mathcal{C}(e_n, \theta)$. Owing to the elliptic non-degeneracy \eqref{eq: ellipticnondegenfinal}, we may take $w \in B_{c_\sigma v_n}$ such that
$$u(t_0,x_0,v_0+w) =\tilde{c} v_n^2,$$

where $\tilde{c} = \tilde{c}(\sigma,n)$ still depends only on $n$ and $\sigma$. Then the decomposition \eqref{eq: conerepresentation} and the directional monotonicity \eqref{eq: vmonotonefinal} imply that  
$$u(t_0,x_0,v_0+v) \ge \tilde{c} v_n^2$$
for any $v \in B_{r_0/2} \cap \mathcal{C}\left(e_n,\frac{1+\theta}{2}\right),$
as desired. The claim follows with $\displaystyle \psi = \frac{1+\theta}{2}.$  Note that we have assumed $v_n>0$, for otherwise there is nothing to prove. 

\end{proof}
We may now turn to our main regularity result for the free boundary. The proof amounts to combining the point-wise quadratic lower bound on $v-$cones from the previous lemma with the Lipschitz estimates from Theorem \ref{thm: optimalreg}. 
\begin{prop}
\label{prop: graph}
Given $\sigma>0$, there exists $r_0(\sigma,n,M)>0$ and $R(\sigma,n,M)>0$ such that if $u \in \mathcal{P}_R(MR^2)$ for $R \ge R_0$, and $\delta_1(u) \ge \sigma$, then there exists a function $f \in C^{0,1/2}_{t,x} \cap C^{0,1}_v$ with $\|f\|_{C^{0,1/2}_{t,x} \cap C^{0,1}_v} \le C(M,n,\sigma)$ such that, perhaps after a rotation, there holds
$$\Gamma_u \cap Q_{r_0/2} = \{\{t,x,v) \in Q_{r_0/2} \ : \ v_n = f(t,x,v')\}$$
and
$$\Lambda_u \cap Q_{r_0/2} = \{(t,x,v) \in Q_{r_0/2} \ : \ v_n < f(t,x,v')\}.$$
\end{prop}
\begin{proof}
Let $\psi, \theta, \rho, r_0$, and $R_0$ be as in the last proposition. Let us first remark that $r_0\le c_n\sigma$, and since we are primarily interested in $\sigma$ small, we assume $r_0 \in (0,1)$. We also assume $R \ge R_0 \ge 1$. After rotating, there is no loss in generality in assuming that $u$ is as in the proposition as well. Using the Lipschitz estimates of Theorem \ref{thm: optimalreg}, there exists a dimensional constant $C>0$ such that
$$\|\partial_t u\|_{L^\infty(Q_1)} + \|\nabla_x u\|_{L^\infty(Q_1)} \le CM.$$

This estimate is independent of $R>1$. Consequently, there exists $a=a(M,n, \sigma)>0$ such that the H{\"o}lder cone
\begin{equation}
    \label{eq:holdercone}
\mathcal{A}_{r_0} \coloneqq \left\{(t,x,v) \in Q_{r_0/2} \ : \ a(|t|+|x|)^{1/2} + |v|< \frac{v_n}{\psi}\right\}
\end{equation}
is contained in the set $\Omega_u$. Indeed, for any $(t,x,v) \in \mathcal{A}_{r_0}$, the previous Proposition implies that for some $c = c(n,\sigma)>0$, there holds
\begin{align} u(t,x,v) &\ge u(0,0,v) - |u(t,x,v) - u(0,0,v)| \ge cv_n^2 - CM(|t|+|x|)\\
& \ge \left(c- \frac{CM}{a^2\psi^2}\right)v_n^2>0
\end{align}
provided we take $a = a(M,n,\sigma)>0$ large enough. 
Finally, we claim that $u\equiv 0$ on the ``backwards" cone (relative to the $v_n$ direction) given by 
    $$- \mathcal{A}_{r_0} \coloneqq \left\{(t,x,v) \in Q_{r_0/2}\ : \ a(|t| + |x|)^{1/2} + |v| < -\frac{v_n}{\psi}\right\}.$$

     To that end, there are three possibilities: 
     \begin{enumerate}
         \item $-\mathcal{A}_{r_0} \subset \Omega_u=\{u>0\}$, or
         \item $-\mathcal{A}_{r_0} \cap \Gamma_u \neq \emptyset$, or
          \item $-\mathcal{A}_{r_0} \subset \Lambda_u = \{u=0\}$.
     \end{enumerate}
     
     Thanks to the monotonicity condition \eqref{eq: vmonotonefinal}, and the fact that $(0,0,0) \in \Gamma_u$, we find that $u(0,0,v) \equiv 0$ on the backwards conical section $B_{r_0/2} \cap \mathcal{C}\left(-e_n, \theta \right)$. Therefore $-\mathcal{A}_{r_0}$ cannot be contained in $\Omega_u$. This rules out the first possibility. Next, if there were to exist $(t_0,x_0,v_0) = z_0 \in - \mathcal{A}_{r_0} \cap \Gamma_u$, the same argument as in the first part of the proof would imply that $u>0$ on 
    $$\mathcal{A}_{r_0}(z_0) \coloneqq \left\{z \in Q_{r_0/2} \ : \ a(|t-t_0| + |x-x_0|)^{1/2} + |v-v_0| < \frac{(v-v_0)\cdot e_n}{\psi}\right\}.$$

    This would follow in the same way by the pointwise lower bound \eqref{eq: vlowerbd} and the monotonicity \eqref{eq: vmonotonefinal}. It is immediate to verify that $(0,0,0) \in \mathcal{A}_{r_0}(z_0)$ if $z_0 \in -\mathcal{A}_{r_0}$. This leads to the contradiction that $u(0,0,0)>0$, which cannot be if $(0,0,0) \in \Gamma_u$. 
 We are left with the conclusion that $u\equiv 0$ on the backwards cone $-\mathcal{A}_{r_0}$.  Repeating this argument at any $z_0 \in \Gamma_u \cap Q_{r_0/2}$, we conclude thanks to the monotonicity of $u$ in the $v_n$ direction that $\Gamma_u \cap Q_{r_0/2}$ is the graph of a function $f = f(t,x,v')$ such that  
    \begin{align} \Omega_u \cap Q_{r_0/2} &= \{(t,x,v) \in Q_{r_0/2} \ : \ v_n > f(t,x,v')\\
    \Gamma_u \cap Q_{r_0/2} & = \{(t,x,v) \in Q_{r_0/2} \ : \ v_n = f(t,x,v')\}
    \end{align}
    where $\|f\|_{C^{0,1/2}_{t,x} \cap C^{0,1}_v} \le C(M,n,\sigma).$

\end{proof}
The proof of our main free boundary regularity result, Theorem \ref{thm: freeboundaryregularity}, follows in the standard way. 
\begin{proof}[Proof of Theorem \ref{thm: freeboundaryregularity}.]
Theorem \ref{thm: freeboundaryregularity} follows by scaling. The map $\sigma \mapsto R_0(\sigma)$ may be assumed monotone increasing and continuous, with $R_0 \to \infty$ as $\sigma \to 0$. By scaling, the map $\sigma \mapsto  \frac{1}{R_0(\sigma)}$ is such that if $u\in \mathcal{P}_1(M)$ and $\delta_{1/R_0}(u) \ge \sigma$, then $\Gamma_u \cap Q_{c/R_0}$ satisfies the conclusion of the theorem for some $c = c(n,M,\sigma).$ Inverting the monotone, continuous map $\sigma \mapsto \frac{1}{R_0}$, we obtain for solutions $u \in \mathcal{P}_1(M)$ a modulus of continuity $\sigma=\sigma(r)$ such that $\delta_r(u) \ge \sigma(r)$ implies $\Gamma_u \cap Q_{cr}$ is $C^{0,1/2}_{t,x} \cap C^{0,1}_v$ regular, for some $c=c(n,M)$. 
\end{proof}

We now turn to proving that under the present assumptions, the graph $\Gamma_u$ is differentiable with respect to the kinetic distance. The notion of differentiability with respect to the kinetic distance is a natural analogue to the notion of differentiability in Euclidean space.
\begin{defn}
    We say a function $g= g(t,x,v)$ is differentiable with respect to the kinetic distance at a point $z_0=(t_0,x_0,v_0)$ if there exists $a \in \R, b \in \R^n$ such that the first order polynomial $p(t,x,v) = a + b\cdot (v-v_0)$ satisfies
        $$\lim_{r \to 0} \frac{\|g-p\|_{L^\infty(Q_r(z_0))}}{r} =0.$$
\end{defn}
The polynomial $p$ in the above definition is first-order with respect to the kinetic scaling. See the introduction and also \cite{IMBERTMOUHOT} for a far more in-depth discussion of kinetic Taylor polynomials.

To prove our claim, one approach we could take is use Lemma \ref{lem: improvementmin} and the arguments in Proposition \ref{prop: vmonotoneglobal} to show that, in the setting of the previous proposition, there exists a radius $r_\delta = r(\sigma, M,n,\delta)>0$ and a constant $a=a(\sigma, M,n,\delta)>0$ such that the forward cone
$$\mathcal{A}_{r_\delta} \coloneqq  \{(t,x,v) \in Q_{r_\delta/2} \ : \ a(|t| + |x|)^{1/2} + |v| < \delta^{-1} v_n\}$$
is contained in $\Omega_u$ and, defined as in the proof of Theorem \ref{thm: freeboundaryregularity}, the backwards cone $-\mathcal{A}_{r_\delta}$ is contained in $\Lambda_u$. This would then imply that
$$f(t,x,v') \lesssim \delta(|t|^{1/2} + |x|^{1/2} + |v'|)$$
as $(t,x,v') \to (0,0,0)$, which is more than enough to conclude.
However, in the interest of proving the second claim in Proposition \ref{prop: graphdiff}, we follow a different route. The following is a standard lemma, which says that if $u$ is close to a half-space solution, then $\Gamma_u$ is trapped between two parallel planes which are constant in $(t,x)$.
\begin{lem}
Let $u \in \mathcal{P}_1(M)$ and suppose that
    $$\sup_{(t,x,v) \in Q_1} |u(t,x,v) - \frac12(v_n)_+^2| \le \eps.$$
Then $$\{v_n>\sqrt{2\eps}\}\cap Q_1 \subset \Omega_u = \{u>0\},$$
and 
$$\{v_n < -\tilde{c}_n\sqrt{\eps}\} \cap Q_{1/2} \subset \Lambda_u = \{u=0\}$$
for some dimensional constant $\tilde{c}_n>0$.
\end{lem}
\begin{proof}
    This is the same as in \cite{PSUBOOK}. The first statement is immediate and the latter statement is an immediate corollary of the non-degeneracy of $u$. Indeed, take any $z^0 = (t^0, x^0, v^0) \in Q_{1/2}$ with $v_n^0 <0$ and $u(z^0)>0$. Setting $r=-v_n^0>0$, by non-degeneracy we have
    $$c_nr^2 < c_n r^2 + u(z^0) \le \sup_{Q_r(z^0)} u \le \eps$$
    with the second inequality coming from \eqref{eq: nondegenpositive} and the final inequality coming from the fact that $(v_n)_+ =0$ on $Q_r(z^0)$. The assumption $u(z^0)$ forces $r=-v_n^0>0$ to not be too large. 
\end{proof}
We use the previous lemma to prove an improvement of flatness statement for the free boundary, which, when combined with Proposition \ref{prop: graph}, immediately implies the stated differentiability property and the remainder of Proposition \ref{prop: graphdiff}.
\begin{lem}
    Let $u \in \mathcal{P}_1(M)$ be such that $\delta_r(\rho) \ge \sigma(\rho)$ for some $\rho>0$, where $\sigma$ is as in Theorem \ref{thm: freeboundaryregularity}. Then there exists a modulus of continuity $\omega: (0,\infty) \to (0,\infty)$ such that $\omega(r) \to 0$ as $r\to 0$, and a direction $e \in \mathbb{S}^{n-1}$ such that 
    $$\Gamma_u \cap Q_r \subset \{(t,x,v) \in Q_r \ : \ |v\cdot e| <\omega(r)r\}.$$
\end{lem}
\begin{proof}
The thickness assumption allows us to assume without loss of generality that the unique blow-up of $u$ at the origin is $u_0(t,x,v)= \frac12(v \cdot e_n)_+^2.$ Suppose for the sake of contradiction that there exists $\delta>0$ and a sequence $r_j \to 0$ such that the statement 
$$\Gamma_u \cap Q_{r_j} \subset \{|v_n| < \delta r_j\}$$
fails. But $u_{r_j} \to u_0$ in $C^{1,\alpha}_\ell(Q_1)$, and according to the previous lemma, for any $\eps>0$ it holds for $j$ large that
$$\Gamma_{u_{r_j}} \subset Q_{1/2} \subset \{|v_n| < \tilde{c}_n \sqrt{\eps}\}.$$
Taking $\tilde{c}_n \sqrt{\eps}<\delta$ we arrive at a contradiction.
\end{proof}
We now prove our second result on free boundary regularity, Proposition \ref{prop: graphdiff}. 
\begin{proof}[Proof of Proposition \ref{prop: graphdiff}.]
    There is no loss of generality in assuming that near the origin, $\Gamma_u$ is the graph in the $v_n$ direction  of a function $f \in C^{0,1/2}_{t,x} \cap C^{0,1}_v$. With the previous lemma in hand, the differentiability of $f$ at the origin with respect to the kinetic distance is immediate. The real work is already done, since we know that $\Gamma_u\cap Q_{r_0} = \{(t,x,v',v_n) \in Q_{r_0/2} \ : \ v_n = f(t,x,v')\}$, and it follows that
$$\frac{f(t,x,v')}{\|(t,x,v')\|} \lesssim \omega(\|(t,x,v')\|)$$
with $\|\cdot\|$ as in section \ref{backgroundonkinetic}.

As for the corkscrew condition, it is proven in Lemma \ref{prop: porosity} that there exists $\kappa(n,M)>0$ such that for every $r>0$, and $z_0 \in \Gamma_u$, there exists $z_r \in Q_r(z_0)$ such that $Q_{\kappa r}(z_r) \subset \Omega_u \cap Q_r(z_0)$. And as a consequence of the previous lemma, we deduce that $\Lambda_u \cap Q_r(z_0)\supset Q_{\kappa r}(\tilde{z}_r)$ for some $\tilde{z}_r \in Q_r(z_0)$ for all $r>0$ small. We remark that the existence of cylinders of comparable radius on both sides of $\Gamma_u \cap Q_r(z_0)$ is essentially due to the nondegeneracy \eqref{eq: nondegen}. It is not implied by the $C^{0,1/2}_{t,x} \cap C^{0,1}_v$ local regularity of $\Gamma_u$.
\end{proof}
We remark that the regularity of $f$ is stronger than that which we expect to be required for a possible future Boundary Harnack principle for the operator $\mcL$. Additionally, the corkscrew condition obtained above is often utilized in the theory of Boundary Harnack principles for non-tangentially accessible domains. 

\bibliographystyle{abbrv}
\bibliography{cite}

\end{document}